\documentclass[a4paper]{article}
\usepackage{setspace}
\usepackage{amssymb,amsmath,amsthm}
\usepackage{color}
\usepackage{enumerate}
\usepackage[all]{xy}
\usepackage{theoremref}
\newtheorem{lemma}{Lemma}
\newtheorem{theorem}{Theorem}
\newtheorem{proposition}{Proposition}
\newtheorem{definition}{Definition}
\newtheorem{corollary}{Corollary}
\theoremstyle{remark}
\newtheorem*{ex}{Example}

\newtheorem*{rem}{Remark}
\newtheorem*{notation}{Notation}
\usepackage{mymacros}
\usepackage{stmaryrd}
\newcommand{\toco}{\xymatrix@C=1.5pc{  \ar[r]^-{\cong} & }}

\title{Deformations of generalized complex branes}
\author{Braxton L. Collier}
\begin{document}
\maketitle
\begin{abstract} We investigate the formal deformation theory of (rank 1) branes on generalized complex (GC) manifolds.  This generalizes, for example, the deformation theory of a complex submanifold in a fixed complex manifold.  For each GC brane $\B$ on a GC manifold $(X,\mathbb{J})$, we construct a formal (pointed) groupoid $\textbf{Def}^{\mathcal{B}}(X,\mathbb{J})$ (defined over a certain category of real Artin algebras) that encodes the formal deformations of $\mathcal{B}$. We study the geometric content of this groupoid in a number of different situations.  Using the theory of (bi)semicosimplicial differential graded Lie algebras (DGLAs), we construct for each brane $\mathcal{B}$ a DGLA $L_{\mathcal{B}}$ that governs the ``locally trivializable" deformations of $\mathcal{B}$.  As a concrete application of this construction, we prove an unobstructedness result.  \end{abstract}
\section{Introduction}
Let $X$ be a complex manifold, and $Z\subset X$ a complex submanifold.  A classical problem in complex geometry is to understand the possible deformations of $Z$ in $X$, i.e. the collection of complex submanifolds of $X$ that are ``close to" $Z$, in an appropriate sense.  In the most well-behaved situation, there exists a \emph{universal family} of such deformation; this is described, for example, in \cite{K} (where Kodaira calls it a ``maximal family"). Such a family consists of an auxiliary complex manifold $\mathcal{M}$ with basepoint $0\in\mathcal{M}$, and a complex submanifold $\hat{Z}\subset X\times \mathcal{M}$, such that, for each $m\in\mathcal{M}$, the fiber $\hat{Z}_m:=\hat{Z}\cap (X\times\{m\})$ is a complex submanifold of $X\times\{m\}\cong X$; the fiber over $0\in\mathcal{M}$ corresponds to $Z$ itself.  For every other family $(\mathcal{M}',\hat{Z}')$, there exists a neighborhood $U\subset \mathcal{M}'$ of the basepoint, and a unique (pointed) holomorphic map $\varphi:U\to \mathcal{M}$ such that, for each $m'\in U$ we have $\hat{Z}'_{m'}=\hat{Z}_{\varphi(m')}$ (regarded as complex submanifolds of $Z$).  We may regard $\mathcal{M}$ (or more precisely the universal family) as a (local) \emph{moduli space} for the deformations of $Z$. For a given $X$ and $Z$, such a moduli space may or may not exist.  

At a formal level, the possible deformations of $Z$ in $X$ may be encoded as a functor \[\textit{Def}_Z:\textrm{Art}_{\C}\to \Set,\] where $\textrm{Art}_{\C}$ is the category of local Artin algebras over $\C$ (with residue field $\C$) \cite{Kol}\cite{M}. Heuristically, given an Artin algebra $A\in\textrm{Art}_{\C}$, we may view an element of  $\textit{Def}_Z(A)$ as a family of complex submanifolds of $X$ deforming $Z$, which is parameterized by  $\textrm{Spec}(A)$ (with basepoint the unique geometric point of $\textrm{Spec}(A)$, i.e. the maximal ideal).  In the situation that there is a well-defined moduli space $\mathcal{M}$ as above, we may identify $\textit{Def}_Z(A)$ with the set of (pointed) maps  $\textrm{Spec}(A)\to \mathcal{M}$;  for example, the value of $\textit{Def}_Z$ on the so-called ``dual numbers" $A=\C[\epsilon]/(\epsilon^2)$ gives the (geometric) tangent space of $\mathcal{M}$ at $0$.   Appealing to the functor of points philosophy, we may view $\textit{Def}_Z$ itself as a formal stand-in for the moduli space that can always be defined.  Even when $\mathcal{M}$ itself does not exist, the ``tangent space" $\textit{Def}_Z(\C[\epsilon]/(\epsilon^2))$ still has an interesting geometric interpretation.  We summarize this as a theorem (although the result is a straightforward consequence of the definition of $\df_Z$).
\begin{theorem}\label{Z tangent}  There is a one-to-one correspondence between elements of  \\* $\textit{Def}_Z(\C[\epsilon]/(\epsilon^2))$ and holomorphic sections of the normal bundle $NZ$.  In other words, there is a natural bijection
\[\textit{Def}_Z(\C[\epsilon]/(\epsilon^2))\cong H^0(Z;\mathcal{O}_{NZ}).\] 
\end{theorem}  In fact, by a general argument (independent of Theorem \ref{Z tangent}), $\textit{Def}_Z(\C[\epsilon]/(\epsilon^2))$ inherits the structure of a $\C$-vector space; Theorem \ref{Z tangent} is then better formulated as an isomorphism of vector spaces. 

Another way in which $\textit{Def}_Z$ encodes the geometry of the putative moduli space involves the \emph{obstruction} properties of $\textit{Def}_Z$.  Given an Artin algebra $A\in\textrm{Art}_{\C}$, a deformation $\hat{Z}\in\textit{Def}_Z(A)$, and a surjective map $\mu: A'\to A$  in $\textrm{Art}_{\C}$, consider the problem of finding an extension  of $\hat{Z}$ to $\textit{Def}_Z(A')$, i.e. an element $\hat{Z}'\in\textit{Def}_Z(A')$ satisfying $\textit{Def}_Z(\mu)(\hat{Z}')=\hat{Z}$.  In general, there will be an obstruction to the existence of such an extension.  If a smooth moduli space $\mathcal{M}$ exists, however, then it is always possible to solve this extension problem and we say $\textit{Def}_Z$ is \emph{unobstructed}.  From a different perspective, we might allow $\mathcal{M}$ to be a more general (possible singular) type of space than a complex manifold; the obstruction problem for $\textit{Def}_Z$  then relates to the smoothness of $\mathcal{M}$ at $0$.
This explains the importance of the following well-known result \cite{Kol}\cite{M}.\begin{theorem}\label{Z obstructions} Let $Z\subset X$ be a complex submanifold satisfying the condition
\[H^1(Z;\mathcal{O}_{NZ})=0.\]  Then the deformation functor $\textit{Def}_{Z}$ is unobstructed.
\end{theorem} Compared to Theorem \ref{Z tangent}, which has a very clear geometric interpretation, Theorem \ref{Z obstructions} is perhaps harder to understand intuitively (and is correspondingly more difficult to prove).   

The goal of the present paper is to extend the above constructions and results as far as possible into the realm of \emph{generalized} complex geometry \cite{H}\cite{G}.  A generalized complex structure on a manifold $X$ is an endomorphism of the vector bundle $TX\oplus T\uv X$ (the direct sum of the tangent bundle and cotangent bundle), which squares to minus the identity, preserves the natural pairing of $TX\oplus T\uv X$ with itself, and satisfies an integrability condition defined with respect to the \emph{Dorfman bracket}
\[\ll(\xi,a),(\eta,b)\rr:=([\xi,\eta],\pounds(\xi)b-\iota(\eta)da).\]
As the name suggests, an ordinary complex structure $J:TX\to TX$ may be viewed as a particular example of a GC structure; a symplectic structure on $X$ also gives an example, however, so the subject could just as well be called ``generalized symplectic geometry".  Indeed, a fundamental appeal of GC geometry is that it provides a framework for treating complex and symplectic geometry in a unified way.  The subject also admits a wide variety of interesting examples: these range from examples that may be viewed as hybrids (or deformations) of complex and symplectic manifolds, to more exotic structures that exhibit features not present in either of the two ``classical" cases \cite{G}.

%The notion of a GC structure fits into a larger framework, sometimes referred to simply as ``generalized geometry".   For example, in addition to GC structures, there exist natural definitions of \emph{generalized Riemannian metrics}, \emph{generalized holomorphic vector bundles}, and so forth.  For the case we are interested in--complex submanifolds--there are in fact a few different ``generalized" versions that one might wish to study.  The particular structures we  investigate in this paper are known as (rank 1) \emph{generalized complex branes}.

To motivate the generalization of complex submanifolds whose deformation theory we study in this paper, as well as to explain some of the issues that arise in generalizing from the case of complex submanifolds, consider the following situation.  Let $(X,\omega)$ be a symplectic manifold, and consider the problem of constructing a well-behaved moduli space of \emph{Lagrangian} submanifolds of $X$. An issue that immediately arises--not present in the complex case--is the need to incorporate the so-called   \emph{Hamiltonian symmetries} of $(X,\omega)$.   Recall that, to any smooth function $f:X\to \R$ we may associate its \emph{Hamiltonian vector field}, which is characterized (up to a sign convention) by the equation \cite{C}
\[\iota(X_f)\omega=-df.\]   Every such vector field is an infinitesimal symmetry of the symplectic structure; in particular the flow $\{\varphi_t\}_{t\in\R}$ generated by $X_f$ (when it exists) satisfies
\[\varphi_t^*\omega=\omega\] for each $t\in \R$.  The symplectomorphisms of $(X,\omega)$ that arise in this way are known as \emph{Hamiltonian symmetries},\footnote{More generally, one can consider flows generated by time-dependent Hamiltonian vector fields corresponding to smooth functions $X\times [0,1]\to \R$.} and they form a subgroup of the group of all symplectomorphisms of $(X,\omega)$.  When forming a moduli space of Lagrangian submanifolds, it is custumary to identity two Lagrangians $L,L'\subset X$ if $L'=\varphi(L)$ for some Hamiltonian symmetry $\varphi$. This is necessary, for example, to have any hope of constructing a finite dimensional moduli space \cite{F1}. 

From the point of view of mirror symmetry (and string theory), it is actually more natural to consider a slightly different moduli space; namely, the collection of all pairs $(Z,\L)$, where $Z\subset X$ is a Lagrangian submanifold, and $\L$ is a Hermitian line bundle with a flat connection supported on $Z$ \cite{Kon}\cite{F1,F2}.   Adopting terminology from physics \cite{KO}, such an object may be called a (rank 1) \emph{Lagrangian brane}.  A nice property of this moduli space (when it can be defined), is that it carries a natural complex structure.  An added subtlety is the need to keep track of the possible equivalences between line bundles with connection.   %As we will see, the framework of GC geometry allows one to treat such equivalences on an equal footing with the Hamiltonian equivalences discussed above; this is one of the attractive features of the GC geometry approach.  

 On any GC manifold $(X,\J)$, a (rank 1) \emph{GC brane}  is similarly defined\footnote{In \cite{G3}, Gualtieri uses the term ``generalized complex brane" to refer to a different type of object compared to the ones studied in this paper; our terminology is adopted from \cite{KL}.} as a submanifold $Z\subset X$, together with a Hermitian line bundle with unitary connection $\L$ supported on $Z$; the pair $\B=(Z,\L)$ must satisfy a certain condition defined with respect to $\J$.    In the case of a complex manifold $X$, the compatibility condition implies that $Z$ must be a complex submanifold, and the curvature form $F\in\Omega^2(Z)$ of $\L$ must be of type $(1,1)$ with respect to the induced complex structure on $Z$ (i.e. the connection must induce a holomorphic structure on $\L$) \cite{G}\cite{KL}. In particular, any complex submanifold--equipped with the trivial line bundle--gives an example of a GC brane.  In the symplectic case, the Lagrangian branes described above give one type of example, but there are actually others as well.  These so-called ``coisotropic" A-branes were first introduced by Kapustin and Orlov in \cite{KO}, motivated by considerations coming from string theory and homological mirror symmetry; the fact that these more exotic objects arise naturally in the context of GC geometry is another appealing feature of this framework \cite{G}. 
 
In this paper we will study the formal deformation theory of GC branes. 
In analogy to the symplectic case described above, it will be important to incorporate the action of the \emph{generalized Hamiltonian} symmetries of $(X,\J)$ on GC branes.  On a GC manifold $(X,\J)$, one may associate to any \emph{complex}-valued function $f:X\to \C$ its \emph{generalized Hamiltonian vector field}, which is a section $\xx_f\in\Cinf(TX\oplus T\uv X)$.  As in the symplectic case, $\xx_f$ is an infinitesimal symmetry of $(X,\J)$, and may (at least locally) be integrated to a family of symmetries (a ``flow").  Such a symmetry consists of a pair $(\varphi,u)$, where $\varphi:X\toco X$ is a diffeomorphism and $u\in\Omega^1(X)$ is a 1-form.  Roughly speaking, the diffeomorphism component acts on a brane by pulling it back, and the 1-form component acts by changing the connection on the line bundle. 

%In the symplectic case, writing $f=f_R+if_I$ for $f_R,f_I$ real-valued functions, we have 
%\[\xx_f=(X_{f_R},df_I),\] where $X_{f_R}$ is the (usual) Hamiltonian vector field associated to $f_R$ as above.   
 % If the vector field component of a generalized Hamiltonian vector field is integrable, $\xx_f$ integrates to a symmetry of the GC manifold (just as a Hamiltonian vector field in symplectic geometry generates a symmetry of the symplectic manifold).  In particular, such symmetries act on the collection of GC branes.  Our definition of the deformation functor of a GC brane incorporates this notion of equivalence. 
 
Let describe the main constructions and results of the paper.  For each GC brane $\B$, we construct a functor\footnote{For simplicity, we define the functor on $\Art$ in this paper, but an extension to $\textrm{Art}_{\C}$ is also possible.} 
\[\textit{Def}_{\B}:\Art\to \textrm{Set}\] encoding the formal deformations of $\B$ (Definition \ref{brane deformation functor}).    This functor is defined as the truncation of a certain formal groupoid $\De\uB(X,\J)$ (Definition \ref{full deformation groupoid}), which encodes all the relevant notions of equivalence for deformations (including those induced by generalized Hamiltonian symmetries).   The construction of this functor is more intricate than in the case of complex submanifolds discussed above, and is done in several steps.

As a first step in studying the deformations of a GC brane $\B$,  we consider the value of the functor $\textit{Def}_{\B}$ on the dual numbers $\DN$; these are the first-order deformations of $\B$. It was argued in \cite{KM} that such first-order deformations should correspond to elements of a certain Lie algebroid cohomology group associated to $\B$; we recover their result in a rigorous framework. 

\begin{theorem}\label{brane tangent} For every GC brane $\B$, there is a natural bijection between elements of $\textit{Def}_{\B}(\R[\epsilon]/(\epsilon^2))$ and elements of the Lie algebroid cohomology group $H^1(\B)$.
\end{theorem}

 The calculation needed to prove this result is somewhat involved--and different from the one in \cite{KM}--so the fact that we reach the same conclusion may be regarded as a check on our Definition \ref{brane deformation functor}.  To explain the connection to Theorem \ref{Z tangent} stated above, suppose $Z\subset X$ is a complex submanifold, regarded as a GC brane $\B_Z$.  There is a natural isomorphism
 \[H^1(\B_Z)\cong H^0(Z;\mathcal{O}_{NZ})\oplus H^0(Z;\Omega^1(Z)),\] where $\Omega^1(Z)$ denotes the sheaf of holomorphic one-forms on $Z$.  The first component $H^0(Z;\mathcal{O}_{NZ})$ corresponds to the possible deformations of the submanifold $Z$, as discussed above.  The second component $H^0(Z;\Omega^1(Z))$ corresponds to deformations of the trivial holomorphic line bundle on $Z$. From a different perspective, if we denote by $\mathcal{E}_Z$ the coherent sheaf of $\mathcal{O}_X$-modules determined by $Z$ (the push-forward of its structure sheaf), then there are natural isomorphisms
 \[H^k(\B_Z)\cong \textrm{Ext}^k(\mathcal{E}_Z,\mathcal{E}_Z)\] for each $k$.  In particular, first-order deformations of $\B_Z$ as a GC brane correspond to first-order deformations of $\mathcal{E}_Z$ as a coherent sheaf (since these are parameterized by $\textrm{Ext}^1(\mathcal{E}_Z,\mathcal{E}_Z)$).
  
To go deeper into the geometric content of the deformation functor (and formal groupoid) associated to a GC brane, we next take a closer look at the relationship between symmetries of the GC manifold $(X,\J)$ and the deformations of $\B$.  It is here we encounter the first fundamental difference from the complex and Lagrangian cases described above. In the case of a complex submanifold $Z$ of a complex manifold $(X,J)$, for instance, every deformation of $Z$ is locally induced by a (local) symmetry of $(X,J)$.  To explain this, consider the case of 1st order deformations, which by Theorem \ref{Z tangent} correspond to holomorphic sections of the normal bundle $NZ$.  Given such a section $\xi$, for each $z\in Z$ it is always possible to extend $\xi$ to a holomorphic vector field $\tilde{\xi}$ on some neighborhood $U\subset X$ of $z$; heuristically, the corresponding deformation is then given near $z$ by ``flowing" $Z$ along $\tilde{\xi}$ for an infinitesimal time.  An analogous construction in the case of an arbitrary $GC$ brane is \emph{not} always possible, however.  We call deformations that have this property \emph{locally trivializable}.  Although an arbitrary GC brane may have deformations which are not locally trivializable, we introduce a large class of branes, which we call  \emph{leaf-wise Lagrangian} (LWL) branes, and prove that their deformations are always locally trivializable (Theorem \ref{essentially surjective}). This is a non-trivial statement about the geometry of such branes, and relies on the existence of a local normal form (which we prove in Theorem \ref{LWL normal form}).

The final part of the paper is devoted to proving a version of Theorem \ref{Z obstructions} for GC branes.  The result we prove applies only to those branes with locally trivializable deformations (in particular, by Proposition \ref{essentially surjective}, to holds for LWL branes).  We prove the theorem using the machinery of (bi)semicosimplicial differential graded Lie algebras (DGLAs), as developed in \cite{FMM}\cite{I}\cite{BM}.  Specifically, by adapting a construction given in \cite{I},  we construct, for each brane $\B$,  a DGLA $L_{\B}$ (depending on a choice).  We prove that $L_{\B}$ governs the locally trivializable deformations of $\B$.    A concrete calculation involving $L_{\B}$ then allows us to prove the following result.

\begin{theorem}\label{brane obstructions} Let $\B$ be a leaf-wise Lagrangian brane on a GC manifold (or more generally a brane with locally trivializable deformations).  If the Lie algebroid cohomology group $H^2(\B)$ vanishes, then the functor $\textit{Def}_{\B}$ is unobstructed.
\end{theorem} 

This result can clearly be sharpened in various situations: for example, in the case of a Lagrangian brane $\B$, it is not hard to see that the DGLA $L_{\B}$ is (homotopy) abelian, so that in this case $\df_{\B}$ is always unobstructed.  We leave the problem of strengthening Theorem \ref{brane obstructions}--as well as extending it to arbitrary branes--for future work. 

  \subsection{Organization of the paper}  In \S\ref{Courant algebroid} through \S\ref{LWL submanifolds}, we present the definitions and results in GC geometry that will be needed for the rest of the paper.  Although much of this material is well-known, there are a few definitions and results that have not--to the best of our knowledge--appeared before.  This includes, for example, the notion of a leaf-wise Lagrangian submanifold (or brane), which we introduce in Definition \ref{definition LWL submanifold}.  Our treatment of the symmetries of the standard Courant algebroid is also somewhat non-standard.  % Also new is Definition \ref{LWL submanifold} of a leaf-wise Lagrangian (LWL) submanifold, which appears in Theorem \ref{brane obstructions} stated above.  \S\ref{LWL submanifolds} is devoted to proving the existence of a local normal form these LWL submanifolds, stated in Theorem \ref{LWL normal form}. 

In \S\ref{section formal}, we recall the basic definitions and results about Artin algebras and nilpotent Lie algebras  we will need to formulate the definition of the functor $\df_{\B}$ associated to a GC brane (Definition \ref{brane deformation functor}).  In particular, we introduce a ``formal" (infinitesimal) version of the group  of symmetries of the Courant algebroid, and establish some of its basic properties.  In \S\ref{section deformations}, we construct the formal groupoid $\De\uB(X,\J)$ of deformations of a GC brane (Definition \ref{full deformation groupoid}); the deformation functor $\df_{\B}:\Art\to \Set$ of $\B$ is formed by taking $\pi_0$ of this groupoid (Definition \ref{brane deformation functor}).    

After defining the deformation functor (and formal groupoid), we proceed in  \S\ref{first order deformations} to prove Theorem \ref{brane tangent} about first-order deformations.  In \S\ref{section functoriality} we study the behavior of the deformation groupoid under equivalences between different branes, and in particular prove the invariance of the deformation functor under such equivalences.  In \S\ref{LWL branes example} we consider the deformations of leaf-wise Lagrangian branes; the main result of this section is Theorem \ref{essentially surjective}, which implies that every deformation of a LWL brane is locally induced by a symmetry of the ambient GC manifold. 
 In \S\ref{induced deformations}-\ref{cosimplicial groupoids} we investigate the relationship between deformations and symmetries in more detail.  
 
 The remainder of the paper is devoted to applying the theory of DGLAs to the deformation theory of GC branes.  In \S\ref{DGLA theory}, we recall some material from \cite{FMM}\cite{I}\cite{BM} concerning DGLAs, the construction of the Deligne groupoid, and semicosimplicial objects.  We then use this setup to construct a DGLA $L_{\B}$ governing the locally trivializable deformations of a GC brane $\B$.  In the final section (\S\ref{cohomology}), we use this construction to prove Theorem \ref{brane obstructions}.  Finally, there is a short appendix in which we prove a technical result stated in the main body of the paper.
 
\textbf{Acknowledgements:} I would like to thank Florian Sch\"{a}tz, Michael Bailey, Andrei Caldararu, Uli Bunke, and Dan Freed for useful discussions related to the subject matter of this paper.  I would also like to acknowledge the useful questions and feedback received in response to several talks I gave at the University of Regensburg.  Special thanks are also due to Simone Diverio for pointing me towards the reference \cite{I}; this was in response to a question I posed on the mathoverflow website, and I would also like to thank Domenico Fiorenza and Urs Schreiber for  useful answers.

\section{The standard Courant algebroid and its symmetries}\label{Courant algebroid}
Let $X$ be a smooth manifold, with tangent bundle $TX$ and cotangent bundle $T\uv X$.  We denote by $\TT X$ the direct sum $TX\oplus T\uv X$. We have the natural pairing 
\begin{align*} \la \cdot,\cdot \ra: \TT  X\oplus \TT X & \to X\times \R \\  (\xi,a), (\eta,b) & \mapsto \frac{1}{2}(\iota(\xi) b+\iota(\eta)a),\end{align*} as well as the Dorfman bracket 
\begin{align}\label{Dorfman Bracket} \ll\cdot,\cdot\rr: &\Cinf(\TTX)\times\Cinf(\TTX) \to \Cinf(\TTX) \\ & (\xi,a),(\eta,b) \mapsto ([\xi,\eta],\pounds(\xi)b-\iota(\eta)da),\notag\end{align} which is closely related to the Courant bracket \cite{G}.  The vector bundle $\TTX$ equipped with these structures, together with the projection $\pi:\TTX\to TX$ have the structure of an \emph{exact Courant algebroid} \cite{G,G2}.  In this paper we work with this ``standard" Courant algebroid, leaving for future work the extension of our results to a more general setting.    %We will see, however, that even in the case of the standard Courant algebroid it is useful conceptually to view it as being associated to the (trivial) gerbe.  

Let $\Diff(X)$ denote the group of diffeomorphisms of $X$, and $G(X)$ denote the opposite group $\Diff(X)^{op}$.  Given a diffeomorphism $\varphi:X \toco X$, we denote the corresponding element of $G(X)$ by $\varphi\us$, so that multiplication in $G(X)$ is given by $\varphi\us\psi\us=(\psi\varphi)\us$.  There is a left action of $G(X)$ on $\Omb(X)\oplus \Cinf(TX)$ by pullback, where by definition the pull-back of a vector field $\xi$ by a diffeomorphism $\varphi$ is the push-forward of $\xi$  by $\varphi^{-1}$:
\[\varphi\us\xi:=(\varphi^{-1})_*\xi.\]  This action is compatible with the differential graded commutative algebra structure on $\OmbX$, the Lie bracket on $\Cinf(TX)$, and the operation of contracting a vector field with a differential form.  By the formula (\ref{Dorfman Bracket}) for the Dorman bracket, it follows the action is also compatible with the Courant algebroid structure on $\TTX$ in the following sense: given any sections $\xx,\yy\in\Cinf(\TTX)$, and any $g=\varphi\us\in G(X)$, we have 
\begin{equation}\label{sym 1}\la g\xx,g\yy\ra=g\la\xx,\yy\ra,\end{equation} \begin{equation} \label{sym 2} \ll g\xx,g\yy\rr=g\ll\xx,\yy\rr,\end{equation} and 
\begin{equation}\label{sym 3} \pi(g\xx)=g\pi\xx.\end{equation}

A key feature of the Courant algebroid $(\TTX,\la\cdot,\cdot\ra,\ll\cdot,\cdot\rr,\pi)$ is that it admits a group of symmetries which is larger than $\Diff(X)$ \cite{G}. Namely, given a closed 2-form $B\in \Omega^2(X)$, the bundle map $e^B:\TTX\to \TTX$ given by 
\[(\xi,a)\mapsto (\xi,a-\iota(\xi)B)\] is compatible with the Dorfman bracket and preserves the pairing and the projection map. In particular, any 1-form $u\in\Omega^1(X)$ determines a symmetry of the Courant algebroid by setting $B=du$.  This motivates the following definition.  \begin{definition}\label{Definition cG} Let $\cG(X)$ denote the semi-direct product $G(X)\ltimes \Om^1(X)$, i.e. as a set $\cG(X)=G(X)\times \Om^1(X)$, and the group multiplication is given by the formula
\begin{equation}\label{cGmult}(\varphi\us,u)(\psi\us,w)=((\psi\varphi)\us,u+\varphi\us w).\end{equation} 
\end{definition}
\begin{proposition}\label{leftgroupaction}  There is a left action of $\cG(X)$ on $\Cinf(\TTX)$ given as follows: for each $g=(\varphi\us,u)\in\cG(X)$, and each $\xx\in \Cinf(\TTX)$, we define 
\begin{equation}\label{group action}g\cdot\xx=e^{du}\varphi\us\xx.\end{equation}  Moreover, for each $g\in \cG(X)$, the identities (\ref{sym 1}), (\ref{sym 2}), and (\ref{sym 3}) hold.
\end{proposition}
\begin{rem}\label{2 group}  Since the action of a group element $(\varphi\us,u)\in \cG(X)$ on $\Cinf(\TTX)$ depends only on the diffeomorphism $\varphi$ and the 2-form $du\in\Omega^2(X)$, one might ask why we did not define $\cG(X)$ to be the semi-direct product of $\textrm{Diff}(X)^{op}$ with the space  of closed 2-forms on $X$ (as is perhaps more standard, see for example \cite{G2}).  The reason is that this group would not act on the collection of branes on $X$.  As discussed in the introduction, such a brane consists of a pair $\B=(Z,\L)$, where $Z\subset X$ is a submanifold, and $\L$ is a Hermitian line bundle with unitary connection supported on $Z$.  As discussed below in Remark \ref{geometric brane 2}, a diffeomorphisms acts on such a brane $\B$ by pull-back, whereas an element of the form $(0,u)\in \cG(X)$ acts by altering the connection on $\L$ (by adding $-2\pi iu|_Z$).

From a more fundamental point of view, we should view $\cG(X)$ as a stand-in for a certain 2-group $\mathcal{G}$ (a group object in categories).  An object of $\mathcal{G}$ is a pair $(\varphi^*,\mathcal{E})$, where $\varphi$ is a diffeomorphism of $X$, and $\mathcal{E}$ is a Hermitian line bundle with unitary connection (or equivalently, a principal $U(1)$-bundle) over $X$.  The bundle component $\mathcal{E}$ acts on a brane $(Z,\L)$ by first restricting $\mathcal{E}\uv$ to $Z$ and then tensoring with $\L$. This 2-group $\G$ is the symmetry group of the trivial gerbe (with connective structure) on $X$; such symmetries naturally act on the category of branes on X.

Given an element $(\varphi^*,u)\in\cG(X)$, for example, we may view it as an object $(\varphi\us,\mathcal{E})$ of $\G$, where $\mathcal{E}$ is the trivial Hermitian line bundle on $X$ with connection $\nabla=d+2\pi i$.  Viewed in this way, a \emph{morphism} between two elements $(\varphi^*,u)$ and $(\varphi^*,u')$ of $\cG(X)$ (with the same diffeomorphism component) consists of a function $g:X\to U(1)$ satisfying 
\[g^{-1}dg=2\pi i(u'-u).\]  Roughly speaking, the elements of $\cG(X)$ correspond to the objects of the ``identity component" of $\mathcal{G}$; for the purposes of this paper these are all the symmetries we need. 

Because we treat only the standard Courant algebroid $TX\oplus T\uv X$ in this paper (corresponding to the trivial gerbe on $X$), the extra categorical structure of the 2-group will not be needed explicitly.  In order to extend to the case of more general Courant algebroids, however, it will almost certainly be necessary to replace $\cG(X)$ with the relevant 2-group.
\end{rem}  We now return to the main exposition, and prove Proposition \ref{leftgroupaction}.

\begin{proof}  Let $\xx=(\xi,a)$.  We have  
\begin{align*} (\varphi\us,u)\cdot((\psi\us,w)\cdot\xx) &= (\varphi\us,u)\cdot(\psi\us\xi,\psi\us a-\iota(\psi\us\xi)dw)  \\
&= (\varphi\us\psi\us\xi,\varphi\us\psi\us a-\varphi\us(\iota(\psi\us\xi)dw)-\iota(\varphi\us\psi\us\xi)du) \\
&= ((\psi\varphi)\us\xi,(\psi\varphi)\us a-\iota(\varphi\us\psi\us\xi)d(\varphi\us w)-\iota(\varphi\us\psi\us\xi)du)\\
&= ((\psi\varphi)\us\xi,(\psi\varphi)\us a-\iota((\psi\varphi)\us\xi))d(u+\varphi\us w)) \\
&= ((\psi\varphi)\us,u+\varphi\us w)\cdot \xx \\
&= ((\varphi\us,u)(\psi\us,w))\cdot\xx
\end{align*}

Since for each 1-form $u\in\Omega^1(X)$ its exterior derivative $du\in\Omega^2(X)$ is closed, it follows that equations (\ref{sym 1}), (\ref{sym 2}), and (\ref{sym 3}) hold.
\end{proof}
\begin{rem}\label{symmetry groupoid} More generally, we may define a groupoid $\cG$ whose objects are smooth manifolds, such that a morphism from a manifold $X$ to a manifold $Y$ is a pair $(\varphi^*,u)$, where $\varphi: Y\toco X$ is a diffeomorphism and $u\in\Omega^1(Y)$.  Composition of morphisms in this groupoid is given by the formula (\ref{cGmult}).  For a fixed manifold $X$, the group $\cG(X)$ is then recovered as the automorphism group of $X$.  Note also that a morphism from $X\to Y$ in $\cG$ determines an isomorphism $\Cinf(\TTX)\to \Cinf(\TT Y)$ using the same formula (\ref{group action}).
\end{rem}
\begin{rem} For each $g=(\varphi\us,u)\in\cG(X)$, the action (\ref{group action}) on $\Cinf(\TTX)$ is induced by a bundle map $\TTX\to\TTX$ covering the diffeomorphism $\varphi^{-1}:X\xymatrix@C=1.5pc{  \ar[r]^-{\cong} & } X$.  In particular, for every bundle endomorphism $F:\TTX\to \TTX$, the map $\Cinf(\TTX)\to \Cinf(\TTX)$ given by 
\[\xx\mapsto g\cdot F(g^{-1}\cdot \xx)\] is induced by a unique bundle endomorphism $\TTX\to\TTX$, which we denote by $g\cdot F$. Clearly, the map $F\mapsto g\cdot F$ defines an action of $\cG(X)$ on $\Cinf(\End(\TTX))$.  \end{rem}

We next introduce a Lie algebra that is the infinitesimal version of  $\cG(X)$, in the same way that the infinitesimal version of $G(X)=\textrm{Diff}^{op}(X)$ is the Lie algebra of vector fields on $X$. 
\begin{definition}\label{definition cg}  Let $\g(X)$ denote $\Cinf(TX)$, viewed as a real Lie algebra with respect to the Jacobi-Lie bracket.  We define $\cg(X)$ to be the semi-direct product $\g(X)\ltimes \Omega^1(X)$, where the action of $\g(X)$ on $\Omega^1(X)$ is given by the Lie derivative.  Explicitly, as a vector space we have \[\cg(X)=\g(X)\oplus \Omega^1(X)\cong\Cinf(\TTX),\] and the bracket is given by the formula
\begin{equation}\label{ghatbracket}[(\xi,a),(\eta,b)]=([\xi,\eta],\pounds(\xi)b-\pounds(\xi)a)\end{equation} for each $\xi,\eta\in\g(X)$ and $a,b\in\Omega^1(X)$.
\end{definition}  
\begin{rem} Returning briefly to the discussion in Remark \ref{2 group}, it is actually more natural to view elements of $\cg(X)$ as objects of a certain Lie 2-algebra, corresponding to the \emph{infinitesimal} symmetries of the trivial gerbe on $X$.  Such symmetries were introduced and studied in the author's Ph.D. thesis \cite{Co}.  As mentioned in Remark \ref{2 group}, the fact that we deal only with the standard Courant algebroid $TX\oplus T\uv X$ allows us to avoid  the use of the full Lie 2-algebra structure explicitly.  Even in this case, however, the reader may find the categorical perspective conceptually useful.
\end{rem}

Note that, although the underlying vector space of $\cg(X)$ is isomorphic to $\Cinf(\TTX)$,  the Lie bracket (\ref{ghatbracket})  is \emph{not} the same as the Dorfman bracket (\ref{Dorfman Bracket}); in particular, the latter is not even a Lie bracket \cite{G}.  On the other hand, the two brackets are compatible in a certain sense, as explained in the following proposition.

  \begin{proposition}\label{Courant action prop} \begin{enumerate}\item There is a left action of the Lie algebra $\cg(X)$ on $\Cinf(\TTX)$ given by the same formula as the Dorfman bracket (\ref{Dorfman Bracket}), i.e. for each $\xi=(\xi,a)\in\cg(X)$ and each $\alpha=(\tau,c)\in \Cinf(\TTX)$ we define
\[\xx\cdot\alpha = ([\xi,\tau],\pounds(\xi)c-\iota(\tau)da).\] \item Let $F:\TTX\to \TTX$ be a bundle map. For each $\xi\in \cg(X)$,  the map $\Cinf(\TTX)\to \Cinf(\TTX)$ given by 
\begin{equation}\label{end action}\alpha\mapsto \xx\cdot(F\alpha)-F(\xx\cdot\alpha)\end{equation} for all $\alpha\in\Cinf(\TTX)$ is induced by a unique bundle endomorphism $\TTX\to \TTX$, which we denote by $\xx\cdot F$.   The map $F\mapsto \xx\cdot F$ defines a Lie algebra action of $\cg(X)$ on $\Cinf(\End(\TTX))$.
\end{enumerate}
\end{proposition} 
\begin{proof}  To prove the first part, we must show that for every $\xx,\yy\in \cg(X)$, and every $\alpha\in\Cinf(\TTX)$, we have 
\begin{equation}\label{lieaction1}\xx\cdot(\yy\cdot \alpha)-\yy\cdot(\xx\cdot\alpha)=[\xx,\yy]\cdot \alpha.\end{equation} Writing $\xx=(\xi,a),\yy=(\eta,b)\in \cg(X)$, and $\alpha=(\tau,c)\in \Cinf(\TTX)$, we calculate
\begin{align*} (\xi,a)\cdot((\eta,b)\cdot(\tau,c)) & =(\xi,a)\cdot([\eta,\tau,\pounds(\eta)c-\iota(\tau)db)\\
& =([\xi,[\eta,\tau]],\pounds(\xi)\pounds(\eta)c-\pounds(\xi)\iota(\tau)db-\iota([\eta,\tau])da.\end{align*}
Therefore
\begin{align}\label{expression} & \xx\cdot (\yy\cdot\alpha)-\yy\cdot(\xx\cdot\alpha) \nonumber \\
& = ([\xi,[\eta,\tau]]-[\eta,[\xi,\tau]],\pounds(\xi)\pounds(\eta)c-\pounds(\eta)\pounds(\xi)c-\pounds(\xi)\iota(\tau)db+\pounds(\eta)\iota(\tau)da-\iota([\eta,\tau])da+\iota([\xi,\tau])db. 
\end{align}
 By the Jacobi identity for vector fields, we have 
\begin{equation}\label{calc 1}[\xi,[\eta,\tau]]-[\eta,[\xi,\tau]]=[[\xi,\eta],\tau].\end{equation} To calculate the one-form part, recall the following identities: given vector fields $\xi,\eta$, we have $[\pounds(\xi),\pounds(\eta)]=\pounds([\xi,\eta])$,  $\iota([\xi,\eta])=[\pounds(\xi),\iota(\eta)]$, and $\pounds(\xi)=d\iota(\xi)+\iota(\xi)d$.  The one-form component of (\ref{expression}) is therefore equal to
\begin{align}\label{calc 2} & [\pounds(\xi),\pounds(\eta)]c-\pounds(\xi)\iota(\tau)db+\pounds(\eta)\iota(\tau)da-\pounds(\eta)\iota(\tau)da+\iota(\tau)\pounds(\eta)da+\pounds(\xi)\iota(\tau)db-\iota(\tau)\pounds(\xi)db \nonumber\\
& =\pounds([\xi,\eta])c+\iota(\tau)(\pounds(\eta)da-\pounds(\xi)db) \nonumber\\
&= \pounds([\xi,\eta])c+\iota(\tau)(d\iota(\eta)da-d\iota(\xi)db) \nonumber\\
&= \pounds([\xi,\eta])c-\iota(\tau)d(\pounds(\xi)b-\pounds(\eta)a).
\end{align}  Combining (\ref{calc 1}) and (\ref{calc 2}) we see that (\ref{expression}) is equal to 
\[([\xi,\eta],\tau),( \pounds([\xi,\eta])c-\iota(\tau)d(\pounds(\xi)b-\pounds(\eta)a))=[(\xi,a),(\eta,b)]\cdot (\tau,c).\]   This includes the proof of the first part of the Proposition.

As for the second part, let $\xx=(\xi,a)\in\g(X)$, $T\in\Cinf(\End(\TTX)$, and $\alpha=(\eta,b)\in \Cinf(\TTX)$.  We have
\begin{align*} \xx\cdot(f\alpha) & =([\xi,f\eta],\pounds(\xi)(fb)-\iota(f\eta)da) \\
& = (f[\xi,\eta]+\xi(f)\eta,f\pounds(\xi)b+\xi(f) b-f\iota(\eta)da) \\
& = f\xx\cdot\alpha+\xi(f)\alpha.
\end{align*}  Since $T$ is linear over $\Cinf(X)$, we see that 
\begin{align*} (\xx\cdot T)(f\alpha) & =f\xx\cdot T(\alpha)+\xx(f)T\alpha-fT\xx\cdot\alpha-\xx(f)T\alpha\\
& = f(\xx\cdot T)(\alpha)
\end{align*}  This shows that the map $\Cinf(\TTX)\to \Cinf(\TTX)$ defined by the right-hand side of equation (\ref{end action}) is linear over functions, and therefore is induced by a well-defined section of $\Cinf(\End(\TTX))$.  It is also easy to see that equation (\ref{lieaction1}) implies that 
\[\xx\cdot(\yy\cdot T)-\yy\cdot(\xx\cdot T)=[\xx,\yy]\cdot T\] holds for every $\xx,\yy\in\cg(X)$.
\end{proof} 
 
 Recall that a vector field $\xi\in\g(X)=\Cinf(TX)$ is called \emph{complete} if it generates a flow  
\begin{align*} \Phi: & X\times \R\to X \\ & (x,t) \mapsto \varphi_t(x).\end{align*}  Given such a vector field $\xi\in \g(X)$, we define
\[e^{t\xi}=\varphi\us_t\in G(X).\]  The (left) action of $\g(X)$ on $\Omb(X)\oplus \Cinf(TX)$ by Lie derivative is the infinitesimal version of the (left) action of $G(X)=\textrm{Diff}^{op}(X)$ on $\Omb(X)\oplus \Cinf(TX)$ in the sense that for every complete $\xi\in\g(X)$, and every $\alpha\in\Omb(X)\oplus \Cinf(TX)$, we have 
\[\frac{d}{ds}|_{s=t}e^{s\xi}\alpha=\pounds(\xi)e^{t\xi}\alpha,\] where the derivative with respect to $s$ is defined point-wise.

Similarly, we say that an element $\xx=(\xi,a)\in\cg(X)$ is \emph{complete} if the vector field $\xi$ is.  In this case, define 
\[a^{t\xi}=\int_0^te^{s\xi}ads,\] and 
\[e^{t\xx}=(e^{t\xi},a^{t\xi})\in \cG(X).\]  We then have the following result, similar to \cite[Prop 2.3]{G2}.
\begin{proposition}\label{exp Dorfman}  Let $\xx\in\cg(X)$ be complete.  Then \begin{enumerate} \item  for every $t,t'\in\R$, we have 
\[e^{t\xx}e^{t'\xx}=e^{(t+t')\xx}.\]
\item For every every $\alpha\in\Cinf(\TT X)$ and every $t\in \R$, we have 
\[\frac{d}{ds}|_{s=t}e^{s\xx}\cdot \alpha=\xx\cdot(e^{t\xi}\cdot \alpha),\] where the derivative with respect to $s$ is defined point-wise.
\end{enumerate}
\end{proposition}
\begin{proof}
Given $\xx=(\xi,a)$, we have 
\begin{equation}\label{multiplication}e^{t\xx}e^{t'\xx}=(\varphi_t^*\varphi_{t'}^*,a^{t\xi}+\varphi_t^*a^{t'\xi}).\end{equation}  We have 
\[\varphi_t^*a^{t'\xi}=\int_0^{t'}\varphi^*_{t+s}ads=\int_{t}^{t+t'}\varphi^*_{s'}ads',\] where the second equality follows from the change of variables $s'=s+t$.  Therefore 
\begin{align*}a^{t\xi}+\varphi_t^*a^{t'\xi} & =\int_0^t\varphi_s^*ads+\int_{t}^{t+t'}\varphi_s^*ads \\ &= \int_0^{t+t'}\varphi_s^*ads=a^{(t+t')\xi}.\end{align*}  Substituting into equation (\ref{multiplication}) we obtian the desired result
\begin{equation}\label{exponential rule}e^{t\xx}e^{t'\xx}=e^{(t+t')\xx}.\end{equation}
Given $\yy=(\eta,b)\in\Cinf(\TT X)$, consider
\begin{equation}\label{t derivative} \frac{d}{dt}|_{t=0}(e^{t\xx}\cdot \yy)=(\frac{d}{dt}|_{t=0}(\varphi_t^*\eta), \frac{d}{dt}|_{t=0}(\varphi_t^*b)-\frac{d}{dt}|_{t=0}(\iota(\varphi_t^*\eta)d\int_0^{t}\varphi_s^*ads).\end{equation}  We have 
\begin{equation}\label{t derivative 1}\frac{d}{dt}|_{t=0}(\varphi_t^*\eta)=[\xi,\eta]\end{equation} and 
\begin{equation}\label{t derivative 2}\frac{d}{dt}|_{t=0}(\varphi_t^*b)=\pounds(\xi)b.\end{equation}  Since $a^{t\xi}$ vanishes at $t=0$, the Leibnitz rule together with the fundamental theorem of calculus implies that  implies that 
\begin{equation}\label{t derivative 3}\frac{d}{dt}|_{t=0}(\iota(\varphi_t^*\eta)d\int_0^{t}\varphi_s^*ads=\iota(\eta)da.\end{equation} Substituting equations (\ref{t derivative 1}), (\ref{t derivative 2}), and (\ref{t derivative 3}) into (\ref{t derivative}), we see that 
\begin{equation}\label{t derivative 4}  \frac{d}{dt}|_{t=0}(e^{t\xx}\cdot \yy)=\ll \xx,\yy\rr.\end{equation}  

Finally, combining (\ref{exponential rule}) with (\ref{t derivative 4}), easily verify the second part of the proposition.

\end{proof}

\section{Generalized complex structures}\label{GC structures}
\subsection{Basic definitions and examples}

A good reference for the material in this subsection is \cite{G}. As described in the introduction, a \emph{generalized complex structure} on a manifold $X$ is an endomorphism 
\[\J:\TTX\to\TTX\] that preserves the pairing $\la\cdot,\cdot\ra$, satisfies $\J^2=-id_{\TTX}$, and satisfies a certain integrability condition.  To describe this condition,  note that, since $\J$ squares to minus the identity, we may decompose 
\[\TTX\otimes\C=L\oplus\bar{L},\] where $L$ and $\bar{L}$ are the $+i$ and $-i$ eigen-bundles of (the complex-linear extension of) $\J$, respectively.  We require that $\Cinf(L)$ be involutive with respect to the Dorfman bracket, i.e. we require that
\[\ll\Cinf(L),\Cinf(L)\rr\subset\Cinf(L).\]    

Given a GC structure $\J$ on $X$, the restriction of the Dorfman bracket to $\Cinf(L)$ endows $L$ with the structure of a complex Lie algebroid over $X$, with anchor map given by  projection $\pi:L\to TX\otimes\C$.  In particular, we may associate to $\J$ its \emph{generalized Dolbeault} complex
\[\xymatrix{ \Cinf(\Lambda^0L\uv) \ar[r]^{\delta_L} & \Cinf(\Lambda^1L\uv) \ar[r]^{\delta_L} & \Cinf(\Lambda^2L\uv) \ar[r]^-{\delta_L} & \cdots}.\] %We will denote this complex by $(C^{\bullet}(L),\delta_L)$ and its cohomology groups by $H^{\bullet}(L)$.  
For future reference, let us explicitly describe the first two differentials in this complex.  Given $f:X\to \C$ (viewed as a section of $\Lambda^0L\uv$), the section $\delta_Lf\in\Cinf(\Lambda L\uv)$ is given by 
\[\delta_Lf(\xx)=\pi(\xx)\cdot f\] for every $\xx\in\Cinf(L)$, where we recall $\pi:L\to TX\otimes\C$ is the projection (anchor) map.  Given $\alpha\in\Cinf(\Lambda L\uv)$, for every $\xx,\yy\in\Cinf(L)$ we have 
\[\delta_L\alpha(\xx,\yy)=\pi(\xx)\cdot \alpha(\yy)-\pi(\yy)\cdot\alpha(\xx)-\alpha([\xx,\yy]).\]

\begin{rem}A useful observation is that, since both $L$ and $\bar{L}$ are maximally isotropic sub-bundles of $\TTX\otimes \C$ (with respect to the $\C$-linear extension of the pairing), the pairing determines an isomorphism 
\[\bar{L}\xymatrix@C=1.5pc{  \ar[r]^-{\cong} & } L^{\vee}.\]  
\end{rem}
\begin{ex}\label{complex} Any ordinary complex structure $J$ on $X$ determines a GC structure given by  
\begin{equation}\label{Jcomplex}\J_{J}:=\left(\begin{array}{cc} -J & 0 \\ 0 & J^{\vee} \end{array}\right).\end{equation}  In this case we have 
\[L=(TX)^{0,1}\oplus(T\uv X)^{1,0}\] and 
\[L\uv\cong\bar{L}=(TX)^{1,0}\oplus(T\uv X)^{0,1}.\]  The generalized Dolbeault complex  is isomorphic to $\Omega^{0,\bullet}(X;\Lambda^{\bullet}(TX)^{1,0})$ with differential the $\bar{\partial}$-operator corresponding to the standard holomorphic structure on $\Lambda^{\bullet}(TX)^{1,0}$ \cite{G}.
\end{ex}
\begin{ex}\label{symplectic}  Let $\omega$ be a symplectic structure on $X$, viewed as an isomorphism $TX\to T\uv X$.   This determines a GC structure given by 
\begin{equation}\label{Jsymplectic} \J_{\omega}:=\left(\begin{array}{cc} 0 & -\omega^{-1} \\ \omega & 0 \end{array}\right).\end{equation}
In this case, we have 
\[L=e^{i\omega}(TX\otimes\C)=\{(\xi,-i\omega\xi):\xi\in TX\}.\]  The bundle map 
\[e^{i\omega}:TX\otimes\C\to L\] is an isomorphism of complex Lie algebroids, and in particular the generalized Dolbeualt complex is isomorphic to $(\Omega^{\bullet}(X;\C),d_{dR})$, the ordinary de-Rham complex of $X$ with complex coefficients.  
\end{ex}
\begin{ex} Given a GC manifolds $(X,\J)$ and $(X',\J')$, the product $X\times X'$ inherits a natural GC structure $\J\times\J'$.
\end{ex}
Given a GC manifold $(X,\J)$ with $+i$ eigen-bundle $L\subset \TTX\otimes \C$, the \emph{type} of $\J$ at a point $x\in X$ is the complex dimension of $L\cap (T\uv X\otimes \C)$; equivalently, it is the complex codimension of the projection $\pi_{TX\otimes\C}(L)\subset TX\otimes \C$.  The GC manifold is said to be \emph{regular} at a point $x\in X$ if the type of $\J$ is constant on some neighborhood of $x$.  

An equivalent definition of the type can be given as follows. Decompose $\J$ as 
\begin{equation}\J=\left(\begin{array}{cc} \J_{11} & \J_{12} \\ \J_{21} & \J_{22}\end{array}\right),\end{equation}
 where $\J_{11}:TX\to  TX$, $\J_{21}:T X\to T\uv X$, $\J_{12}:T\uv X\to TX$, and $\J_{22}:T\uv X\to T\uv X$.  The fact that $\J$ preserves the natural pairing and squares to minus the identity implies that $P:=\J_{12}:T\uv X\to TX$ is skew-symmetric, i.e. for every $a,b\in T\uv X$ we have 
\[a(P(b))=-b(P(a)).\]  In fact, the integrability of $\J$ implies that $P$ is a Poisson structure on $X$.  Defining $R=\Ker(P)\subset T\uv$, we easily check that $\J_{22}:T\uv\to T\uv$ preserves $R$ and in fact restricts to a complex structure on $R$.  It is then easy to check that the type of $\J$ is equal $\textrm{Dim}_{\C}(R)=\frac{1}{2}\textrm{Dim}_{\R}(R)$.  In particular,  $\J$ is regular at a point $x\in X$ if and only if the Poisson structure $P$ is regular at $x$, i.e. if and only if $\textrm{Dim}_{\R}(P(T\uv X))$ is constant on a neighborhood of $x\in X$.  

As an example, given a symplectic manifold $(X,\omega)$ and a complex manifold $(Y,J)$ of complex dimension $k$,  the product $(X\times Y,\J_{\omega}\times\J_J)$ is of type $k$ at every point; in particular, it is everywhere regular.  Conversely, the \emph{generalized Darboux theorem}, due to Gualtieri \cite{G}, says if $x$ is a regular point of an arbitrary GC manifold $(X,\J)$, then there exists some neighborhood of $x$ on which $\J$ is equivalent to such a product. Let us give a precise statement of this result in a form convenient for our purposes.  Endow $\R^{2m}\cong (\R^m)\uv$ with the standard symplectic structure
\[\omega=dx^{m+1}\wedge dx^1+\cdots +dx^{2m}\wedge dx^{m},\] and $\R^{2n}\cong \C^n$ with the standard complex structure.  Let $X_0^{m,n}=(\R^{2m+2n},\J_0^{m,n}=\J_{\omega}\times \J_J)$ be the product GC manifold.   We may then state (a slight variation) of the generalized Darboux theorem \cite[Th. 4.35]{G}.\begin{theorem}\label{generalized Darboux} Let $(X,\J)$ be a GC manifold of dimension $2(m+n)$, and let $x\in X$ be a regular point such that $\J$ is of type $n$ at $x$.  Then there exists a neighborhood $U$ of $x$, a neighborhood $U_0$ of the origin in $X^{m,n}_0$, a diffeomorphism $\Phi:U\toco U_0$ taking $x$ to the origin, and a 1-form $u\in \Omega^1(U)$, such that 
\[\Phi^*(\J^{m,n}_0|_{U_0})=e^u\cdot(\J|_{U}).\]
\end{theorem}
 \subsection{Generalized Holomorphic Vector Fields}

Let $\J$ be a GC structure on $X$, and consider the projection maps
\[ \pi_L:\TTX\otimes\C\to L\] given by
\[\xx\mapsto \xx^{(1,0)}:= \frac{1}{2}(\xx-i\J\xx).\] and
 \[ \pi_{\bar{L}}:\TTX\otimes\C\to \bar{L}\] given by
\[\xx\mapsto \xx^{(0,1)}:= \frac{1}{2}(\xx+i\J\xx).\] 
  
Also define $\mu:\TTX\to L\uv$ by  
\begin{equation}\label{mudef} \mu(\xx)= 2\la\xx,\cdot\ra|_{L}.\end{equation}  We easily check that 
\[\J^{\vee}\mu(\xx)=\mu(-\J\xx),\] so that if we regard $\TTX$ as a complex vector bundle with complex structure $-\J$, $\mu$ is a map of complex vector bundles. We also note that, since $\xx=\xx^{(1,0)}+\xx^{(0,1)}$ and $L$ is isotropic, we have
\[\mu(\xx)=2\la \xx^{(0,1)},\cdot\ra|_L.\]  Since $\pi_{\bar{L}}:(\TTX,-\J)\xymatrix@C=1.5pc{  \ar[r]^-{\cong} & } \bar{L}$ is an isomorphism of complex vector bundles, and $\bar{L}$ pairs non-degenerately with $L$, it follows that $\mu$ is an isomorphism.

\begin{definition}\label{generalized holomorphic vector field} A section $\xx\in \cg(X)\cong C^{\infty}(\TTX)$ is a \emph{generalized holomorphic vector field} if it satisfies
\begin{equation}\label{ghol}\delta_{L}\mu(\xx)=0.\end{equation}  We denote the space of all generalized holomorphic vector fields by $\T(X)\subset \cg(X)$.
\begin{notation} The space of generalized holomorphic vector fields of course depends on the GC structure $\J$.  When necessary, will use the more precise notation $\T(X,\J)$ to indicate which GC structure is appearing in the condition (\ref{ghol}).
\end{notation} 
\end{definition} The next proposition shows that we may view the space of generalized holomorphic vector fields as the infinitesimal symmetries of the GC structure $\J$.
\begin{proposition}\label{holomorphic symmetries} For each $\xx\in\cg(X)$, we have 
\[\xx\cdot\J=0\] if and only if 
\[\delta_{L}\mu(\xx)=0.\] In particular, if $\xx$ is complete then 
\[e^{\xx}\J=\J.\]
 \end{proposition}  
 \begin{rem} It is easy to see from this Proposition that the subspace $\T(X)\subset \cg(X)$ is closed under the Lie bracket (\ref{ghatbracket}).
 \end{rem}
\begin{proof} The following lemma will be useful for the proof of the proposition. \begin{lemma} For each $\xx\in\cg(X)$, we have $\xx\cdot\J=0$ if and only if 
\[\xx\cdot\Cinf(L)\subset\Cinf(L).\] 
\end{lemma}
\begin{proof} Suppose $\xx\cdot\J=0$.  Then for every $v\in\Cinf(L)$ we have 
\[0=\xx\cdot(\J v)-\J(\xx\cdot v)=i\xx\cdot v-\J(\xx\cdot v)\] so that 
\[\J(\xx\cdot v)=i\xx\cdot v,\] so we see that $\xx\cdot v\in\Cinf(L)$.  Conversely, suppose that $\xx\cdot\Cinf(L)\subset\Cinf(L)$.  Since $\xx$ is real it follows that $\xx\cdot\Cinf(\bar{L})\subset\Cinf(\bar{L})$ also.  For arbitrary $v\in\Cinf(\TTX\otimes\C)$, we see that 
\[\xx\cdot\pi_L(v)=\pi_L(\xx\cdot\pi_L(v))=\pi_L(\xx\cdot v)\] and 
\[\xx\cdot\pi_{\bar{L}}(v)=\pi_{\bar{L}}(\xx\cdot\pi_{\bar{L}}(v))=\pi_{\bar{L}}(\xx\cdot v).\] Therefore 
\begin{align*} (\xx\cdot\J)(v) &= \xx\cdot(\J v)-\J(\xx\cdot v) \\
& =i\xx\cdot\pi_L(v)-i\xx\cdot\pi_{\bar{L}}(v)-i\pi_L(\xx\cdot v)+i\pi_{\bar{L}}(\xx\cdot v)\\
&= 0.
\end{align*}  
\end{proof} 
By the lemma, to prove the Proposition, we must show that $\delta_L\mu(\xx)=0$ if and only if $\xx\cdot\Cinf(L)\subset\Cinf(L)$.  Since $L$ is a maximal isotropic subspace of $\TTX\otimes \C$, we have $\xx\cdot\Cinf(L)\subset\Cinf(L)$ if and only if for every pair of sections $v,w\in\Cinf(L)$ we have 
\begin{equation}\label{preservation condition}\la\xx\cdot v,w\ra=0.\end{equation}  We claim that the left-hand side of equation (\ref{preservation condition}) is equal to $\frac{1}{2}\delta_L(\mu(\xx))(w,v)$, so that it vanishes for arbitrary $v$ and $w$ if and only if $\xx\in\T(X)$.  To see this, recall the following two identities satisfied by the Dorfman bracket (see e.g. \cite[\S 3.2]{G})
\begin{enumerate} \item For every $A,B\in \Cinf(\TTX)$ 
\begin{equation}\label{Dorfman1}\ll A,B\rr=-\ll B,A\rr +2(0,d\la A,B\ra).\end{equation}
\item For every $A,B,C\in\Cinf(\TTX)$ 
\begin{equation}\label{Dorfman2}\pi(A)\la B,C\ra=\la\ll A,B\rr,C\ra+\la B,\ll A,C\rr\ra.\end{equation}
\end{enumerate}  Using these we have
\begin{align*} \la \xx\cdot v,w\ra &= \la \ll\xx,v\rr,w\ra \\
&= -\la\ll v,\xx\rr,w\ra+2\la d\la \xx,v\ra,w\ra \\
&= -(\pi(v)\la\xx,w\ra-\la\xx,[v,w]\ra)+\pi(w)\la\xx,v\ra \\
&= \pi(w)\la\xx,v\ra-\pi(v)\la\xx,w\ra-\la\xx,[w,v]\ra \\
&=\frac{1}{2}\delta_L(\mu(\xx))(w,v).
\end{align*} 

\end{proof}

\begin{ex}  Given a complex structure $J$ on $X$, let $\J_J$ be the induced GC structure described in Example \ref{complex}.  It is straightforward to check that
a vector field $\xi\in \Cinf(TX)$ and a 1-form $u\in \Omega^1(X)$ determine a generalized holomorphic vector field $\xx=(\xi,u)$ on $(X,\J_J)$ if and only if 
\begin{equation}\label{holomorphic 1} \bar{\partial}\xi^{1,0}=0\end{equation} and 
\begin{equation}\label{holomorphic 2} \bar{\partial}u^{0,1}=0.\end{equation}

Note that condition (\ref{holomorphic 1}) is equivalent to requiring 
\[\pounds(\xi)J=0,\] whereas condition (\ref{holomorphic 2}) is equivalent to requiring that $du$ be of complex type $(1,1)$.  

\end{ex}
\begin{ex} Given a symplectic structure $\omega$ on $X$, let $\J_{\omega}$ be the induced GC structure described in Example \ref{symplectic}. In this case,
a vector field $\xi\in \Cinf(TX)$ and 1-form $u\in \Omega^1(X)$ determine a generalized holomorphic vector field $\xx=(\xi,u)$ on $(X,\J_{\omega})$ if and only if 
\[\pounds(\xi)\omega=0\] and \[du=0.\]  
Thus, an infinitesimal symmetry of $\J_{\omega}$ consists of an infinitesimal symplectomorphism $\xi$ together with a closed 1-form $u$.

\end{ex}

\subsection{Generalized Hamiltonian Vector Fields}

Given a symplectic manifold $(X,\omega)$ and a real-valued function $f:X\to \R$, recall that the \emph{Hamiltonian vector field} associated to $f$ is defined (up to a sign convention) by \begin{equation}\label{ordinary hamiltonian} X_f=-\omega^{-1}(df).\end{equation}  This vector field is an infinitesimal symmetry of the symplectic structure in the sense that 
\[\pounds_{X_f}\omega=0,\] and the collection of all Hamiltonian vector fields form a sub-algebra of $\Cinf(TX)$.  We next introduce a generalization of this construction.  Namely, given a GC manifold $(X,\J)$, for every \emph{complex}-valued function $f:X\to \C$ we will define infinitesimal symmetry $\xx_f$ of $\J$.  %In the case that $\J$ corresponds to a symplectic structure and $f$ is real-valued, we have $\xx_f=X_f$ is the ordinary Hamiltonian vector field associated to $f$. Even this case, however, we will see that extra symmetries we obtain by allowing complex-valued $f$ play an important role.
\begin{definition}\label{genhamdef}  Let $(X,\J)$ be a GC manifold. For every smooth function $f:X\to \C$, the \emph{generalized Hamiltonian vector field} $\xx_f\in\Cinf(\TTX)$ associated to $f$ is given by  
\begin{equation}\label{genham}\xx_f=-\textrm{Re}(2i(0,df)^{(0,1)})=\textrm{Re}(\J(0,df)-(0,idf)).\end{equation}
\end{definition}

The author learned of this construction (for real-valued $f$) in \cite[Prop 6]{H2}, where it is shown that $\xx_f$ is an infinitesimal symmetry of $(X,\J)$.  We give a different proof of this result (for complex-valued $f$) as part (1) of the following proposition.  To the best of our knowledge,  part (2) of Proposition \ref{hamiltonian} has not appeared before.

\begin{proposition}\label{hamiltonian} \begin{enumerate} \item For every $f:X\to \C$, the generalized Hamiltonian vector field $\xx_f$ is an element of $\T(X)$, i.e. satisfies 
\[\delta_{L}\mu(\xx_f)=0.\]
\item The collection $\H$ of generalized Hamiltonian vector fields is a Lie sub-algebra of $\T(X)$.
\end{enumerate}
\end{proposition}

\begin{proof} Given $f:X\to \C$, one easily calculates using Definition \ref{genhamdef} and and (\ref{mudef}) that 
\[\mu(\xx_f)=-i\delta_{L}f,\] where on the right-hand side $f$ is regarded as a section of $\Lambda^0 L\uv$.  The first part of the proposition therefore follows from the identity $\delta_{\J}^2=0$.
To prove the second part, we must show that given any $f,g:X\to \C$, there exists $h:X\to\C$ such that $[\xx_f,\xx_g]=\xx_h$.  First, suppose that $f$ and $g$ are purely imaginary.  Writing $f=if_I$ and $g=ig_I$, we have 
\[[\xx_f,\xx_g]=[(0,df_I),(0,dg_I)]=0,\] so the result is trivial in this case.  Next, suppose that $f$ is purely real and $g=ig_I$ is purely imaginary.  Decomposing the GC structure as
\begin{equation}\label{J decom}\J=\left(\begin{array}{cc} J & P \\ \sigma & K\end{array}\right),\end{equation} we have 
\[[\xx_f,\xx_g]=[(P df,Kdf_r),(0,dg_I)]=(0,\pounds(P df)dg_I).\]  Using the Cartan formula for the Lie derivative, we see that this is equal to 
\[(0,d\iota(P df)dg_I)=\xx_{i\iota(P df)dg_I}.\]  The last case, where both $f$ and $g$ are purely real, is more difficult, and uses the integrability of $\J$ in a crucial way.  It will be useful to rewrite the integrability condition for $\J$.  To do so, first recall the definition of the \emph{Courant Bracket} on sections of $\TT X$, which is given by a formula closely related to that for the Dorfman bracket \cite{G}:
\[[(\xi,a),(\eta,b)]_C=([\xi,\eta],\pounds(\xi)b-\pounds(\eta)a-\frac{1}{2}d(\iota(\xi)b-\iota(\eta)a)).\]  Given an almost generalized complex structure $\J$, we define the \emph{Nijenhuis tensor} $\Nij:\TT X\otimes \TT X\to \TT X$ of $\J$ by the formula 
\[\Nij(A,B)=[\J A,\J B]_C-\J[\J A,B]_C-\J[A,\J B]_C-[A,B]_C\] for every pair of sections $A,B\in\Cinf(\TT X)$.  Integrability of $\J$ is equivalent to the vanishing of $\Nij$ \cite{AB}.

Returning to the proof of Proposition \ref{hamiltonian}, the integrability of $\J$ implies that, for every pair of real-valued functions $f$ and $g$ we have 
\begin{align}\label{nij0}0 & =\Nij((0,df),(0,dg)) \notag\\
& =[\J(0,df),\J(0,dg)]_C-\J[\J(0,df),(0,dg)]_C-\J[(0,df),\J(0,dg)]_C-[(0,df),(0,dg)]_C.\end{align}  In terms of the decomposition (\ref{J decom}) of $\J$, we have $[\J(0,df),\J(0,dg)]_C$=
\begin{equation}\label{eqn11} ([P df, P dg],\pounds(P df)Kdg-\pounds(P dg)Kdf-\frac{1}{2}d\iota(P df)dg+\frac{1}{2}d\iota(P dg)df). \end{equation} It is easy to check that $P:T\uv X\to TX$ must be skew-symmetric, in the sense that for every $a,b\in \Omega^1(X)$ we have $a(P(b))=-b(P(a)$.  Therefore the expression (\ref{eqn11}) is equal to  
\begin{equation}\label{nij1}  ([P df, P dg],\pounds(P df)Kdg-\pounds(P dg)Kdf-d\iota(P df)dg).\end{equation}
%\begin{equation}\label{nij2} \J(
Similarly, we calculate 
\begin{align}\label{nij2} \J[\J(0,df),(0,dg)]_C & = \J(0,\pounds(P df)dg-\frac{1}{2}d\iota(P df)dg) \notag\\
& = (P\pounds(P df)dg-\frac{1}{2}P d\iota(P df)dg, K\pounds(P df)dg-\frac{1}{2}K d\iota(P df)dg) \notag \\
& = (P d\iota(P df)dg-\frac{1}{2}P d\iota(P df)dg,K d\iota(P df)dg-\frac{1}{2}Kd\iota(P df)dg) \notag \\
& =(\frac{1}{2}P d\iota(P df)dg-,\frac{1}{2}K d\iota(P df)dg) ,\end{align} and 
\begin{align}\label{nij3} \J[(0,df),\J(0,dg)]_C & =(-P \pounds(P dg)df+\frac{1}{2}P d\iota(P dg)df, -K \pounds(P dg)df+\frac{1}{2}Kd\iota(P dg)df)\notag \\
& = (-P d\iota(P dg)df+\frac{1}{2}P d\iota(P dg)df,-K d\iota(P dg)df+\frac{1}{2}K d\iota(P dg)df) \notag\\
& =  (-\frac{1}{2}P d\iota(P dg)df,-\frac{1}{2}K d\iota(P dg)df) .\end{align}  Also, we easily see that $[(0,df),(0,dg)]_C=0$.  Adding the 1-form components of (\ref{nij1}), (\ref{nij2}), and (\ref{nij3}) (and again using the skew symmetry of $P$), the 1-form component of equation (\ref{nij0}) yields:
\[0=\pounds(P df)Kdg-\pounds(P dg)Kdf-d(\iota(P df)Kdg)-Kd\iota(P dg)df,\] or equivalently 
\begin{equation}\label{nij4}\pounds(P df)Kdg-\pounds(P dg)Kdf=d(\iota(P df)Kdg)+Kd\iota(P dg)df.\end{equation}  On the other hand, adding the vector field components of (\ref{nij1}), (\ref{nij2}), and (\ref{nij3}) we see that
\begin{equation}\label{nij5}[P d\iota(P df)dg,P d\iota(P dg)df)=P d\iota(P dg)df.\end{equation}  With equations (\ref{nij4}) and (\ref{nij5}) in hand, we calculate
\begin{align} [\xx_f,\xx_g] & = [(P(0,df),K(0,df),(P(0,df),K(0,df))] \notag \\
& = ([P(0,df),P(0,dg)],\pounds(P (0,df))K(0,dg)-\pounds(P (0,dg))K(0,df)) \notag \\
& = (P d\iota(P dg)df, Kd\iota(P dg)df)+(0,d\iota(P df)Kdg).
\end{align}  If we define $h=\iota(P dg)df+i\iota(P df)Kdg$, we see that 
\[[\xx_f,\xx_g]=\xx_h.\]
\end{proof}

\begin{ex} In the case that $\J$ is induced by an ordinary complex structure $J$, the generalized Hamiltonian vector field associated to $f:X\to \C$ is given by 
\[(0,2\textrm{Re}(\bar{\partial}f)).\]
\end{ex}
\begin{ex}  Suppose $\J$ is induced by a symplectic structure $\omega$. Given $f:X\to \C$, write $f=f_R+if_I$ for  real-valued functions $f_R,f_I$.  The generalized Hamiltonian vector field associated to $f$ is given by 
\[\xx_f=(X_{f_R},df_I),\] where $X_{f_R}$ is the ordinary Hamiltonian vector field for the function $f_R:X\to \R$ given by formula (\ref{ordinary hamiltonian}).
\end{ex}

\section{Generalized complex submanifolds and branes}\label{GC submanifolds}
\subsection{Generalized submanifolds}

The following definition is a special case of \cite[Def. 7.4]{G}.
\begin{definition} A \emph{generalized submanifold} of a smooth manifold $X$ is a pair $(Z,F)$, where $Z\subset X$ is a submanifold, and $F\in \Omega^2(Z)$ is a closed 2-form.
\end{definition}
\begin{rem}\label{geometric brane} As discussed in the introduction, we will primarily be interested in a related structure, which we call (following \cite{KL}) a (rank 1) \emph{brane} on $X$.  Such an brane is a pair $\B=(Z,\L)$, where $Z\subset X$ is a submanifold, and $\L$ is a Hermitian line bundle with unitary connection supported on $Z$.  In particular, by setting $F\in\Omega^2(Z)$ to be the curvature form of $\L$, every such $\B$ determines a generalized submanifold (which we sometimes call the \emph{underlying} generalized submanifold of $\B$) .  When we study the deformation theory of such branes later in the paper, it will actually be conventient to use a slightly different--but essentially equivalent--definition (Definition \ref{brane definition 1})
\end{rem}

Given a submanifold $Z\subset X$, let $i:Z\hookrightarrow X$ denote the inclusion, and let $(TX)|_Z$ and $(T\uv X)|_Z$ denote the restrictions to $Z$ of the tangent and cotangent bundles of $X$.  We then have the push-forward 
\[i_*:TZ\to (TX)|_Z\] and pull-back 
\[i^*:(T\uv X)|_Z\to T\uv Z.\] 
\begin{definition}\label{gen tangent def} \cite[Def. 7.5]{G} The \emph{generalized tangent bundle} $\TT(Z,F)$  of a generalized submanifold $(Z,F)$ is the sub-bundle of $\TTX|_Z$ whose fiber at $z\in Z$ is given by 
\[\TT_z(Z,F)=\{(i_*\xi,a)\in T_zX\oplus T\uv_zX:\xi\in T_zZ, i^*a=\iota(\xi)F\}.\]
\end{definition}
\begin{rem} The generalized tangent bundle $\TT(Z,F)$ is a maximal isotropic sub-bundle of $\TTX|_Z$ with respect to the (restriction of) the pairing $\la\cdot,\cdot\ra$.  It fits into an exact sequence
\[\xymatrix{ 0 \ar[r] & \Ann(TZ) \ar[r] & \TT(Z,F) \ar[r] & TZ \ar[r] & 0}.\]  When $F=0$, there is a natural splitting of this sequence $TZ\to \TT(Z,F)$ given by 
\[\xi\mapsto (i_*\xi,0),\] which exhibits $\TT(Z,F)$ as the direct sum $TZ\oplus\Ann(TZ)$. 
\end{rem} 
\begin{definition}\label{definition K} Let \[r:\cg(X)=\Cinf(\TTX)\to \Cinf(\TTX|_Z)\] denote the restriction map.  We define $K^{(Z,F)}\subset \cg(X)$ by 
\[K^{(Z,F)}=r^{-1}(\Cinf(\TT(Z,F))).\]  In other words, $K^{(Z,F)}$ consists of those section of $\TTX$ which extend section of $\TT(Z,F)$.
\end{definition} 

\begin{lemma}\label{K closed under bracket}  The subspace $K^{(Z,F)}\subset \Cinf(\TTX)$ is closed under the Dorfman bracket.
\end{lemma}
%\begin{rem}\label{restrictbracket} Using this lemma, it is possible to show that the Dorman bracket induces a well-defined bracket on sections of $\TT(Z,F)$, as follows.  Given sections $\xx,\yy\in \Cinf(\TT(Z,F))$, choose sections $\tilde{\xx},\tilde{\yy}\in K^{(Z,F)}\subset \Cinf(\TT X)$ that extend $\xx$ and $\yy$.  By the lemma, the Dorman bracket $\ll\tilde{\xx},\tilde{\yy}\rr$ is again an element of $K^{(Z,F)}$, so we obtain a section $r(\ll\tilde{\xx},\tilde{\yy}\rr)\in\Cinf(\TT(Z,F))$.  One easily checks that the resulting section is independent of the choice of extensions $\tilde{\xx}$ and $\tilde{\yy}$.
%\end{rem}
\begin{proof}
Let $\xx=(\xi,v)$, $\yy=(\eta,w)$ be elements of $\KK^{(Z,F)}$.  By the assumption that $\xx$ and $\yy$ are elements of $\KK^{(Z,F)}$, the vector fields $\xi$ and $\eta$ are tangent to $Z$, so there exist unique vector fields $\tau,\zeta\in C^{\infty}(TZ)$ such that the restriction of $\xi$ to $Z$ is equal to $i_*\tau$ and the restriction of $\eta$ to $Z$ is equal to $i_*\zeta$.  Furthermore, we have $i^*v=\iota(\tau)F$ and $i^*w=\iota(\zeta)F$.  Since $\xi$ and $\eta$ are both tangent to $Z$, their Lie bracket $[\xi,\eta]$ is as well; in fact its restriction to $Z$ is equal to $i_*[\tau,\zeta]$.  We have 
\[\ll\xx,\yy\rr= ([\xi,\eta],\pounds(\xi)w-\iota(\eta)dv),\] so we need to show that 
\[i^*(\pounds(\xi)w-\iota(\eta)dv)=\iota([\tau,\zeta])F.\]  Expanding the right hand side using the identity $\iota([\tau,\zeta])=[\pounds(\tau),\iota(\zeta)]$ and the Cartan formula we obtain
\begin{align*} \iota([\tau,\zeta])F & = \pounds(\tau)\iota(\zeta)F-\iota(\zeta)\pounds(\tau)F \\
& = \pounds(\tau)\iota(\zeta)F-\iota(\zeta)d\iota(\tau)F-\iota(\zeta)\iota(\tau)df \\
& = \pounds(\tau)\iota(\zeta)F-\iota(\zeta)d\iota(\tau)F \\
& = \pounds(\tau)i^*v-\iota(\zeta)di^*w\\
& = i^*(\pounds(\xi)v-\iota(\eta)dw).
\end{align*}
\end{proof}

The following definition is taken from \cite{KM}.
\begin{definition} The \emph{generalized normal bundle} of the generalized submanifold $(Z,F)$ is the quotient
\[\N(Z,F)=\TTX|_Z/\TT(Z,F).\]
\end{definition}

\begin{rem}  Let $q:\TTX|_Z\to \N(Z,F)$ denote the quotient map. Since $\TT(Z,F)\subset \TTX|_Z$ is maximal isotropic, it follows that the there is a well-define pairing of $\N(Z,F)$ with $\TT(Z,F)$ given by 
\[\la q(\xx),\yy\ra=\la\xx,\yy\ra,\] which identifies $\N(Z,F)$ with the dual of $\TT(Z,F)$.
\end{rem} 
\subsection{Compatibility with a GC structure}
\begin{definition}\label{GC submanifold} \cite[Def. 7.6]{G} Let $(X,\J)$ be a GC manifold.  A generalized submanifold $(Z,F)$ of $X$ is \emph{compatible with $\J$} if 
\begin{equation}\label{compatibility} \J|_Z(\TT(Z,F))=\TT(Z,F).\end{equation}  In the case we say that $(Z,F)$ is a \emph{generalized complex submanifold} of $(X,\J)$.
\end{definition}
\begin{rem}\label{geometric brane 1} Given a brane $\B=(Z,\L)$ on $X$, as described in Remark \ref{geometric brane}, we say that $\B$ is compatible with $\J$ if its underlying generalized submanifold $(Z,F)$ is.  In this case, we simply call $\B$ a \emph{generalized complex (GC) brane}.  As mentioned above, the definition we use later (Definition \ref{GC brane definition 1}) is actually slightly different.
\end{rem}

\begin{rem}\label{rem1}It will be convenient to recast the compatibility condition (\ref{compatibility}) in a slightly different form. Given a GC manifold $(X,\J)$, define
\[Q_{\J}:\TTX\times\TTX\to X\times\R\] by 
\begin{equation}\label{Q def}Q_{\J}(\xx,\yy)=\la\J\xx,\yy\ra.\end{equation}We easily verify that $Q_{\J}$ is skew-symmetric and non-degenerate.  Furthermore, since $\TT(Z,F)$ is a maximal isotropic sub-bundle of $\TTX|_{Z}$, it follows that $\TT(Z,F)$ is preserved by $\J$ if and only if the restriction of $Q_{\J}$ to $\TT(Z,F)$ vanishes.  If we define 
\begin{equation}\label{I def}\II^Z(X)=\{f\in \Cinf(X):f|Z=0\},\end{equation} then a generalized submanifold $(Z,F)$ is compatible with $\J$ if and only if 
\[Q_{\J}(K^{(Z,F)},K^{(Z,F)})\subset \II^Z.\]
\end{rem}

\begin{ex}\label{symplectic sub} In the case where $\J=\J_J$ comes from a complex structure $J$ on $X$, it is shown in \cite[Ex. 7.7]{G} that a generalized submanifold $(Z,F)$ is compatible with $\J$ if and only if $Z$ is a complex submanifold of $X$, and $F$ is of type $(1,1)$ with respect to the induced complex structure on $Z$.
\end{ex}
\begin{ex}\label{complex sub}Suppose $\J=\J_{\omega}$ comes from a symplectic structure on $X$.  If $Z\subset X$ is a Lagrangian submanifold, then $(Z,F)$ is compatible with $\J$ if and only if $F=0$.  Conversely, a generalized submanifold of the form $(Z,0)$ is compatible with $\J$ if and only if $Z$ is Lagrangian.  On the other hand, there exist more exotic examples where $F$ is non-zero, and $Z$ is coisotropic but of dimension greater than $\textrm{Dim}_{\R}(X)/2$ \cite[Ex. 7.8]{G}
\end{ex}

\subsection{The action symmetries on GC submanifolds }

\begin{definition}  Let $(Z,F)$ be a generalized submanifold of a manifold $X$. Let $Y$ be another manifold, $\varphi:Y\to X$ , a diffeomorphism, and $u\in \Omega^1(Y)$ a 1-form. Let $g=e^u\varphi^*:X\to Y$ be the morphism in the groupoid $\cG$ introduced in Remark \ref{symmetry groupoid}. Consider the submanifold $\varphi^{-1}(Z)\subset Y$, and let $\varphi_Z:\varphi^{-1}(Z)\to Z$ denote the diffeomorphism induced by $\varphi$.  We define $g\cdot (Z,F)$ to be the generalized submanifold of $Y$ given by $(\varphi^{-1}(Z),(\varphi_Z)^*F-di_{\varphi^{-1}(Z)}^*u).$
\end{definition} 

\begin{proposition}\label{action on GC submanifolds} \begin{enumerate} \item 
\[g\cdot K^{(Z,F)}=K^{g\cdot(Z,F)}.\]
\item Given a GC structure $\J$ on $X$, a generalized submanifold $(Z,F)$ on $X$ is compatible with $\J$ if and only if $g\cdot (Z,F)$ is compatible with the GC structure $g\cdot \J:=g\cdot\J\cdot g^{-1}$ on $Y$.
\end{enumerate}
\end{proposition} 
\begin{proof}  First, let consider the case where $g$ is of the form $g=(\textrm{id}_X,u)$ for some $u\in\Omega^1(X)$; we then have 
\[g\cdot (Z,F)=(Z,F+di_Z^*u).\]  Set $F'=F+di_Z^*u$.  Given $\xx=(\xi,a)\in K^{(Z,F)}$, let $\xx'=g\cdot\xx=(\xi',a')$; we then have $\xi'=\xi$ and $a'=a-\iota(\xi)du$.  Since $\xx\in K^{(Z,F)}$, $\xi$ is tangent to $Z$. Denoting the restriction of $\xi$ to $Z$ by $\tau\in\Cinf(TZ)$, we than have \[i^*a=\iota(\tau)F.\] This implies that 
\[i^*a'=\iota(\tau)F-\iota(\tau)di^*u=\iota(\tau)F',\] and we therefore have $\xx'\in K^{(Z,F')}$.

Set $Z'=\varphi^{-1}(Z)$ and $F'=\varphi_Z^*F\in\Omega^2(Z')$, so that $g\cdot (Z,F)=(Z',F')$.  Let $i_Z:Z\hookrightarrow X$ and $i_{Z'}:Z'\hookrightarrow X$ the inclusion maps, and note that by construction we have 
\begin{equation}\label{diff com1}i_{Z}\varphi_Z=\varphi i_{Z'}\end{equation} and 
\begin{equation}\label{diff com}i_{Z'}\varphi^{-1}_Z=\varphi^{-1} i_Z.\end{equation} Given $z\in Z$, and $z'=\varphi^{-1}(z)\in Z'$, let $\Psi_z:\TT_zX\to \TT_{z'}$ denote the linear isomorphism induced by the pair $(\varphi,u)$; in other words, for each section $\xx\in\Cinf(\TT X)$, we have 
\[(g\cdot \xx)_{z'}=\Psi(\xx_z).\]  Part (1) of the Proposition is equivalent to the statement that, for each $z\in Z$ we have 
\[\Psi_z(\TT_z(Z,F))=\TT_{z'}(Z',F').\]  Given $\xx=(\xi,a)\in\T_z(Z,F)$, by definition we have $\xi=(i_Z)_*\tau$ for some $\tau\in T_zZ$,  $i_Z^*a=\iota(\tau)F$.  Let us write $(\xi',a')=\Psi_z(\xi,a)\in \TT_{z'}X$.  Defining $\tau'=(\varphi_Z^{-1})_*\tau\in T_{z'}Z'$, the equality (\ref{diff com}) implies that 
\begin{align*} \xi' & = \varphi^{-1}_*(\iota_Z)_*\tau \\
& = (\varphi^{-1}\iota_Z)_*\tau \\
& = (i_{Z'}\varphi_Z^{-1})_*\tau \\
& = (i_{Z'})_*\tau'.
\end{align*}  Similarly, using (\ref{diff com1}) we see that
\begin{align*} i_{Z'}^*a' & = i_{Z'}^*\varphi^*a \\
& = (\varphi i_{Z'})^*a \\
& = (i_Z\varphi_Z)*a \\
& = \varphi_Z^*i_Z^*a\\
& = \varphi_Z^*(\iota(\tau)F) \\
& = \iota(\tau')F'.
\end{align*}  This completes the proof of part (1).

Let $\J'$ denote $g\cdot \J$.  Using part (1) as well as Proposition \ref{leftgroupaction} we have
\[\la \J'K^{(Z',F')},K^{(Z',F')}\ra = \la g\J g^{-1}gK^{(Z,F)},gK^{(Z,F)} \ra= \varphi^* \la \J K^{(Z,F)},K^{(Z,F)}\ra.\]  On the other hand, we clearly have $\varphi^*I^Z=I^{Z'}$, so it follows that $\la \J'K^{(Z',F')},K^{(Z',F')}\ra\subset I^{Z'}$ if and only if $\la \J K^{(Z,F)},K^{(Z,F)}\ra\subset I^Z$.  Recalling the discussion in Remark \ref{rem1}, this proves part (2) of the Proposition.
\end{proof}
\begin{rem}\label{geometric brane 2} Continuing with Remarks \ref{geometric brane} and \ref{geometric brane 1}, we may define a similar action of $\cG(X)$ on the set of branes on $X$.  Given a brane $\B=(Z,\L)$ and a diffeomorphism $\varphi:X\toco X$,  define
\[g\cdot\B:=(\varphi^{-1}(Z),\varphi_Z^*\L),\] where as above $\varphi_Z:\varphi^{-1}(Z)\toco Z$ denotes the diffeomorphism induced by $\varphi$.  Given a 1-form $u\in\Omega^1(X)$, the group element $(\textrm{id}_X,u)$ acts on $\B$ changing the connection $\nabla$ on $\L$ according to
\[\nabla\mapsto \nabla-2\pi i u|_Z.\] Using Proposition \ref{action on GC submanifolds}, it is then easy to check that for any $g\in G$, if the brane $\B$ is compatible with a GC structure $\J$ then $g\cdot \B$ is compatible with the GC structure $g\cdot \J$.  In particular, the symmetries of a fixed $\J$ act on the set of $GC$ branes on $(X,\J)$. \end{rem}

\subsection{The Lie algebroid complex of a GC submanifold}\label{Lie alg cohomology}  Let $(Z,F)$ be a GC submanifold of $(X,\J)$.  Since $\J$ restricts to an endomorphism of $\TT(X,F)$ which squares to minus the identity, we have the decomposition 
\[\TT(Z,F)\otimes \C=l\oplus \bar{l},\] where $l$ is the $+i$-eigenbundle of $\TT(Z;F)\otimes \C$ and $\bar{l}$  the $-i$-eigenbundle.  By construction, $l$ is a sub-bundle of $L|Z$.  Furthermore, as described in \cite{G}, the Lie algebroid bracket on $L$ (i.e. the restriction of the Dorman bracket to $L$) induces a well-defined bracket on sections of $l$, giving $l$ the structure of a complex Lie algebroid over $Z$ with anchor map the projection $\pi:l\to TZ\otimes\C$.  Explicitly, given sections $\xx,\yy\in\Cinf(l)$, choose sections $\tilde{\xx},\tilde{\yy}\in\Cinf(L)$ that extend $\xx$ and $\yy$, i.e. such that $r(\tilde{\xx})=\xx$ and $r(\tilde{\yy})=\yy$.  Then it is straightforward to check that 
$r(\ll\tilde{\xx},\tilde{\yy}\rr)\in\Cinf(L|_Z)$ lies in $\Cinf(l)$ and is independent of the choice of extensions $\tilde{\xx}$ and $\tilde{\yy}$.   We therefore have a well-defined bracket given by the formula
\begin{equation}\label{brane bracket}\ll\xx,\yy\rr_{\B}=r(\ll\tilde{\xx},\tilde{\yy}\rr).\end{equation}  Just as we did for the Lie algebroid $L$, we construct from $l$ a differential complex 
\[\xymatrix{ \Cinf(\Lambda^0l\uv) \ar[r]^{\delta_l} & \Cinf(\Lambda^1l\uv) \ar[r]^{\delta_l} & \Cinf(\Lambda^2l\uv) \ar[r]^{\delta_l} & \cdots}.\]

Recall that we defined a linear isomorphism $\mu:\TT X\to L^{\uv}$ on any GC manifold.  Similarly, for any GC submanifold $(Z,F)$, we have a linear bijection 
\[\mu:\N^{(Z,F)}\to l\uv\] characterized by the equation
\[\mu(q(\xx))(v)=\la \xx,v\ra\] for any pair of sections $\xx\in \Cinf(\TT X|_Z)$ and $v\in \Cinf(l)$, where $q:\TT X|_Z\to \N^{(Z,F)}$ is the quotient map introduced above.  

\begin{definition} A \emph{generalized holomorphic section} of the generalized normal bundle is a section $\xx\in\Cinf(\N Z)$ satisfying
\[\delta_l\mu(\xx)=0.\]
\end{definition} 
The following proposition is an easy consequence of the definitions, and we omit the proof.
\begin{proposition}\label{induced holomorphic} Let $\xx\in\Cinf(\TT X)$ be a generalized holomorphic vector field on $X$.  Then $r(\xx)\in\Cinf(\N(Z,F))$ is a generalized holomorphic section of $\N(Z,F)$.
\end{proposition}

\section{Leaf-Wise Lagrangian GC Submanifolds}\label{LWL submanifolds}\label{LWL submanifolds}

Let $\J$ be an a generalized complex structure on a manifold $X$. As discussed in $\S\ref{GC structures}$, $\J$ induces a Poisson structure $P:T\uv X\to TX$.  Let $S=P(T\uv X)\subset TX$ be the (not necessarily constant rank) distribution induced by $P$.  The distribution $S$ inherits symplectic structure, i.e. a non-degenerate, skew-symmetric pairing $\omega$ (which is ``closed" in an appropriate sense) defined by 
\begin{equation}\label{sym dis}\omega(\xi,\eta)=a(\xi),\end{equation} where $a\in 
\Cinf(T\uv X)$ is any section with $P(a)=\eta$.  For example, when $\J$ is induced by a symplectic structure as in Example \ref{symplectic}, then $S=TX$ and the induced pairing $\omega$ is given by the original symplectic form.  
\begin{definition}\label{definition LWL submanifold} Let $(Z,F)$ be a GC submanifold of $(X,\J)$, such that $\J$ is regular at each point of $Z$.  We say that $(Z,F)$ is \emph{leaf-wise} Lagrangian if for each $z\in Z$ the intersection $T_zZ\cap S_z$ is a Lagrangian subspace of $S_z$ with respect to the pairing (\ref{sym dis}).
\end{definition}

\begin{ex} \label{standard brane} Let $(X^{m,n}_0=\R^{2m+2n},\J_0^{m,n}=\J_{\omega}\times\J_J)$ be the ``standard" GC manifold that appeared in Theorem \ref{generalized Darboux}.  Introduce coordinates $(s_0,\cdots, s_{2m},t_0,\cdots,t_{2n})$ on $X_0^{m,n}$, where $s_0,\cdots s_{2m}$ are standard coordinates on $\R^{2m}$, and $t_0,\cdots t_{2n}$ are standard coordinates on $\R^{2n}\cong \C^n$.  For a natural number $k\leq n$, and let $Z^k_0\subset X^{m,n}_0$ be the product of the Lagrangian submanifold $\R^m\subset \R^{2m}$ with the complex submanifold $\C^k\subset \C^n$;  explicitly, $Z^k_0$ is defined by the conditions $s_{m+1}=\cdots= s_{2m}=t_{2k+1}=\cdots = t_{2n}=0$.  We easily check that $Z^k_0$, equipped with the zero 2-form, is a leaf-wise Lagrangian GC submanifold of $(X_0,\J_0)$.
\end{ex}

The following result, giving a local normal form for LWL GC submanifolds, is a variation on the generalized Darboux theorem (Theorem \ref{generalized Darboux}).

\begin{theorem}\label{LWL normal form} Let $(Z,F)$ be a LWL GC sub manifold of $(X,\J)$, where $X$ is of dimension $2(m+n)$, $Z$ is of dimension $m+2k$, and $\J$ is of type $n$ at each point of  $Z$.  Then for each $z\in Z$, there exists a neighborhood $U$ of $z$ in $X$, a 1-form $u\in \Omega^1(U)$, a neighborhood $U_0$ of the origin in $X^{m,n}_0$, and a diffeomorphism $\Phi:U\to U_0$ such that 
\begin{enumerate} \item $\Phi(Z\cap U)=Z^k_0\cap U_0$ and $\Phi(z)=0$,
\item $\Phi^*(\J^{m,n}_0|_{U_0})=e^u\cdot\J|_{U}$, and 
\item $F|_{U\cap Z}=d\rho (u)$, where $\rho:\Omega^1(U)\to \Omega^1(U\cap Z)$ is the pull-back (restriction) map.
\end{enumerate}
\end{theorem}
\begin{proof}

For simplicity write $(X_0,\J_0)=(X_0^{m,n},\J_0^{m,n})$ and $Z_0^k=Z_0$.
By the generalized Darboux Theorem \ref{generalized Darboux}, we may assume without loss of generality that $X$ is an open subset of $X_0$, and $z=0$.  
Decompose that tangent space of $X_0=\R^{2m}\times\C^n$ at the origin as 
\[T_0 (\R^{2m}\times\C^n)=T_0\R^{2m}\oplus T_0\C^n.\]   Define $E=T_0Z\cap T_0\R^{2m}$, and let $K\subset T_0\C^n$ be the projection of $T_0Z$ onto $T_0\C^n$.  By assumption, $E$ is a Lagrangian subspace of $T_0\R^{2m}\cong \R^{2m}$.
\begin{lemma} $K$ is a complex subspace of $T_0\C^n\cong\C^n$.
\end{lemma}
\begin{proof} Introduce the notation $S=T_0\R^{2m}$, $W= T_0\C^{2m}$, so that we have a decomposition $\TT_0X=S\oplus W\oplus S\uv\oplus W\uv$.  For every $v\in K$, there exists $u\in S$, $a\in S\uv$ and $b\in W\uv$ such that $((u,v),(a,b))\in \TT_0(Z,F)$.  We have 
\[\J(((u,v),(a,b)))=((-\omega^{-1}(a),-J(v)),(\omega(u),J\uv(b))).\]  Since by assumption this is again an element of $\TT_0(Z,F)$, we must have \[(-\omega^{-1}(a),-J(v))\in T_0Z.\] Therefore
\[-J(v)=\pi_{T_0\C^n}(-\omega^{-1}(a),-J(v))\in K,\] so we see that $J(K)=K$.
\end{proof}

Recall that the set of linear symplectomorphisms of $\R^{2m}$ acts transitively on the set of Lagrangian subspaces, and that the set of complex linear automorphisms of $\C^n$ acts transitively on the set of complex subspaces of fixed dimension.  Therefore, by applying a linear change of coordinates on $X_0$, we may assume without loss of generality that $E$ is the subspace spanned by $\{\frac{\partial}{\partial s^i}\}_{i=1}^m$, and $K$ is the subspace spanned by $\{\frac{\partial}{\partial t^i}\}_{i=1}^{2k}$.  In particular, if we let $\pi:Z\to Z_0$ denote the projection map, then we see that 
\[\pi_*TZ\to TZ_0\] is an isomorphism at the origin.  We may therefore choose a neighborhood $\tilde{Z}\subset Z$ of the origin, such that the restriction of $\pi$ to $\tilde{Z}$ give a diffeomorphism from $\tilde{Z}$ to $\tilde{Z}_0:=\pi(\tilde{Z})$.  

Introduce the alternate coordinates $\{x^1,\cdots x^{d:=m+2k},y^1,\cdots y^{d':=m+2n-2k}\}$   given by $x_1=s_1,\cdots x_m=s_m,x_{m+1}=t_1,\cdots, x_{m+2k}=t_{2k}$ and $y_1=s_{m+1}\cdots y_m=s_{2m}, y_{m+1}=t_{2k+1}\cdots y_{m+2n-sk}=t_{2n}$.  These coordinates are compatible with the decomposition $X_0\cong Z_0\times Y$, where $(x^1,\cdots,x^d)$ are coordinates on $Z_0$ and $(y^1,\cdots ,y^{d'})$ are coordinates on $Y\cong \R^{m+2n-2k}$.  In particular, the restriction of $(x^1,\cdots, x^d)$ give coordinates on $\tilde{Z}$, and we may write 
\[\tilde{Z}=\{(x^1,\cdots, x^d,\psi^1(x^1,\cdots x^d),\cdots \psi^{d'}(x^1,\cdots, x^d):(x^1,\cdots x^d))\in \tilde{Z}_0\}\] for unique smooth functions $\{\psi^I\in \Cinf(\tilde{Z}_0)\}$. Introduce the notation $\tilde{X}_0=\tilde{Z}_0\times Y\subset X_0$, and define $\Psi:\tilde{X}_0 \to \tilde{X}_0$ by 
\[(x^1,\cdots x^d,y^1,\cdots y^{d'})  \mapsto (x^1,\cdots x^d,y^1-\psi^1(x),\cdots, y^{d'}- \psi^{d'}(x)).
\]  By construction, we have 
\[\Psi^{-1}(\tilde{Z}_0)=\tilde{Z}.\]  Choose an $\epsilon$-ball $\tilde{Y}\subset Y$ around the origin such that 
\[\Psi^{-1}(\tilde{Z}_0\times \tilde{Y})\cap Z=\tilde{Z}.\]  Define 
\[U_0=  \tilde{Z}_0\times \tilde{Y},\]
\[U=\Psi^{-1}(\tilde{Z}_0\times \tilde{Y}),\] and let the diffeomorphism
\[\Phi:U\to U_0\] be the restriction of $\Psi$ to $U$.

Let $\tilde{F}$ be the restriction of $F$ to $\tilde{Z}$.  By shrinking $\tilde{Z}$ if necessary, we can find $w\in\Omega^1(\tilde{Z})$ such that $dw=F$.  Let $\pi_{\tilde{Z}}:\tilde{X}_0\to \tilde{Z}$ be the projection, and define $u'=\pi_{\tilde{Z}}^*w$; if we denote by $\rho_{\tilde{Z}}:\Omega\ub(\tilde{X}_0)\to \Omega\ub(\tilde{Z}_0)$, then clearly $\rho(u')=w$.  Finally, let $u\in\Omega^1(U)$ be the restriction of $u'$ to $U$.

Consider the quadruple $(U,U_0,\Phi,u)$.  By construction, we have $\Phi(Z\cap U)=Z_0\cap U_0$, $\Phi(0)=0$, and $d\rho_{Z\cap U}(u)=F|_{Z\cap U}$. Therefore, the proof of Theorem \ref{LWL normal form}  will be complete if we can show that 
\begin{equation}\label{pullback of gc} e^u\Phi^*(\J_0|_{U_0})=\J_0|_U.\end{equation}  Defining $g=e^{u'}\Psi^*\in\cG(\tilde{X}_0)$ and $\tilde{\J}_0=\J_0|_{\tilde{X}_0}$, we will show that 
\[g\cdot \tilde{J}_0=\tilde{J}_0,\] which clearly implies the condition (\ref{pullback of gc}).

Let $R\subset \TTX$ be the sub-bundle spanned by $\{(\frac{\partial}{\partial y^I},0),(0,dx^i)\}_{i,I}$ and $S\subset \TTX$ the sub-bundle spanned by $\{(\frac{\partial}{\partial x^i},0),(0,dy^I)\}_{i,I}$.  The following facts are clear by inspection:
\begin{enumerate} \item $\TTX=R\oplus S$.  Furthermore, both $R$ and $S$ are maximal isotropic so the pairing gives an identification $S\cong R\uv$ and $R\cong S\uv$. 
\item $\J_0(R)=R$ and $\J_0(S)=S$.
\item $S|_{Z_0}=\TT (Z_0,0)$.
\end{enumerate}
Let $\Y$ be the set of vector fields on $X_0$ of the form $\sum_{I}c^I\frac{\partial}{\partial y^I}$ for real constants $\{c^I\}_I$.  Define also 
$\RR\subset \Cinf(R)$ and $\SS\subset \Cinf(S)$ to be the subspaces of elements $\xx\in \Cinf(R),\Cinf(S)$ satisfying $(\xi,0)\cdot\xx=0$ for all $\xi\in \Y$.
\begin{lemma}\label{basic facts} \begin{enumerate}[(1)]
\item $\ll\RR,\SS\rr\subset \RR$,
\item $\ll\RR,\RR\rr=0$.  
\item $\J_0(\SS)=\SS$, $\J_0(\RR)=\RR$.
\item For every $\xx\in\RR$, we have $(\xx\cdot\J_0)(\RR)=0 , (\xx\cdot\J_0)(\SS)\subset \RR$, and $\xx\cdot(\xx\cdot\J_0)=0$.
\item Every $\xx\in \RR$ is complete, and satisfies
\begin{equation}\label{simple action}e^{\xx}\J_0=\J_0+\xx\cdot\J_0.\end{equation}

\end{enumerate}
\end{lemma}
\begin{proof} Every $\xx\in \RR$ is of the form
\begin{equation} \label{R form}\xx=(\xi^I(x)\frac{\partial}{\partial y^I},a_i(x)dx^i),\end{equation} whereas every element $\yy\in\SS$ is of the form 
\begin{equation}\label{S form}\yy=(\eta^i(x)\frac{\partial}{\partial x^I},b_I(x)dy^I).\end{equation} We calculate
\begin{equation}\label{brack1}[\xi^I(x)\frac{\partial}{\partial y^I},\eta^i(x)\frac{\partial}{\partial x^i}]=-\eta^i\frac{\partial\xi^I}{\partial x^i}\frac{\partial}{\partial y^I},\end{equation}
\begin{align}\label{brack2} \pounds(\xi^I(x)\frac{\partial}{\partial y^I})b_I(x)dy^I& = \iota(\xi^I(x)\frac{\partial}{\partial y^I})(\frac{\partial b_I}{\partial x^i}dx^i\wedge dy^I)+d(\xi^Ib_I) \notag \\ & = -\xi^I\frac{\partial b_I}{\partial x^I}dx^i+\frac{\partial \xi_I}{\partial x^i}b_Idx^i+\xi^I\frac{\partial b_I}{\partial x^i}dx^i \notag \\
& = \frac{\partial \xi_I}{\partial x^i}b_Idx^i,
\end{align} and 
\begin{equation}\label{brack3}\iota(\eta^i\frac{\partial}{\partial x^i})d(a_idx^i) = \eta^j(\frac{\partial a^i}{\partial x^j}-\frac{\partial a^j}{\partial x^i})dx^i.\end{equation}  Using (\ref{brack1}), (\ref{brack2}), and (\ref{brack3}), we see that for $\xx$ given by (\ref{R form}) and $\yy$ by (\ref{S form}) the Dorfman bracket $\ll \xx,\yy\rr$ is equal to 
\[(-\eta^i\frac{\partial\xi^I}{\partial x^i}\frac{\partial}{\partial y^I}, \frac{\partial \xi_I}{\partial x^i}b_Idx^i+\eta^j(\frac{\partial a^i}{\partial x^j}-\frac{\partial a^j}{\partial x^i})dx^i) \in \RR.\]  This verifies part (1) of the lemma.  Part (2) is verified by a similar computation, which we omit.

Part (3) follows from an easy calculation using the explicit form of $\J_0=\J_{\omega}\times \J_J$ given by (\ref{Jsymplectic}) and (\ref{Jcomplex}).  For example, for $i=1,\cdots, m$ we have \begin{align*} & \J_0((\frac{\partial}{\partial x^i},0))=(0,-dy^i),\\ & \J_0((\frac{\partial}{\partial y^i},0))=(0,dx^i),\\ & \J_0((0,dx^i))=(-\frac{\partial}{\partial y^i},0),\\ &\J_0((0,dy^i))=(\frac{\partial}{\partial x^i}).\end{align*}
%we use the explicit form of $\J_0=\J_{\omega}\times \J_J$ given by   We see that
%\begin{enumerate} \item For 
%\item For $i=1,\cdots k$ we have $\J_0(\frac{\partial}{\partial x^{m+2i-1}})=-\frac{\partial}{\partial x^{m+2i}}, $$\J_0(\frac{\partial}{\partial x^{m+2i}})=-\frac{\partial}{\partial x^{m+2i-1}},$$\J_0(dx^{m+2i-1})=-dx^{m+2i}$, and $\J_0(\end{enumerate}

Part (4) follows from combining parts (1), (2), and (3).

To prove part (5), recall that $\xx=(\xi,a)\in\Cinf(\TT X)$ is complete if and only its vector field component $\xi\in\Cinf(TX)$ is complete, i.e. if and only if it generates a well-defined flow $\{\varphi_t:X\to X\}$ for all $t\in \R$. If $\xx\in\RR$, its vector field component is of the form $\xi^i(x)\frac{\partial}{\partial y^i}$, for which the flow is explicitly given by the formula
\[\varphi_t(x^1,\cdots, x^d,y^1,\cdots, y^{d'})=(x^1,\cdots, x^d,y^1+\xi^1(x),\cdots y^{d'}+\xi^{d'}(x).\] Given $\yy\in\Cinf(\TT X)$, for each $t\in\R$ define $\yy_t=e^{t\xx}\cdot\yy$. It follows from Proposition \ref{exp Dorfman} that the map $t\mapsto \yy_t$ is characterized by the conditions
 $\yy_0=\yy$, and 
 \begin{equation} \label{deriv} \frac{d}{ds}|_{s=t}\yy_s=\ll\xx,\yy_t\rr.\end{equation}  
  We claim that, if $\yy$ is an element of either $\RR$ or $\SS$, then 
  \begin{equation}\label{integrate}\yy_t=\yy+t\ll\xx,\yy\rr,\end{equation} and in particular
\begin{equation}\label{int1}e^{\xx}\yy=\yy+\ll\xx,\yy\rr.\end{equation}   To see this, note that $\yy_0=\yy$, and that the left-hand side of (\ref{deriv}) is equal to $\ll\xx,\yy\rr$ for all $t\in\R$.  If $\yy \in \RR$, then by part (4) of the Lemma $\ll\xx,\yy\rr=0$, so that both sides of (\ref{deriv}) vanish.  If $\yy\in \SS$, then $\ll\xx,\yy\rr\in \RR$ and therefore $\ll\xx,\ll\xx,\yy\rr\rr=0$, so that the right-hand side of (\ref{deriv}) is equal to 
  \[\ll\xx,\yy+\ll\xx,\yy\rr\rr=\ll\xx,\yy\rr,\] and again we see that equation (\ref{deriv}) holds.
  
  By inspection, $e^{\xx}\J_0$ is clearly determined by its action $\RR\oplus\SS\subset \Cinf(\TTX)$, so that part (5) of the lemma follows from equation (\ref{int1}).

\end{proof}

Consider the vector field on $V\times Y$ given by $\xi(x,y)=\psi^I(x)\frac{\partial}{\partial y^I}$.  It is easy to check that $e^{(\xi,0)}=\varphi^*$ in $\cG(V\times Y)$.  Furthermore, since $(\xi,0)$ and $(0,u)$ are both elements of $\RR$, it follows from Lemma \ref{basic facts} that 
\[(0,\pounds(\xi)u)=\ll(\xi,0),(0,u)\rr=0,\] so that $e^{\xi}u=u$.  Therefore, defining $\xx=(\xi,u)\in \RR$, we see that 
\[g:=e^u\varphi^*=e^{\xx}.\]  Therefore, to complete the proof of the Proposition, it will be sufficient to show that 
\[\xx\cdot \J_0=0.\]

Actually, we will prove the equivalent statement $-\xx\cdot \J_0=0$.  Since $-\xx\in \RR$, it follows from Lemma \ref{basic facts} that $-\xx\cdot \J$ vanishes if and only if 
\[\la(-\xx\cdot\J_0)\SS,\SS\ra = 0;\] since $g^{-1}\cdot\J_0=\J_0-\xi\cdot\J_0$ and $\J_0(\SS)=\SS$, this is equivalent to
\[\la(g^{-1}\cdot\J)\SS,\SS\ra =0,\] which is in turn equivalent to 
\[\la \J (g\cdot\SS),g\cdot\SS\ra =0.\] It follows from Proposition \ref{action on GC submanifolds} that 
\[g\cdot \KK^{(Z_0,0)}=\KK^{(Z,F)}.\]  Since $\SS\subset \KK^{(Z_0,0)}$ and $(Z,F)$ is compatible with $\J$, it follows that 
\[\la \J (g\cdot\SS),g\cdot\SS\ra\subset I^Z.\]  On the other hand, since for each $x,y\in \SS$ the quantity $\la \J (g\cdot x),g\cdot y\ra$ is invariant under the flow generated by the vector field $\xi$, it follows that if $\la \J (g\cdot x),g\cdot y\ra$ vanishes on $Z$ then it must vanish on all of $\tilde{X}_0$, so we see that 
\[\la \J (g\cdot\SS),g\cdot\SS\ra=0.\]

\end{proof}

\section{Formal symmetries of generalized complex structures}\label{section formal}

The remainder of the paper will be devoted to studying the deformation theory  of generalized complex branes.  As discussed in the introduction, our first task will be to define, for every GC brane $\B$, a functor
\[\De_{\B}:\Art\to \Set\] encoding the formal deformations of $\B$; we do this in Definition \ref{brane deformation functor}.  Before we can formulate Definition \ref{brane deformation functor}, however, it will be necessary to introduce a framework that includes certain ``formal" versions of geometric structures studied in \S\ref{Courant algebroid}-\S\ref{LWL submanifolds}.  We therefore begin by recalling some basic facts about local Artin algebras, as well as nilpotent Lie algebras and their exponentiation.  We then use this theory to recast some of the definitions and results about GC geometry presented in the previous sections.

\subsection{Artin algebras, nilpotent Lie algebras, and exponentiation}

Recall that a ring is called \emph{Artinian} if it satisfies the descending chain condition on ideals.
\begin{definition}Let $\Art$ denote the category of local, unital Artinian $\R$-algebras with residue field $\R$.  
\end{definition} 

For a discussion of Artin algebras and their use in deformation theory, see \cite{KS}\cite{Ha}. The basic facts stated below about $\Art$ may be found in these references.

\begin{rem}  Every $A\in\Art$ may be decomposed, as a real vector space, as 
\[A=\R\cdot 1_A\oplus \m,\] where $\m\subset A$ is the unique maximal ideal.  Furthermore, $\m$ is finite dimensional over $\R$, and coincides with the set of nilpotent elements of $A$.  Any $\R$-algebra homomorphism $\varphi:A'\to A$ for $A,A'\in\Art$ necessarily satisfies $\varphi(\m_{A'})\subset \m_{A}$.   
\end{rem}
%\begin{notation}Given $A\in\Art$, we denote the unique maximal ideal by $\m\da\subset A$. 
%\end{notation}
\begin{ex} For each natural number $N$, the $\R$-algebra 
\[A=\R[\epsilon]/(\epsilon^{N+1})\] is an object of $\Art$, with maximal ideal $(\epsilon^N)$. 
\end{ex}

\begin{definition} Given $A\in\Art$, a \emph{small extension} of $A$ is a pair $(A',\pi)$, where $A'$ is an object of $\Art$, and $\pi:A'\to A$ is a surjective map of $\R$-algebras, such that the kernel of $\pi$ is a principal ideal $I\subset A'$ satisfying $\m_{A'}I=0$.
\end{definition}
\begin{ex} Continuing with the previous example, for each natural number $N$ we have a small extension
\[\R[\epsilon]/(\epsilon^{N+2})\to \R[\epsilon]/(\epsilon^{N+1}).\]
\end{ex}

The following basic result will enable us to prove a number of results by inducting on small extensions \cite{Ha}.
\begin{proposition}\label{induct on small extensions}  Any surjective map in $\Art$ can be factored through a finite sequence of small extensions.  In particular, for any $A\in\Art$, the quotient map $A\to \R$ can be factored through a finite sequence of small extensions.  \end{proposition}

Let $\g$ be a real Lie algebra.  Then $\g$ determines a functor from $\Art$ to the category of real Lie algebras, given by
\begin{equation}\label{vect functor} A\mapsto \g_A:=\m\otimes_{\R}\g.\end{equation}  We will refer to such a functor as a \emph{ Lie algebra over $\Art$}.  As a slight abuse of notation, we will sometimes use the same symbol ``$\g$" to refer both to the Lie algebra itself and the Lie algebra over $\Art$ it determines; the meaning will always be clear from context.
%As another notational note: given a morphism $\tau:A\to A'$ in $\Art$, we will denote the induced map $\g_A\to \g_{A'}$  by the same symbol $\tau$, i.e. for every $a\in \m$ and every $v\in \g$ we have 
%\[\tau:a\otimes v\mapsto \tau(a)\otimes v.\]

\begin{rem}\label{over art} More generally, for any category $\CC$, we may define an object of $\CC$ over $\Art$ to be functor $\Art\to \CC$, and a morphism (over $\Art$) between such objects to be a natural transformation. 
For example, a real vector space $V$ determines two different vector spaces over $\Art$, given respectively by 
\[A\mapsto A\otimes_{\R}V\] or 
\[A\mapsto \m\otimes_{\R}V.\] A representation of a group $G$ over $\Art$ on a  vector space $V$ over $\Art$ consists of a representation of $G\da$ on $V\da$ for each $A\in\Art$, the collection of which must be suitably compatible with morphisms $A\to A'$ in $\Art$.  

The groups over $\Art$ we consider in this paper all have the property that their value on the trivial Artin algebra $\R\in\Art$ is the trivial group.  We will refer to such a group as a \emph{formal group over $\Art$}, or more succinctly as a \emph{formal group}.
  \end{rem}

To each real Lie algebra $\g$, recall that we can associate a formal group over $\Art$ as follows. For each $A\in\Art$, the Lie algebra $\g\da:=\m\otimes_{\R}\g$ is nilpotent, and  may therefore be  ``exponentiated" it to form a group $e^{\g\da}$.  By definition, there is a canonical bijective map from the underlying set of $\g\da$ to the underlying set of $e^{\g\da}$, which we write as $x\mapsto e^{x}$.  For each pair of elements $x,y\in\g\da$, the product of $e^x$ with $e^y$ is defined using the Baker-Campbell-Hausdorff formula 
\[e^x e^y=e^{x+y+\frac{1}{2}[x,y]+\cdots},\] where the formally infinite sum $x+y+\frac{1}{2}[x,y]\cdots$ terminates after a finite number of terms due to the nilpotence.  This construction $\g\mapsto e^{\g}$ is functorial: every homomorphism of Lie algebras $\varphi:\g\to\g'$ induces a  homomorphism of formal groups  $e^{\varphi}:e^{\g}\to e^{\g'}$, given for each Artin algebra $A\in\Art$ and each $x\in\g\da$ by the formula $e^{x}\mapsto e^{\varphi\da(x)}$.  

%\begin{rem} As we did here, we will often define constructions involving objects ``over $\Art$" by simply describing them for each fixed Artin algebra $A\in\Art$.  This is done to streamline the exposition, and in all case the extra structure involving homomorphisms $A\to A'$ in $\Art$ should be clear from the context.  
%\end{rem}
We record here for later use some well-known facts about this construction.   
    
\begin{proposition}\label{exp prop} Let $\g$ be a real Lie algebra with corresponding formal group $e^{\g}$.
\begin{enumerate} \item   Every action of $\g$ on a real vector space $V$ induces an action of $e^{\g}$ on $V$ (viewed as a vector space $A\mapsto A\otimes_{\R} V$ over $\Art$ ). This action is given by the exponential formula: for each Artin algebra $A\in\Art$, each $x\in\g\da$, and each $v\in A\otimes_{\R}V$, we set\begin{equation}\label{exp action}e^{x}v=v+x\cdot v+\frac{1}{2}x\cdot(x\cdot v)\cdots.\end{equation} 

\item Consider the action of $e^{\g}$ on $\g$ obtained by exponentiating the adjoint representation of $\g$ on itself, i.e. for every Artin algebra $A$ and every $x,y\in A\otimes_{\R}\g$ have
\[e^{x}y=y+[x,y]+\frac{1}{2}[x,[x,y]]+\cdots.\]  Then for every $x,y\in\g\da$, we have the identity 
\[e^x e^y e^{-x}=e^{e^{x}y}.\]
\item Fix $A\in\Art$.  Every group homomorphism $\varphi:\R\to e^{\g\da}$ is of the form 
\begin{equation}\label{1parsubformal}t\mapsto e^{tx}\end{equation} for a unique $x\in\g\da$.  Conversely, for each $x\in\g\da$, the formula (\ref{1parsubformal}) determines a homomorphism $\R\to e^{\g\da}$.
\end{enumerate}
\end{proposition}
\subsection{The formal symmetries of a Courant algebroid}
Fix a smooth manifold $X$.  As in \S\ref{Courant algebroid}, we denote by $\g(X)$ the real Lie algebra of vector fields on $X$, and by $\cg(X)=\g(X)\ltimes \Omega^1(X)$ the semi-direct product Lie algebra introduced in Definition \ref{definition cg}.  Following the convention  introduced above, we may also regard $\g(X)$ and $\cg(X)$ as a Lie algebras over $\Art$, which assign to each $A\in\Art$ with maximal ideal $\m\subset A$ the nilpotent Lie algebras $\g\da(X):=\m\otimes\g(X)$ and $\cg\da(X)=\m\otimes_{\R}\cg(X)$, respectively.   More generally, we have sheaves of Lie algebras over $\Art$, which assign to each open set $U\subset X$ and each $A\in\Art$ the Lie algebras $\g\da(U)=\m\otimes_{\R}\g(U)$ and $\cg(U)=\m\otimes_{\R}\g(U)$.

 We also introduce the notation for differential forms
\[\Omega\ub\da(X):=A\otimes_{\R}\Omega\ub(X)\] and 
\[\Omega\ub\dm(X):=\m\otimes_{\R}\Omega\ub(X),\]  and similarly
\[\Cinf\da(TX):=A\otimes_{\R} \Cinf(TX),\] 
\[\Cinf\dm(TX):=\m\otimes_{\R} \Cinf (TX),\]
\[\Cinf\da(\TTX):=A\otimes_{\R}\Cinf(\TTX),\] and 
\[\Cinf\dm(\TTX):=\m\otimes_{\R}\Cinf(\TTX).\]  A potentially confusing  point about this notation: while $\Cinf(TX)$ and $\g(X)$ both denote the space of vector fields on $X$, given $A\in\Art$ the spaces $\Cinf\da(TX)$ and $\g\da(X)$ are not the same.  This corresponds to the different roles played by the space of vector fields in our development.

Consider the formal group $e^{\g(X)}$, whose elements may be viewed as infinitesimal diffeomorphisms of $X$.  The  action of $\g(X)$ on the spaces $\Omega\ub(X)$ and $\Cinf(TX)$ via the Lie derivative may be exponentiated (as described in Proposition \ref{exp prop}) to an action of $e^{\g(X)}$.   Unsurprisingly, this action is compatible with the relevant structures on $\Omega\ub(X)$, and $\Cinf(TX)$,  such as the wedge product of differential forms, and the contraction operation between vector fields and forms.  This is summarized in the following Proposition, whose proof  is deferred to the appendix. 

\begin{proposition}\label{formal courant action} Fix an Artin algebra $A\in\Art$.  For each $g\in e^{\g\da(X)}$, each $a,b\in A\otimes_{\R} \Omb(X)$, and each $\xi,\eta \in A\otimes_{\R} \Cinf(TX)$ we have 
\[g(da)=dg(a),\]
\[g(a\wedge b)=g(a)\wedge g(b),\]
\[g(\iota(\xi)a)=\iota(g(\xi))g(a),\] and
 \[g([\xi,\eta])=[g(\xi),g(\eta)].\]  
\end{proposition}

Turning next to the formal group $e^{\cg(X)}$, we may similarly exponentiate the action of $\cg(X)$ on $\Cinf(\TTX)$ (described in Proposition \ref{Courant action prop}) to give an action of $e^{\cg(X)}$.  Using Proposition \ref{formal courant action}), we may then easily prove the following.

\begin{proposition} \label{action} For each $A\in\Art$, the action of $e^{\cg\da(X)}$ on $\Cinf\da(\TTX)$ is compatible with the (A-linear extensions of the) Dorfman bracket (\ref{Dorfman Bracket}), the natural pairing $\la\cdot,\cdot\ra$, and the projection $\pi:\Cinf\da(\TTX)\to\Cinf\da(TX)$.
\end{proposition}

%The corresponding sheaf of formal groups is then given by $U\mapsto $ denote the space $\Cinf(U)$ of smooth vector fields on $U$.  We view the sheaf $U\mapsto \g(U)$ as a sheaf of real Lie algebras on $X$.  Adopting the notation for Lie algebras introduced above,   denote thesmooth vector fields on $U$  Using a similar notation to \S\ref{Courant algebroid},   we denote the sheaf of vector fields on $X$ (viewed as a sheaf of real Llet $\g(X)$ denote the Lie algebra $\Cinf(TX)$ of vector fields on $X$.  

% denote the real Lie algebra of vector fields on $X$.  We also introduce the notation $\h$ for the space of differential from $\Omega\ub(X)=\oplus_{i=0}^{\textrm{Dim}_{\R}(X)}\Omega^i(X)$, regarded as an abelian Lie algebra.  Consider the semi-direct product Lie algebra $\g(X)\ltimes\h(X)$:  the underlying vector space is the direct sum $\g(X)\oplus \h(X)$, and the bracket is given by 
%\begin{equation}\label{bracket}[\xi+a,\eta+b]=[\xi,\eta]+\pounds(\xi)b-\pounds(\eta)a\end{equation} for $\xi,\eta\in \g(X)$ and $a,b\in \h(X)$. Note that this has the Lie subalgebra $\cg(X)=\g(X)\otimes\Omega^1(X)$ introduced in Definition \ref{definition cg}.  

It will be convenient to have a formal version of Proposition (\ref{exp Dorfman}), which gives an explicit formula for the exponential map $\cg\da(X)\mapsto e^{\cg\da(X)}$ in terms of the subgroups $e^{\g\da(X)}$ and $e^{\m\otimes_{\R}\Omega^1(X)}$; introducing  the notation $\h(X)$ for $\Omega^1(X)$ (regarded as an abelian Lie algebra), the latter group may be denoted by $e^{\h\da(X)}$. Adapting the same notation used in the statement of Proposition (\ref{exp Dorfman}), for each $\xi\in\g\da(X)$ and each $a\in\h\da(X)$, we define
\begin{equation}\label{formal a def}a^{\xi}=\int_0^1e^{t\xi}(a)dt,\end{equation} where   the precise meaning of the right-hand side is as follows: for each $\xi\in \g\da(X)$ and $a\in \h\da(X)$, the expression 
\[e^{t\xi}(a):=a+t\pounds(\xi)a+\frac{t^2}{2}\pounds(\xi)\pounds(\xi)a+\cdots\] defines a polynomial in $t$, i.e. an element of $\h\da(X)[t]$.  The integral with respect to $t$ in (\ref{formal a def}) is then defined algebraically: for each natural number $n$ and each $b\in\h\da$ we set 
\[\int_0^1 bt^ndt:=\frac{b}{n+1}.\]  
\begin{proposition}\label{exponential} For all $(\xi,a)\in \g\da(X)\ltimes\h\da(X)$, we have 
\[e^{(\xi,a)}=e^{(0,a^{\xi})}e^{(\xi,0)}.\]
\end{proposition}

\begin{proof}  Given $(\xi,a)\in\cg\da(X)$, consider the map 
\[\varphi:\R\to e^{\cg\da(X)}\] given by 
\[t\mapsto e^{\int_0^te^{s\xi}ads}e^{t\xi}.\]  By Proposition \ref{exp prop},  for each $\xi\in\g\da(X)$ and each $a\in \h\da(X)$ we have 
\[e^{\xi}e^{a}e^{-\xi}=e^{e^{\xi}(a)};\]  the same calculation used in the proof of Proposition \ref{exp Dorfman} may then be used to show that $\varphi:\R\to e^{\cg\da(X)}$ is a group homomorphism.  Therefore, by part (2) of Proposition \ref{exp prop}, it follows that $\varphi$ is of the form 
\begin{equation}\label{poly1}t\mapsto e^{t\xx}\end{equation} for a unique $\xx\in\cg\da(X)$.  Note that 
\[a^{t\xi}=\int_0^te^{s\xi}ads=ta+t\sum_{k=1}^{\infty}\frac{t^k}{(k+1)!}\xi^k\cdot a:=ta+th(t),\] where $h(t)$ is an element of $\cg\da(X)[t]$ (depending on $\xi$ and $a$), so that by the Baker-Campbell-Hausdorff formula we have 
\begin{equation}\label{poly2} e^{a^{t\xi}}e^{t\xi}=e^{t(\xi+a)+t^2\tilde{h}(t)},\end{equation} where $\tilde{h}(t)$ is an element of $\cg\da(X)[t]$.  Comparing equations (\ref{poly1}) and (\ref{poly2}), we see that $\tilde{h}(t)=0$ and $\xx=(\xi,a)$.
\end{proof}

  By construction, we have inclusions of $e^{\g\da(X)}$ and $e^{\h\da(X)}$ as sub-groups of $e^{\g\da(X)\ltimes\h\da(X)}$. Proposition \ref{exponential} implies that these subgroups (which have trivial intersection) generate $e^{\g\da\ltimes\h\da(X)}$, so that $e^{\g\da(X)\ltimes\h\da(X)}$ may be identified with the semi-direct product $e^{\g\da(X)}\ltimes e^{\h\da(X)}$ (this is of course a special case of a more general result about the exponentiation functor). In particular, we have the following corollary, recorded here for later use.
\begin{corollary}\label{exp formula cor} For each $\xi\in\g\da(X)$, the linear map $\h\da(X)\to\h\da(X)$ mapping $a\mapsto a^{\xi}$ is a bijection.
\end{corollary}

%Consider the group $e^{\g\da(X)}$, which is a formal version of the group $G(X)=\textrm{Diff}(X)^{op}$ studied in \S\ref{Courant algebroid}.  The actions of the Lie algebra $\g(X)$ on $\Omega\ub(X)$ and $\Cinf(TX)$ via the Lie derivative extend via A-linearity to actions of $\g\da(X)$ on the spaces 
%\[\Omega\da\ub(X):=A\otimes_{\R}\Omega\ub(X)\] and 
%\[\Cinf\da(TX):=A\otimes_{\R}\Cinf(TX).\]  By Proposition \ref{exp prop}, 
%the group $e^{\g\da(X)}$ therefore acts on these spaces as well, via the exponential formula; this is simply a formal version of the action of $\textrm{Diff}(X)^{op}$ on forms and vector fields by pull-back. The proof of the following result is differed to the appendix.  

%Consider next the Lie algebra $\cg\da(X)$ (which as a vector space is simply $\m\otimes_{\R}\Cinf(\TTX)$, and its corresponding group $e^{\cg\da(X)}$; the latter is a formal version of the group $\cG(X)$ introduced in \S\ref{Courant algebroid}.  In Proposition \ref{Courant action prop}, we defined actions of $\cg(X)$ on the vector spaces $\Cinf(\TTX)$ and $\Cinf(\End(\TTX))$.  Extending by $A$-linearity, these induce actions of $\cg\da(X)$ (and hence $e^{\cg\da(X)}$) on the spaces
%\[\Cinf\da(\TTX):=A\otimes_{\R}\Cinf(\TTX)\] and 
%\[\Cinf\da(\End(\TTX)):=A\otimes_{\R}\Cinf(\End(\TTX)).\]  

%The following result is an easy Corollary of Proposition \ref{formal courant action}.

Next, let $\J$ be a GC structure on $X$.  Recall from \S\ref{GC structures} the Lie algebra $\T(X)$ of generalized holomorphic vector fields (Definition \ref{generalized holomorphic vector field}), which are the infinitesimal symmetries of $\J$, as well as the sub-algebra $\H(X)\subset \T(X)$ of generalized Hamiltonian vector fields (Definition \ref{genhamdef}).  For each $A\in\Art$ we may then define the nilpotent Lie algebra $\T\da(X)$ as well as the subalgebra $\H\da(X)\subset \T\da(X)$.  By inspection of the exponential formula (\ref{exp action}), we immediately see the following
 \begin{proposition}\label{formal holomorphic symmetries}  For each $\xx\in\T\da(X)$ we have 
\[e^{\xx}\cdot\J=J.\]

\end{proposition}

\section{The deformation functor of a generalized complex brane}\label{section deformations}
Using the framework introduced in the last section, we now to turn to the problem of defining for each $GC$ brane $\B$ a functor
\[\df_{\B}:\Art\to \Set\] encoding the formal deformations of $\B$.  As mentioned in the introduction, this functor (given in Definition \ref{brane deformation functor}), will be constructed in terms of a formal groupoid $\De\uB(X,\J)$ over $\Art$ (Definition \ref{full deformation groupoid}), which encodes the appropriate notions of equivalence between deformations of $\B$.

%A GC brane on a GC manifold $(X,\J)$, as defined in \bnote{make sure to define earlier}, consists of a submanifold $Z\subset X$ together with a Hermitian line bundle with unitary connection $\L$ supported on $Z$.   The collection of such deformations are most naturally viewed as a groupoid (or sheaf of groupoids).  To deal with this added layer of structure we will need to introduce some auxiliary machinary.  For example, we will need the notion of a group acting on a category.  None of this theory is particular difficult, but we have tried to lay it out in as carefully as possible.  We should also note that we have found it useful define the category of line bundles with connection in a ``Cech" framework, defined with respect to an open cover.  While it is probably possible to formulate everything wtihout resorting to choosing an open cover, we have found the concreteness of this approach to be very helpful.  
\subsection{Deformations of a Hermitian line bundle with unitary connection}

Given a brane $\B=(Z,\L)$, we may simultaneously deform both $Z$ and $\L$.  As a preliminary to studying the general situation, we first consider the deformations of $\L$ only, i.e. the deformations of a Hermitian line bundle with unitary connection over a fixed space.

Let $X$ be a smooth manifold, and $\mathcal{W}=\{W_I\}$ an open cover of $X$.  We denote by $W_{I_1 I_2\cdots I_k}$ the $k$-fold intersection $W_{I_1}\cap \cdots W_{I_k}$.  
\begin{definition}  The set $\Herm(X)$ consists of triples $\L=(\{W\si\},\{c\sij\},\{a\si\})$, where $\{W\si\}$ is an open cover of $X$,  $\{c_{IJ}\in \Cinf(W_{IJ})\}$ is a collection of real-valued functions  satisfying 
\[c_{JK}-c_{IK}+c_{IJ}\in \Z,\]  and  $\{a_{I}\in \Omega^1(W_I)\}$ is a collection of (real) 1-forms satisfying 
\[a_J-a_I=dc_{IJ}\] on the intersections $W_{IJ}$.
\end{definition} 
\begin{rem}\label{gluing bundles}    Given an element $\L=(\W,\{c\sij\},\{a\si\})\in \Herm(X)$ as above, we may construct a Hermitian line bundle with unitary connection on $X$ by gluing.   The line bundle is constructed by gluing  trivial (Hermitian) line bundles on each $W_I$  using transition functions $g\sij=e^{2\pi i c_{IJ}}$. The unitary connection is then locally specified on each $W\si$ by 
\[\nabla=d+2\pi i.\]  Furthermore, for each manifold $X$, there exists an open cover $\W$ such that every Hermitian line bundle with unitary connection is isomorphic to one of this form.  More formally, the gluing construction may be extended to an equivalence of categories (defined on objects as above).  \end{rem}

\begin{notation} For the most part, we will work with a fixed open cover $\{W_I\}$ and omit it from the notation, i.e. we will simply write $\L=(\{c_{IJ}\},\{a\si\})$ to specify an element of $\Herm(X)$.
\end{notation}

\begin{rem}  An element $\L=(\{c\sij\},\{a\si\})\in \Herm(X,\mathcal{W})$ determines a closed 2-form $F\in \Omega^2(X)$, which is given by  $F|_{W_I}=da_I$ on each open set $W\si$. 
\end{rem}
\begin{rem}   Given an open set $U\subset X$, there is a natural restriction map $\Herm(X)\to \Herm(U)$, which we denote by $\L\mapsto \L|_U$ for $\L\in\Herm(X)$.
\end{rem}

 %\begin{rem}Given such a pair
%\[\L=(\{a_I\},\{c_{IJ}\}),\] we may construct a Hermitian line bundle with unitary connection $\L_{glue}=(L,\nabla)$ over $Z$.  By construction, over each open set $W_I$ the bundle $L$ has a canonical unitary section $s\si$ satisfying 
%\[\nabla s_I=2\pi i a_Is_I .\]  These sections are related on overlaps by  
%\[s\sj=e^{2\pi ic\sij}s\si.\]  Moreover, given an arbitrary Hermitian line bundle with unitary connection $(L,\nabla)$ over $Z$, we can always find an open cover $\{W\si\}$ and a collection of data $(\{a_I\},\{c_{IJ}\})$ defined relative to $\{W\si\}$, such that $(L,\nabla)$ is isomorphic to $\L_{glue}$.
%\end{rem}

\begin{rem}\label{groupoid over art} Recall from Remark \ref{over art},  that for any category $\CC$ we defined an \emph{object of $\CC$ over $\Art$} to be a functor $\Art\to \CC$.  A similar definition can be made when $\CC$ is a 2-category, for example the 2-category $\textrm{Gpd}$ of (small) groupoids.  For example a (strict) \emph{groupoid over $\Art$} may be defined as a (strict) functor $\Art\to \textrm{Gpd}$; explicitly, this consists of a groupoid $\G\da$ for every $A\in\Art$, and for every homomorphism $A\to A'\in \Art$ a functor $\G\da\to\G_{A'}$.  For our present purposes, we require the composition of these functors to be strictly compatible with composition in $\Art$.  The groupoids over $\Art$ we consider will all have the property that they map the trivial Artin algebra $\R\in\Art$ to a trivial groupoid, i.e. a groupoid with a single object and a single morphism.  We will refer these as \emph{formal groupoids} (over $\Art$).  The notion of a (strict) functor between formal groupoids is defined in the obvious way.  We will sometimes drop the modifier ``formal" when it is clear from context.

To streamline the exposition, we will usually define specific formal groupoids (and functors between them) by simply describing their value on each $A\in\Art$. The additional structure needed to make things completely precise will always be clear from the context.\end{rem}

\begin{definition} \label{Herm def} Given $\L=(\{c\sij\},\{a\si\})\in\Herm(X)$, we define a formal groupoid $\De^{\L}(X)$  (over $\Art$) as follows.  For each $A\in\Art$, an object of $\De^{\L}_A(X)$, which we call an \emph{A-deformation} of $\L$,  is a pair $\hL=(\{\hc\sij\},\{\ha\si\})$, where
\begin{enumerate}[(i)]\item each $\ha\si\in\Omega^1_A(W\si)$, and is of the form 
\[\ha\si=a\si+u\si\] for $u\si\in\Omega^1\dm(U\si)$,% \bnote{Maybe want to say earlier that there is natural inclusion of $\Omega\ub(X)$ into $\Omega\ub\da(X)$ for every $A\in\Art$.  Might also want to say we have canonical isomorphism $\Omega\ub_{\R}(X)\cong \Omega\ub(X)$, and the inclusion of $\Omega\ub$ into $\Omega\ub\da$ composed with the homomorphism induced by the augmentation map is the identity on $\Omega\ub$.}
\item each $\hc\sij\in \Cinf\da(W\sij)$ is of the form 
\[\hc\sij=c\sij+f\sij\] for $f\sij\in \Cinf\dm(W\sij)$,
\item on each overlap $W\sij$ we have  \[\ha\sj-\ha\si=d\hc\sij,\] or equivalently $u\sj-u\si=df\sij$,  
\item on each 3-fold intersection $W_{IJK}$ we have \[\hc_{JK}-\hc_{IK}+\hc_{IJ}\in \Z,\] or equivalently $f_{JK}-f_{IK}+f_{IJ}=0$.
\end{enumerate}
Given $A$-deformations $\hL=(\{\hc\sij\},\{\ha\si\})$ and $\hL'=(\{\hc'\sij\},\{\ha'\si\})$ of $\L$, an \emph{isomorphism} $\hL\to\hL'$ is a collection $g=\{g\si\in \Cinf\dm(W\si)\}$ such that 
\begin{enumerate}\item $\ha'\si=\ha\si+dg\si$ ($\Leftrightarrow u'\si=u\si+dg\si$) and 
\item $\hc'\sij=\hc\sij+g\sj-g\si$ ($\Leftrightarrow f'\sij=f\sij+g\sj-g\si$).

The composition of isomorphisms is given by addition.  For any  $A$-deformation of $\L$, the identity isomorphism is given by $\{g_I=0\}$.

\end{enumerate}
\end{definition}
%\begin{rem}\label{2 functoriality} \bnote{May be useful to introduce 2-category of ``strict groupoids over $\Art$".  Could then combine with next remark by saying we get a sheaf on $X$ valued in the 2-category of strict groupoids over $A$.} Given $\L\in\Herm(X)$ and  $A,A'\in\Art$, a homomorphism $\varphi:A\to A'$ induces a functor $\De_{\L}(\varphi):\De^{\L}_A(X)\to \De_{\L}(A')$, defined in an obvious way.  Furthermore, if we have another homomorphism $\varphi':A'\to A''$ in $\Art$, the composition of functors $\De_{\L}(\varphi')\circ\De_{\L}(\varphi):\De^{\L}_A(X)\to \De_{\L}(A'')$ is equal (on the nose) to $\De_{\L}(\varphi'\circ\varphi)$.  Thus, we have a (strict) 2-functor from $\Art$ to the 2-category of groupoids.
%\end{rem}
\begin{rem}\label{bundle def sheaf} Given an open set $U\subset X$, let us introduce the notation 
\[\De^{\L}(U):=\De^{\L|_U}(U).\] For every inclusion of open subset $V\subset U$, we then have a natural restriction functor (between formal groupoids)
\[\De^{\L}(U)\to\De^{\L}(V).\]  Altogether, we obtain a (strict) sheaf of formal groupoids on $X$ sending $U$ to the formal groupoid  $\De^{\L}(U)$.
\end{rem}

%\begin{rem} Given an open set $U\subset X$, there is a natural restriction functor 
%\begin{equation}\label{deformation restriction} \De^{\L}_A(X)\to \De_{\L|_U}(A).\end{equation}  Formally, we have a sheaf of groupoids on $X$ whose groupoid of global sections is $\De^{\L}_A(X)$.  Furthermore, for each map of Artin rings $A\to B$ the structure maps of this sheaf (for example the restriction functors (\ref{deformation restriction})) are compatible with the maps $\De_{\L|_U}(A)\to \De_{\L|_U}(B)$.  More precisely, we have a sheaf of formal deformation functors on $X$ which assigns to each open set $U\subset X$ the deformation functor $A\mapsto \De_{\L|_U}(A)$. \rnote{Fix this discussion}
%\end{rem}

\begin{proposition} \label{pi zero bundle}Every $A$-deformation of $\L$ is isomorphic to one of the form $\hL=(\{c\sij\},\{\ha\si\})$, i.e. to one with undeformed transition functions $\hat{c}\sij=c\sij$.
\end{proposition} 
\begin{proof}  Given $\hL=(\{c_{IJ}+f_{IJ}\},\{a_I+u_I\})$, by part (II) of Definition \ref{Herm def} we see that 
\[f_{JK}-f_{IK}+f_{IJ}=0.\]  The functions $f_{IJ}\in \Cinf(W_{IJ})\otimes \m\da$ therefore define a Cech cocycle.  Since the sheaf $\underline{C}^{\infty}$ of smooth, real-valued functions on $X$ admit partitions of unity, so does the sheaf $\underline{C}^{\infty}\otimes \m\da$.  We may therefore choose $\{g_I\in\Cinf\dm(W_I)\}$ satisfying 
\[g_J-g_I=f_{IJ}\] on the overlaps $W_{IJ}$.  This implies that $\hL$ is isomorphic to $\hL'=(\{c_{IJ}\},\{a_I+dg_I\})$.
\end{proof}

\begin{definition} \label{group action on category} Let $\CC$ be a groupoid, and $G$ a group.  A \emph{strict left action of $G$ on $\CC$} is a collection of functors $\{F_g\}_{g\in G}$ satisfying the following:
\begin{enumerate} \item For every $g,g'\in G$, we have $F_g\circ F_{g'}=F_{gg'}$.
%\item For every $g\in G$, we have $F_{g^{-1}}=F^{-1}_g$.
\item The functor $F_1$ induced by the identity element $1\in G$ is the identity functor.
\end{enumerate}  A strict \emph{right} action is defined similarly, except that for every $g,g'\in G$ we require that $F_{g}\circ F_{g'}=F_{g'g}$.
\end{definition}
\begin{rem} Given a formal group $G$ and a formal groupoid $\CC$, we may define an action of $G$ on $\CC$ to consist of an action of $G\da$ on $\CC\da$ for each $A\in\Art$, subject to some obvious compatibility relations.
\end{rem}

\begin{notation} Given a strict left action of $G$ on $\CC$, we will use the notation 
\[F_g(x)=g\cdot x\] for each $g\in G$ and each object $x\in \CC$.  Similarly, given a morphism $\varphi:x\to y$ in $\CC$ we denote $F_g(\varphi)$ by $g\cdot\varphi$.  Similarly, if $G$ acts on $\CC$ on the right, we will use the notation $x\cdot g$ and $\varphi\cdot g$.
\end{notation}

We now describe an action of $e^{\cg(X)}$ (the formal symmetries of the Courant algebroid) on $\De^{\L}(X)$. 

 \begin{definition}\label{bundle action} Given $A\in\Art$, $\hL=(\{\hc\sij\},\{\ha\si\})\in\De^{\L}_A(X)$, and $x=e^ue^{\tau}\in e^{\cg\da(Z)}$, let 
\[x\cdot\hL=(\{e^{\tau|_{W\sij}}\hc\sij\},\{e^{\tau|_{W\si}}\ha_I-u|_{W\si}\})\in \De^{\L}_A(X).\]  Given an isomorphism $\{g\si\}:\hL\to\hL'$ in $\De^{\L}\da(X)$, define 
\[x\cdot \{g\si\}=\{e^{\tau}g\si\}:x\cdot\hL\to x\cdot\hL'\]
\end{definition}
We then have the following easy result, the proof of which is omitted.
\begin{proposition} \label{bundle action}Definition \ref{bundle action} determines a strict left action of $e^{\cg(X)}$ on $\De^{\L}(X)$.

\end{proposition}

\subsection{Deformations of branes}  We now turn to the deformations of GC branes.  For each GC brane $\B$ on a GC manifold $(X,\J)$, we will define a formal groupoid $\De^{\B}(X,\J)$ encoding the infinitesimal deformations of $\B$ and their equivalences (Definition \ref{full deformation groupoid}).  By passing to equivalence classes of deformations, we construct from $\B$ a functor (Definition \ref{brane deformation functor}) 
\[\df_{\B}:=\pi_0(\De^{\B}(X,\J)):\Art\to \Set.\]  We will see that in many situations it is necessary to work with the formal groupoid $\De\ub(X,\J)$ itself, and not the functor $\df_{\B}$.

The construction of $\De\uB(X,\J)$ will be given in several steps.  First, we ignore the GC structure $\J$ and define a formal groupoid $\Dt\uB(X)$ that does not encode any compatibility condition with respect to $\J$.  We then define a sub-groupoid $\Dt\uB(X,\J)$ whose objects are those deformations compatible with $\J$.  The formal group $e^{\T(X)}$ of symmetries of $(X,\J)$ acts on $\Dt\uB(X,\J)$;  in particular, there is an action of $e^{\H(X)}$,  the formal generalized Hamiltonian symmetries.  Incorporating this action leads to the formal groupoid $\De\ub(X,\J)$.
%With the required preliminaries out of the way, we are finally ready in this section to formulate Definition \ref{brane deformation functor} of the deformation functor for a GC brane.  This will be done in several steps.  First, we will ignore the GC structure and define a groupoid over $\Art$ encoding the deformations of $\B$ that are not necessarily compatible with a GC structure. We explain how the formal symmetries of the Courant algebroid act on such deformations. We next introduce a GC structure and define a subgroupoid whose objects are those deformations compatible with the GC structure.  The subgroup of the Courant algebroid symmetries that preserve the GC structure act on this groupoid, and in particular the subgroup of generalized hamiltonian symmetries.  Encorporating this action leads to the definition of the groupoid \ref{}.  Finally, the deformation functor of the brane is defined by passing to equivalence classes, i.e. by taking $\pi_0$ to the deformation groupoid. 

\begin{definition}\label{brane definition 1} Let $X$ be a smooth manifold.  A (rank 1) \emph{brane} on $X$ is a pair $\B=(Z,\L)$, where $Z\subset X$ is a smooth submanifold and $\L$ is an element of $\Herm(Z)$.  We denote the collection of all such branes on $X$ by $\Br(X)$.
\end{definition}
\begin{rem}  Given such a brane $(Z,\L)$, applying the gluing construction described in Remark \ref{gluing bundles} to $\L$ determines a GC brane in the sense of Remark \ref{geometric brane}.
\end{rem}

\begin{definition} \label{definition par def}  Given a manifold $X$, and a brane $\B\in \Br(X)$, let $\Dt^{\B}(X)$ be the following formal groupoid.  
\begin{enumerate} \item For each $A\in\Art$, an object of $\Dt_A^{\B}(X)$ is a pair $\hB=(\hrho,\hL)$, where 
\begin{enumerate} \item $\hrho:\Omega\ub\da(X)\to \Omega\ub\da(Z)$ is of the form $\hrho=\rho e^{\xi}$ for some $\xi\in \g\da(X)$, where $\rho:\Omega\ub(X)\to\Omega\ub(Z)$ denotes the pull-back map for the inclusion $i:Z\hookrightarrow X$.
\item $\hL$ is an object of  $\De^{\L}_A(Z)$, as in Definition \ref{Herm def}.
\end{enumerate}  We will refer to such an object as an \emph{$A$-deformation of $\B$}.
\item Given two $A$-deformations $\hB=(\hrho,\hL)$ and $\hB=(\hrho',\hL')$, an \emph{equivalence} $\hB\toco \hB'$ is a pair $\Psi=(e^{\tau},\{g\si\})$, where 
\begin{enumerate} \item $e^{\tau}\in e^{\g\da(Z)}$ satisfies 
\[\hrho'=e^{\tau}\hrho,\] and 
\item $\{g\si\}$ is a morphism 
\[\psi:\hL'\to e^{\tau}\cdot\hL\] in the groupoid $\Def^{\L}_A(Z)$. \end{enumerate}
\item Given equivalences $\Psi=(e^{\tau},\{g_I\}):\hB\toco\hB'$ and $\Psi'=(e^{\tau'},\{g'\si\}):\hB\toco \hB''$, their composition is defined by  
\[\Psi'\circ\Psi=(e^{\tau'}e^{\tau},\{e^{\tau'}g\si+g'\si\}):\hB\toco \hB''.\]   For any $A$-deformation $\hB$ of $\B$, the \emph{identity} isomorphism $\hB\to \hB$ is given by $(1,id_{\hL})$.  
\end{enumerate}
\end{definition}
\begin{rem} Let us examine the composition in $\Dt^{\B}_A(X)$ in a little more detail (and in particular verify that it is well-defined).  Suppose $\Psi=(e^{\tau},\{g\si\})$ is an equivalence from $\B=(\hrho,\hat{\L})$ to $\B'=(\hrho',\hL')$  and $\Psi'=(e^{\tau'}, \{g'\si\})$ is an equivalence from $\hB'$ to $\hB''=(\hrho'',\hL'')$.  By definition, we have $\hrho'=e^{\tau}\hrho$, $\hrho''=e^{\tau'}\hrho'$,  $\{g\si\}$ is an equivalence  $\hL'\to e^{\tau}\hL$ and $\{g'\si\}$ is an equivalence $\hL''\to e^{\tau'}\hL'$.  Defining $e^{\tau''}=e^{\tau'}e^{\tau}$, it follows that $\hrho''=e^{\tau''}\hrho$.  Furthermore, $e^{\tau'}\{g_I\}$ is an equivalence from $e^{\tau'}\hL\to e^{\tau'}e^{\tau}\L=e^{\tau''}\L$, so that the composition $(e^{\tau'}\{g\si\})\circ \{g'\si\})$ is an equivalence from $\hL''\to e^{\tau''}\cdot\hL$; therefore 
\begin{equation}\label{comp identity} (e^{\tau''},(e^{\tau'}\{g\si\})\circ\{g'\si\})=(e^{\tau'}e^{\tau},\{e^{\tau'}g_I+g'\si\})\end{equation} does in fact define a morphism in $\Dt^{\B}_A(X)$ from $\hB$ to $\hB''$.   Writing the composition as on the left-hand side of (\ref{comp identity}) makes it easy to check its associativity using Proposition \ref{bundle action}.
\end{rem}

%\bnote{For example, let $\hB=(\hrho,\hL)$ be an object of $\Dt_A^{\B}(X)$, and let $\tau:A\to A'$ be a morphism in $\Art$.  For each $u\in \Omega\ub(X)\subset \Omega\ub_{A'}(X)$, we define 
%\[\tau(\hrho)(u)=\tau(\hrho(u)),\] and then extend by $A'$-linearity to define 
%\[\tau(\hrho):\Omega\ub_{A'}(X)\to \Omega\ub_{A'}(Z)....\]}

\begin{proposition}\label{right action par} There is a strict right action of the formal group $e^{\cg(X)}$ on $\widetilde{\De}^{\B}(X)$, given as follows:
\begin{enumerate}\item  For each $A\in\Art$, each $g=e^{(0,w)}e^{(\xi,0)}\in e^{\cg\da(X)}$, and each $\hB=(\hrho,\hL)\in \widetilde{\De}^{\B}\da(X)$, define
\[\hB\cdot g=(\hrho e^{\xi}, e^{-\hrho(w)}\cdot\hL),\] where we recall that by definition \[e^{-\hrho(w)}\cdot\L=(\{\hat{c}\sij\},\{\hat{a}\si+\hrho(w)|_{W_I}\}).\] 
\item For each equivalence $\Psi=(e^{\tau},\{g_I\}):\hB_1\toco \hB_2$, define \[\Psi\cdot g:\hB_1\cdot g\to \hB_2\cdot g\] to be given by the same pair $(e^{\tau},\{g\si\})$, regarded as an equivalence $\hB\toco \hB'$.
\end{enumerate}
\end{proposition} 
\begin{proof} First, let us check that if $\Psi=(e^{\tau},\{g\si\})$ is an equivalence $\hB_1\toco\hB_2$, then the same pair $(e^{\tau},\{g\si\})$ does in fact define an equivalence  $\hB_1\cdot g\toco \hB_2\cdot g$.  Let us write $\hB_1=(\hrho_1,\hL_1), \hB_2=(\hrho_2,\hL_2)$, $\hB'_1:=\hB_1\cdot g=(\hrho'_1,\hL'_1)$, $\hB'_2:=\hB_2\cdot g=(\hrho'_2,\hL_2')$, and similarly write $\Psi=(e^{\tau},\psi)$ with $\psi=\{g_I\}:\hL_2\to\hL_2$, and $\Psi'=\Psi\cdot g = (e^{\tau'},\psi')$ with $\tau'=\tau$ and $\psi'=\{g\si\}$.  Note that since $\hrho_2=e^{\tau}\hrho_1$, we have $\hrho_2'=\hrho_2e^{\xi}=e^{\tau}\hrho_1e^{\xi}=e^{\tau}\hrho'_1=e^{\tau'}\hrho'_1$.  Furthermore, since $\psi$ is an equivalence from $\hL_2\to e^{\tau}\hL_1$ in $\Def^{\L}_A(Z)$, it follows that $e^{-\hrho_2(u)}\cdot\psi$ is an equivalence from $\L_2'=e^{-\hrho_2(u)}\cdot\hL_2$ to $e^{-\hrho_2(u)}e^{\tau}\hL_1$.  But we have 
\[e^{-\hrho_2(u)}e^{\tau}\hL_1=e^{-e^{\tau}\hrho_1(u)}e^{\tau}\hL_1=e^{\tau}e^{-\hrho_1(u)}\hL_1=e^{\tau'}\hL'_1,\] so that $e^{-\hrho_2(u)}\psi$ is a morphism from $\hL'_2\to e^{\tau'}\hL'_1$ in $\Def^{\L}\da(Z)$.  Thus, we see that $\Psi\cdot g= \Psi'$ does indeed define an equivalence from $\hB_1\cdot g$ to $\hB_2\cdot g$, as claimed.

Since $g$ acts trivially on morphisms, the functoriality property $(\Psi'\cdot g)\circ(\Psi\cdot g)=(\Psi'\circ \Psi)\cdot g$ holds trivially.  Furthermore, by construction, the identity element of $e^{\cg(X)}$ acts as the identity functor.  Therefore, to finish the proof of the proposition we just need to check that $(\hB\cdot g)\cdot g'=\hB\cdot (gg')$ hold for each object $\hB\in\Dt^{\B}_A(X)$ and each $g,g'\in e^{\g\da(X)}$.

Let $\hB=(\hrho,\hL)\in\Dt^{\B}\da(X)$ and $g=e^ue^{\xi}$, $g'=e^{u'}e^{\xi'}\in e^{\cg\da(X)}$.  We have 
\begin{equation}\label{ggprime}(\hB\cdot g)\cdot g'=(\hrho e^{\xi},e^{-\hrho(u)}\hL)g'=(\hrho e^{\xi}e^{\xi'},e^{-\hrho e^{\xi}u'-\hrho u}\hL).\end{equation}
 Since $gg'=e^{u+e^{\xi}u'}e^{\xi}e^{\xi'}$, we see that the right-hand side of (\ref{ggprime}) is equal to $\hB\cdot (gg')$.
 \end{proof} 
 
 \subsection{Compatibility with a generalized complex structure}
 
Every $\B=(Z,\L)\in\Br(X)$  determines a generalized submanifold $(Z,F)$ on $X$, where $F$ is the curvature form of $\L$.  In particular, we have the \emph{generalized tangent bundle} $\TT\B$ and the \emph{generalized normal bundle} $\N\B$, which by definition are  the generalized tangent bundle and generalized normal bundle of $(Z,F)$ (as defined in \S\ref{GC submanifolds}). Similarly, we denote by $K\uB$ the space $K^{(Z,F)}$ introduced in Definition \ref{definition K}, and by $H\ub(\B)$  the Lie algebroid cohomology groups associated to $(Z,F)$ (described in \S\ref{Lie alg cohomology}).  Recall also the notation 
\[I^Z:=\{f\in\Cinf(X):f|_Z=0\}.\]

 \begin{definition}\label{GC brane definition 1}  Let $(X,\J)$ be a GC manifold.  A (rank 1) \emph{GC brane} on $(X,\J)$ is an element $\B\in\Br(X)$ as in Definition \ref{brane definition 1}, such that the underlying generalized submanifold $(Z,F)$ of $\B$ is compatible with $\J$.  We denote the collection of all such GC branes by $\Br(X,\J)$.
\end{definition}

\begin{definition} \label{compatibility1} Given $\B=(Z,\L)\in\Br(X,\J)$, and an Artin algebra $A\in\Art$, an element $g\in e^{\cg(X)}$ is said to be\emph{compatible} with $\J$ with respect to $\B$ if 
\[Q_{g\cdot\J}(K_A^{\B},K_A^{\B})\subset I^Z\da,\] where $Q_{g\cdot\J}$ is the $A$-linear extension of pairing defined by (\ref{Q def}), $K\uB_A:=A\otimes_{\R}K\uB$ and $I^Z\da:=A\otimes_{\R}I^Z$. 
\end{definition}

\begin{proposition}  \label{well-defined} Let $\B\in\Br(X,\J)$, and let $g=e^ue^{\xi}$ and $g'=e^{u'}e^{\xi'}$ be elements of $e^{\cg\da(X)}$ such that 
\begin{enumerate}\item $\rho e^{\xi}=\rho e^{\xi'}$ and 
\item $\rho u=\rho u'$.
\end{enumerate}
Then $g$ is compatible with $\J$ with respect to $\B$ if and only if $g'$ is.
\end{proposition}
\begin{proof}  Suppose we are given $\xi\in\Cinf\dm(TX)$ such that $\rho e^{\xi}=\rho:\Omega\ub\da(X)\to\Omega\ub\da(Z)$; this is equivalent to requiring that $0=\rho\pounds(\xi):\Omega\da\ub(X)\to\Omega\da\ub(Z)$.  Suppose also that we are given $\eta\in\Cinf\da(TX)$ that is tangent to $Z$, with $\rho(\eta)=\tau\in\Cinf\da(Z)$; this means that $\rho(\pounds(\xi)v)=\pounds(\tau)\rho v$ for every $v\in \Omega\ub\da(X)$.    We claim that in this case 
\[e^{\xi}\eta\in\Cinf\da(TX)\] is still tangent to $Z$ with $\rho(e^{\xi}\eta)=\tau$.  To see this, note that for every $v\in\Omega\ub\da(X)$ we have
\begin{align*} \rho(\pounds(e^{\xi}\eta)v) & = \rho(e^{\xi}(\pounds(\eta)e^{-\xi}v)) \\
& = \rho(\pounds(\eta)e^{-\xi}v) \\
& = \pounds(\tau)\rho e^{-\xi}v \\
& = \pounds(\tau)\rho v.
\end{align*}

Next, suppose we are given $h=e^{(0,u)}e^{(\xi,0)}\in e^{\cg\da(X)}$ such that $\rho e^{\xi}=\rho$ and $\rho u=0$.   We claim that in this case 
\[h\cdot K\da^{\B}=K\da^{\B}.\]  To see this, let $\xx=(\eta,a)$ be an arbitrary element of $K\da^{\B}$.  By definition, this means that $\eta$ is tangent to $Z$, say with $\rho(\eta)=\tau\in\Cinf\da(TZ)$, and $\rho(a)=\iota(\tau)F$.  Write $h\cdot \xx=(\tilde{\eta},\tilde{a})$, with 
\[\tilde{\eta}=e^{\xi}\eta\] and \[\tilde{a}=e^{\xi}a-\iota(\tilde{\eta})du.\]  We have already seen that $\tilde{\eta}$ is tangent to $Z$ with $\rho(\tilde{\eta})=\tau$.  Furthermore,   we see that \begin{align*} \rho\tilde{a} & = \rho(e^{\xi}a)-\rho(\iota(\tilde{\eta})du) \\
& = \rho a-\iota(\tau)d\rho(u) \\
& = \rho a = \iota(\tau)F,
\end{align*} so that $h\cdot \xx\in\KK\da^{\B}$, as claimed.  

To complete the proof, let $g=e^{u}e^{\xi}$ and $g'=e^{u'}e^{\xi'}$ with $\rho e^{\xi}=\rho e^{\xi'}$ and $\rho u=\rho u'$.  Defining $h=g'g^{-1}$, we see that 
\[h=e^{u'-e^{\xi'}e^{-\xi}u}e^{\xi'}e^{-\xi}.\]  Defining $\tilde{\xi}$ by the equation
\[e^{\tilde{\xi}}=e^{\xi'}e^{-\xi}\] and $\tilde{u}=u'-e^{\tilde{\xi}}u$, we see that $h=e^{\tilde{u}}e^{\tilde{\xi}}$, and that 
\[\rho e^{\tilde{\xi}}=\rho\] and \[\rho \tilde{u}=0;\] as shown above this implies 
\[h\cdot K\da^{\B}=K\da^{\B}.\]  Therefore, we have
\begin{align}\label{Q equal} Q_{g\cdot\J}(K\da^{\B},K\da^{\B}) & = \la g\J g^{-1} K\da^{\B},K\da^{\B}\ra \notag\\
& = \la h^{-1} g'\J (g')^{-1}hK\da^{\B},K\da^{\B}\ra \notag\\
& = h^{-1}\la (g'\cdot\J)hK\da^{\B},hK\da^{\B}\ra \notag\\
& = h^{-1}\la (g'\cdot\J)K\da^{\B},K\da^{\B}\ra\notag\\
& = h^{-1} Q_{g'\cdot\J}(K\da^{\B},K\da^{\B}).
\end{align}
Since $\rho h=\rho$, in particular both $\tilde{\xi}$ and $-\tilde{\xi}$ are tangent to $Z$, so that it follows that $h(I\da^Z)=h^{-1}(I\da^Z)=I\da^Z$.  Combining this with the equality (\ref{Q equal}), we obtain the desired result.
\end{proof}

With Proposition \ref{well-defined} in hand, we are ready to define what it means for a deformation of GC brane to be compatible with the GC structure.  Let $\B\in\Br(X,\J)$, and let $\hB=(\hrho,\hL)$ be an $A$-deformation of $\B$ as defined in Definition \ref{definition par def}. Choose $\xi\in \cg\da(X)$ such that $\hrho=\rho e^{\xi}$.  Also, for each $I$, let $\tilde{W}_I\subset X$ be an open set such that $\tilde{W}_I\cap Z=W_I,$ and choose $u_I\in \Omega^1(\tilde{W}_I)$ such that $\hat{a}_I=a_I+\rho(u_I)$.  Set $\xi_I=\xi|_{\tilde{W}_I}$, and define $g_I=e^{u_I}e^{\xi_I}\in e^{\cg\da(\tilde{W}_I)}.$
 \begin{definition}\label{compatibility def}  We say that $\hB\in\Dt\da\uB(X,\J)$ is \emph{compatible} with $\J$ if each $g_I$ (chosen as above) is compatible with $\J|_{\tilde{W}_I}$ with respect to $\B|_{\tilde{W}_I}$.  We denote by $\Dt^{\B}(X,\J)$ the formal subgroupoid of $\Dt^{\B}\d(X)$, whose objects for each $A\in\Art$ are the $A$-deformations compatible with $\J$.
 \end{definition}
 \begin{rem} It is clear from Proposition \ref{well-defined} that the condition in Definition \ref{compatibility def} is well-defined, i.e. does not  depend on the particular choices of $\xi$, $\{\tilde{W}_I\}$ or $\{u_I\}$.  
 \end{rem} 
 
 \begin{rem}\label{brane def sheaf} Similarly to the situation described in Remark \ref{bundle def sheaf}, the formal groupoid $\Dt\uB(X,\J)$ extends in a natural way to a presheaf of formal groupoids on $X$.  The fact that this is actually a sheaf (satisfies the descent condition) is Proposition \ref{descent property}.
\end{rem}
%\begin{rem}\label{brane def sheaf} Building on Remark \ref{bundle def sheaf}, we may refine Definition \ref{compatibility def} to define a (strict) presheaf of groupoids on $X$ \bnote{more precisely, preseheaf on $X$ valued-in the 2-category $Gpd_{\Art}$}, which sends an open subset $U\subset X$ to the groupoid 
%\[\Dt\da^{\B}(U,\J):=\Dt\da^{\B|_U}(U,\J|_U),\] where $\B|_U$ is the brane on $U$ given by $(Z\cap U,\L|_U)\in\Br(U,\J|_U)$.  In fact, as we will see in Proposition \ref{descent property}, 
% the presheaf $U\mapsto \Dt\da^{\B}(U,\J)$ satisfies the descent property, i.e. defines a \emph{sheaf} of groupoids on $X$.  Similar remarks apply to Definition \ref{full deformation groupoid} below. 
  %\end{rem}

 \begin{lemma}\label{compatibility under isomorphism} Let $\hB,\hB'$ be isomorphic elements of $\Dt^{\B}_A(X)$. Then $\hB$ is compatible with $\J$ if and only if $\hB'$ is compatible with $\J$. 
 \end{lemma}
 
\begin{proof}
 Write $\hB=(\hrho,\hL)$ and $\hB'=(\hrho',\hL')$ (with the same underlying submanifold $Z\subset X$) with $\hL=(\{\hc_{IJ}\},\{\ha_{IJ}\})$ and 
 $\hL'=(\{\hc'_{IJ}\},\{\ha'_{IJ}\})$, and let $\Psi=(e^{\tau},\{g\si\})$ be an isomorphism $\hB\to\hB'$.  As in Definition \ref{compatibility def}, let $\{\tilde{W}_I\}$ be open sets in $X$ with $W_I=Z\cap \tilde{W}\si$, and choose $\xi\in\g\da(X)$ and $\{u_I\in\Omega^1\dm(\tilde{W}_I)\}$ such that $\hrho=\rho e^{\xi}$ and $\ha_I=a_I+\rho(u_I)$.  Define $x_I=e^{u\si}e^{\xi\si}$, where $\xi\si$ is the restriction of $\xi$ to $\tilde{W}\si$. According to Definition \ref{compatibility def}, $\hB$ is compatible with $\J$ if and only if, for each $I$ we have 
 \begin{equation}\label{compat1}\la (x_I\cdot\J)K^{\B}\da,K^{\B}\da\ra\subset I^{Z}\da,\end{equation} where $K\uB\da$ and $I^Z\da$ are defined here with respect to $\tilde{W}\si$. Let $\eta\in\g\da(X)$ be an extension of $\tau$, and define $\xi'\in\g\da(X)$ by $e^{\eta}e^{\xi}=e^{\xi'}$.  Then 
 \[\hrho'=e^{\tau}\hrho=e^{\tau}\rho e^{\xi}=\rho e^{\eta}e^{\xi}=\rho e^{\xi'}.\] Furthermore, we have 
 \[dg_I=e^{\tau}\ha\si-\ha'\si=e^{\tau}a_I+e^{\tau}\rho(u\si)-\ha'\si\] which implies that 
 \[\ha'\si=a\si+(e^{\tau}a\si-a\si+e^{\tau}\rho(u\si)-dg_I).\]  Choose $\tilde{a}\si\in\Omega\dm^1(\tilde{W}\si)$ such that $\rho(\tilde{a}\si)=a\si$, and $\tilde{g}\si\in\Cinf\dm(\tilde{W}\si)$ such that $\rho(\tilde{g}\si)=g\si$.  Then 
 \[e^{\tau}a\si-a\si-e^{\tau}\rho(u\si)-dg_I)=\rho(e^{\eta}\tilde{a}\si-\tilde{a}\si+e^{\eta}u\si-d\tilde{g}\si).\]  Set 
 \[u'\si=e^{\eta}\tilde{a}\si-\tilde{a}\si+e^{\eta}u\si-d\tilde{g}\si,\] and $x'\si=e^{u'\si}e^{\xi'\si}$, where $\xi'\si$ is the restriction of $\xi'$ to $\tilde{W}\si$. Then $\hB'$ is compatible with $\J$ if and only if for each $I$ we have  
  \begin{equation}\label{compat2}\la (x'\si\cdot \J)K^{\B}\da,K^{\B}\da\ra\subset I^{Z}\da.\end{equation}  We calculate 
 \begin{align*} x'\si x^{-1}\si & =e^{e^{\eta}\tilde{a}\si-\tilde{a}\si+e^{\eta}u\si-d\tilde{g}\si}e^{\eta}e^{\xi} e^{-\xi}e^{-u\si} \\
 & = e^{e^{\eta}\tilde{a}\si-\tilde{a}\si-d\tilde{g}\si}e^{\eta}
 \end{align*}
 
   We have 
 \[e^{\eta}\tilde{a}_I-\tilde{a}_I=d\int_0^1e^{t\eta}\iota(\eta)\tilde{a}_Idt+\int_0^1e^{t\eta}\iota(\eta)d\tilde{a}\si dt.\]  Defining $h\si=\int_0^1e^{t\eta}\iota(\eta)\tilde{a}_Idt-\tilde{g}\si$ we have 
 \[x\si'x^{-1}\si=e^{dh\si}e^{(\eta,\iota(\eta)d\tilde{a}_I)}.\]  Note that $\xx:=(\eta,\iota(\eta)d\tilde{a}\si)$ is an element of $\KK^{\B}\da$, so that by Lemma \ref{K closed under bracket} we have 
 \[e^{-\xx}\KK^{\B}\da=\KK^{\B}\da.\] Using the fact that $e^{dh\si}$ acts trivially on $\Cinf\da(\TTX)$, and writing $\KK:=\KK^{\B}\da$, we see that 
 \begin{align*} \la (x'\si\cdot \J)\KK,\KK\ra & = \la (e^{\xx}x\si\cdot \J)\KK,\KK\ra \\
 & = \la e^{\xx}(x\si\cdot\J)e^{-\xx}\KK,\KK\ra \\
 & = e^{\eta}\la (x\si\cdot\J)e^{-\xx}\KK,e^{-\xx}\KK\ra\\
 & = e^{\eta}\la (x\si\cdot\J)\KK,\KK\ra. 
  \end{align*} Since $\eta$ is tangent to $Z$, it follows that 
 \[e^{\eta}I^{Z}\da=I^{Z}\da,\] so we see that  condition (\ref{compat1}) is indeed equivalent to condition (\ref{compat2}).
 \end{proof}

\begin{proposition}\label{action of formal symmetries} Given $\hB\in\Dt\da^{\B}(X)$ and $g\in e^{\T\da(X)}$, the deformation $\hB\cdot g$ is compatible with $\J$ if and only if $\hB$ is.  In particular, the action of $e^{\cg(X)}$ on $\Dt^{\B}(X)$ restricts to give a well-defined right action of the formal subgroup $e^{\T(X)}$ on $\Dt^{\B}(X,\J)$.
 \end{proposition} 
\begin{proof} Given $\hB=(\hrho,\hL)\in\Dt\da^{\B}(X)$ compatible with $\J$, choose $\{\tilde{W}_I\}$, $\{u\si\}$, and $\xi$, as in Definition \ref{compatibility def}, and set $\xi\si=\xi|_{\tilde{W}_I}$.  By definition, for each $I$ we have 
\begin{equation}\label{compat1}\la (e^{u\si}e^{\xi\si}\J)K^{\B}\da,\KK^{\B}_{\da}\ra\subset I^Z\da.\end{equation}  Given $g=e^{w}e^{\eta}\in e^{\T\da(X)}$, write $\hB':=\hB\cdot g=(\hrho',\hL')$.  Then we have $\hrho'=\hrho e^{\eta}=\rho e^{\xi}e^{\eta}$ and $\hL'=(\{\hc_{IJ}\},\{a\si+\rho u\si+(\rho e^{\xi}w)|_{W\si}\})=(\{\hc_{IJ}\},\{a\si+\rho u'\si\})$ with $u'\si=u\si+e^{\xi\si} w|_{\tilde{W}_I}$.  Defining $\xi'_I$ by $e^{\xi'_I}=e^{\xi_I}e^{\eta|_{\tilde{W}_I}}$, we see that the compatibility of $\hB'$ with $\J$ is equivalent to the conditions
\begin{equation}\label{compat2}\la (e^{u'\si}e^{\xi'\si}\J)\KK^{\B}\da,\KK^{\B}_{\da}\ra\subset I^Z\da.\end{equation}  Letting $g_I=g_{\tilde{W}_I}=e^{w|_{\tilde{W}_I}}e^{\eta|_{\tilde{W}_I}}$, we see that 
\[e^{u'\si}e^{\xi'\si}=e^{u\si}e^{\xi\si} g_I.\]  Since by assumption $g$ is a symmetry of the GC structure $\J$, it follows that the left-hand sides of (\ref{compat1}) and (\ref{compat2}) are equal.

\end{proof}

%\begin{rem} Since $\hrho$ is of the form $\rho e^{\xi}$ for some $\xi\in \g\da(X)$, we see that $\hrho\otimes_{A}\R=\rho$.  Geometrically, we view $\hrho$ as determining a family $\hat{i}:Z\hookrightarrow X$ of embeddings deforming the inclusion $i:Z\hookrightarrow X$.  We then view the pair $(\hrho,u)$ as family of branes $\hB=(\hat{Z},\hat{\L},\hat{\nabla})$ deforming $\B$, where $\hat{Z}=\hat{i}(Z)$, and $(\hat{\L},\hat{\nabla})$ is obtained by adding $u$ to the connection $\nabla$, and then pulling back the bundle with connection to $\hat{Z}$.  \bnote{maybe relate to an earlier discussion about ``honest" deformations.}
%\end{rem}
%\begin{rem} Given a pair $(\hrho,u)$ as in Definition \ref{definition par def}, suppose that $(w,\xi),(w',\xi')\in \cg\da(X)$ satisfy $\rho=\rho e^{\xi}=\rho e^{\xi'}$ and $u=\rho(w)=\rho(w')$.  Then it is easy to check that $e^{w}e^{\xi}\cdot Q_{\J}$ satisfies the condition in part (2) of the definition if and only if it is satisfied by $e^{w'}e^{\xi'}\cdot Q_{\J}$.  Thus, it is sufficient to check condition (2) for an a single choice of $(\xi,w)$.
%\end{rem}

%\begin{definition} \begin{enumerate} Given $\hB,\hB'\in \widetilde{\De}_{\B}(A)$, a \emph{Hamiltonian equivalence} from $\hB$ to $\hB'$ is a pair $(z,\Psi)$, where $z\in e^{\H\da(X)}\subset e^{\T\da(X)}$, and 
%\[\Psi:\hB\to \hB'\] is an isomorphism.
 Recall from Definition \ref{genhamdef} and Proposition \ref{hamiltonian} the subalgebra $\H(X)\subset \T(X)$ of generalized Hamiltonian vector fields on $(X,\J)$. We correspondingly have the formal subgroup$e^{\H(X)}\subset e^{\T(X)}$ of formal generalized Hamiltonian symmetries.\begin{definition}\label{full deformation groupoid} Given $\B\in\Br(X,\J)$, let $\De^{\B}(X,\J)$ be the following formal groupoid (over $\Art$).  For each $A\in\Art$, the groupoid $\De_A^{\B}(X,\J)$ has the same objects as $\widetilde{\De}^{\B}_A(X,\J)$.  Given objects $\hB$ and $\hB'$, a morphism in $\De^{\B}_A(X,\J)$ from $\hB$ to $\hB'$ is a pair $(\psi,z)$, where $z\in e^{\H\da(X)}$ and $\psi$ is a morphism in $\widetilde{\De}^{\B}_A(X,\J)$ from $\hB$ to $\hB'\cdot z$.  Given $(\psi,z):\hB\to\hB'$ and $(\psi',z'):\hB'\to\hB''$, the composition is defined by 
\begin{equation}\label{composition}(\psi',z')\circ (\psi,z)=(\psi'\psi,z'z):\hB\to \hB''.\end{equation}  The identity morphism of any $\hB$ is the pair $(id_{\hB},1)$, where $id_{\hB}$ denotes the identity morphism of $\hB$ considered as an object in $\widetilde{\De}^{\B}_A(X,\J)$.
\end{definition}

\begin{rem} One easily checks using Proposition \ref{right action par} that the composition given by (\ref{composition}) is well-defined, unital, and associative.
\end{rem}

\begin{definition} \label{brane deformation functor} The \emph{deformation functor} 
\[\df_{\B}:\Art\to \textrm{Set}\] of $\B\in\Br(X,\J) $ is given by 
\[A\mapsto \pi_0(\De_A^{\B}(X,\J)).\]
\end{definition}

\section{First order deformations and Lie algebroid cohomology}\label{first order deformations}  In \cite{KM}, an argument was given that first-order deformations of a GC brane $\B$ should correspond to elements of the Lie algebroid cohomology group $H^1(\B)$.  In our framework,  first-order deformations of $\B$ are encoded as elements of $\df_{\B}(\R[\epsilon](\epsilon^2))$; therefore, a natural expectation motivated by \cite{KM} is that elements of $\D_{\B}( \R[\epsilon](\epsilon^2))$ should correspond to classes in $H^1(\B)$.  In this section, we verify this by unpacking Definition \ref{brane deformation functor} in the special case $A=\DN$.  %\begin{theorem}\label{brane tangent} For every GC brane $\B$, there is a natural bijection between elements of $\De_{\B}(\R[\epsilon]/(\epsilon^2))$ and elements of the Lie algebroid cohomology group $H^1(\B)$.
%\end{theorem}
On the hand hand, this result (stated as Theorem \ref{brane tangent}  in the introduction) can be regarded as a rigorous verification of the result obtain in \cite{KM}.  Going in the other direction, we may view Theorem \ref{brane tangent} as a check that our Definition \ref{brane deformation functor} is a reasonable one from a geometric point of view.

We will actually prove Theorem \ref{brane tangent} as part of a slightly better statement, given below as Proposition \ref{first order induced}.  Namely, we will construct an explicit map $H^1(\B)\to \df_{\B}(\DN)$, and prove that it is both well-defined and bijective.

We now turn to the construction.  Let $\B=(Z,\L)\in\Br(X,\J)$ be a GC brane, with generalized tangent bundle $\TT\B$.  Recall that, since $\B$ is compatible with $\J$, the GC structure $\J$ preserves $\TT\B\subset \TT X|_{Z}$, and we define $l\subset \TT\B\otimes\C$ to be the $+i$-eigenbundle, which is a complex sub-bundle of $L|_{Z}$.   Recall also that the generalized normal bundle $\N\B$ is defined as the quotient of $\TT X|_Z$ by $\TT\B$ (with quotient map $q:\TTX|_Z\to \N\B$). Since $\TT\B$ is a maximal isotropic subbundle of $\TT X|_Z$, the pairing on $\TT X|_Z$ gives a well-defined non-degenerate pairing of $\TT\B$ with $\N \B$. Furthermore, since $\J$ preserves $\TT\B$, it determines a well-defined complex structure $\J_{\N\B}$ on $\N\B$; regarding $\N\B$ as a complex vector bundle with complex structure $-\J_{\N\B}$, we have an isomorphism of complex vector bundles
\[\mu:\N\B\to l\uv\] given by 
\[\mu(q(\xx))(v)=\la\xx,v\ra\] for every section $\xx$ of $\TT X|_Z$. %Recall from \rnote{ref} that we defined an isomorphism (also called $\mu$) from $\TT X$ to $\L\uv$; it is easy to see that the following diagram commutes

Given $\xx=(\xi,u)\in \Cinf(\TT X|_Z)$, choose a section $\tilde{\xx}=(\tilde{\xi},\tilde{u})\in \Cinf(\TT X)$ extending $\xx$, i.e. such that $r(\tilde{\xx})=\xx$.  As described in Proposition \ref{right action par}, the generalized vector field $\tilde{\xx}$ determines an object of $\De^{\B}_{\DN}(X)$ given by  
\[\B\cdot e^{\epsilon\tilde{\xx}}=(\rho e^{\epsilon\tilde{\xi}},e^{-\epsilon\rho(\tilde{u})}\cdot\L)\in \De_{\DN}^{\B}(X,\J).\]
\begin{lemma} The deformation $\B\cdot e^{\epsilon\tilde{\xx}}$ depends only on $\xx$, i.e. not on the choice of the extension $\tilde{\xx}$.
\end{lemma}
\begin{proof}  Suppose $\eta\in\Cinf(TX)$ satisfies $r(\eta)=0$, i.e. $\eta$ vanishes at each point of $Z$.  Then for every $w\in\Omega\ub(X)$, we have 
\[\rho(\iota(\eta)w)=0.\]  This implies that, for every $v\in\Omega\ub(X)$, we have 
\[\rho(\pounds(\eta)v)=d\rho(\iota(\eta)v)+\rho(\iota(\eta)dv)=0.\]  Therefore, given two different extensions $\tilde{\xi},\tilde{\xi}'\in\Cinf(TX)$ of $\xi$, we have 
\[\rho\pounds(\tilde{\xi})=\rho\pounds(\tilde{\xi}'):\Omega\ub(X)\to\Omega\ub(Z).\]  This in turn implies that $\rho e^{\epsilon\tilde{\xi}}=\rho+\epsilon\rho\pounds(\tilde{\xi})$ is equal to $\rho e^{\epsilon\tilde{\xi}'}=\rho+\epsilon\rho\pounds(\tilde{\xi}')$.
It is also clear that $\rho(\tilde{u})$ depends only on $r(\tilde{u})$, so that $\rho(\tilde{u})$ is independent of the choice of extension. \end{proof}

Since the deformation depends only on $\xx$, we will use the notation $\hB^{\xx}=\B\cdot e^{\epsilon\tilde{\xx}}$.  Note that the section $\xx$ also determines a section $q(\xx)\in\Cinf(\N\B)$ as well as a section $\mu q(\xx)\in\Cinf(l\uv)$.
 
\begin{proposition} \label{first order induced}
\begin{enumerate}[(1)]
\item  The deformation $\hB^{\xx}$ is compatible with $\J$ if and only if $\delta_l(\mu q(\xx))=0$, i.e. if and only if $q(\xx)$ is a generalized holomorphic section of $\N\B$.

\item Given $\xx,\xx'\in\Cinf(\TTX|_Z)$ such that both $q(\xx)$ and $q(\xx')$ are generalized holomorphic sections of $\N\B$, the deformations $\B^{\xx}$ and $\B^{\xx'}$ are isomorphic in $\De^{\B}_{\DN}(X,\J)$ if and only if the elements $[\mu q(\xx)],[\mu q(\xx')]\in H^1(\B)$ are equal.  
In other words, there is a well-defined injective map $H^1(\B)\to \df_{\B}(\DN)$ mapping $[\mu q(\xx)]$ to the equivalence class of $\B^{\xx}$.
\item Every element of $\De^{\B}_{\DN}(X,\J)$ is isomorphic to one of the form $\B^{\xx}$ for some $\xx\in \Cinf(\TT X|_{Z})$.  In other words, the injective map $H^1(\B)\toco \df_{\B}(\DN)$ is actually a bijection.  \end{enumerate} 
\end{proposition} 
\begin{proof}  Let us begin with part (1).  By Definition \ref{compatibility def}, we see that $\B^{\xx}$ will be compatible with $\J$ if and only if $e^{\epsilon\tilde{\xx}}$ is compatible with $\J$ in the sense of Definition \ref{compatibility1}, i.e. if and only if we have 
\[\la (e^{\epsilon\tilde{\xx}}\cdot\J)K^{\B}\da,K^{\B}\da\ra \subset I^Z\da.\]  Since $e^{\epsilon\tilde{\xx}}\cdot\J=\J+\epsilon\tilde{\xx}\cdot\J$, this is equivalent to \[\la (\tilde{\xx}\cdot\J)K^{\B}\da,K^{\B}\da\ra\subset I^Z\da.\] By definition, this means that for every $A,B\in K\da^{\B}$ we have 
\begin{equation}\label{compatibility cond1}\rho(\la \ll\tilde{\xx},\J A\rr-\J\ll\tilde{\xx},A\rr,B\ra =\rho(\la\ll\tilde{\xx},\J A\rr,B\ra+\la \ll\tilde{\xx},A\rr,\J B\ra)=0.\end{equation}

 Suppose condition (\ref{compatibility cond1}) holds, and let $u,v\in\Cinf(l)$ be arbitrary sections.  Choose sections $\tilde{u},\tilde{v}\in\Cinf(L)$ which extend $u,v$, i.e. such that $r(\tilde{u})=u$ and $r(\tilde{v})=v$.  Since $l\subset \TT\B\otimes \C$, condition (\ref{compatibility cond1}) implies that 
 \begin{equation}\label{compatibility condit2}0=\rho(\la\ll\tilde{\xx},\J \tilde{u}\rr,\tilde{v}\ra+\la \ll\tilde{\xx},\tilde{u}\rr,\J \tilde{v}\ra)=2i\rho(\la\ll\tilde{\xx},\tilde{u}\rr,\tilde{v}\ra).\end{equation}  Making use of the identities (\ref{Dorfman1}) and (\ref{Dorfman2}) satisfied by the Dorfman bracket, we calculate the the right-hand side of (\ref{compatibility condit2}) is equal to 
 \begin{align}\label{compatibility condit3} & 2i\rho(\pi(\tilde{v})\cdot\la\tilde{\xx},\tilde{u}\ra-\pi(\tilde{u})\la\tilde{\xx},\tilde{v}\ra-\la\tilde{\xx},\ll\tilde{u},\tilde{v}\rr) \notag\\  
 =  & 2i(\pi(v)\cdot\la\xx,u\ra-\pi(u)\cdot\la\xx,v\ra-\la\xx,\ll u,v\rr_{\B}) \notag \\
 = & 2i\delta_l(\mu q(\xx)(v,u)),
 \end{align}  where in the second line the expression $\la\xx,v\ra-\la\xx,\ll u,v\rr_{\B}$ denotes the Lie algebroid bracket of the sections $u,b\in\Cinf(l)$, as defined in (\ref{brane bracket}).  Thus, we see that if $\B^{\xx}$ is compatible with $\J$, then $q(\xx)$ is a generalized holomorphic section of $\N\B$.
 
 Conversely, suppose we have $\delta_l(\mu q(\xx))=0$.  Given $u,v\in K^{\B}\da$, define $A,B\in \Cinf(L)$ by $A=u-i\J u$ and $B=v-i\J v$, and note that $r(A),r(B)\in\Cinf(l)$.   Using the above calculation (\ref{compatibility condit3}), we see that 
 \begin{align*} 0 = & \rho(\la\ll\tilde{\xx}, A\rr,B\ra) \\
 = & \rho(\la\ll\tilde{\xx},u-i\J u\rr,v-i\J v\ra) \\
 = & \rho(\la\ll\tilde{\xx},u\rr,v\rr+\la\ll\tilde{\xx},\J u\rr,\J v\ra)+i\rho(\la \ll\tilde{\xx},\J u\rr,v\ra+\la\ll\tilde{\xx},u\rr,\J v\ra).
 \end{align*}  In particular, this implies 
 \[\rho(\la \ll\tilde{\xx},\J u\rr,v\ra+\la\ll\tilde{\xx},u\rr,\J v\ra)=0\] holds for arbitrary $u,v\in K^{\B}\da$; as previously mentioned, this condition is equivalent to the compatibility of $\B^{\xx}$ with $\J$.  This completes the proof of part (1) of Proposition \ref{first order induced}.

We will prove part (2) of Proposition \ref{first order induced} via a series of lemmas.
\begin{lemma}\label{coboundary isomorphism}
Suppose we are given $\xx,\xx'\in \Cinf(\TTX|_Z)$ such that $q(\xx)=q(\xx')$.  Then there exists an isomorphism  $\hB^{\xx}\toco\hB^{\xx'}$ in $\Dt^{\B}_{\DN}(X)$.
\end{lemma}
\begin{proof}  Let us write $\B=(Z,\L)$, with $\L=(\{c\sij\},\{a\si\})$, and also $\xx=(\xi,u)$ and $\xx'=(\xi',u')$.  Also, let $\tilde{\xx}=(\tilde{\xi},\tilde{u})\in\Cinf(\TTX)$ be a choice of extension of $\xx$ and $\tilde{\xx}'=(\tilde{\xi}',\tilde{u}')\in\Cinf(\TTX)$ be a choice of extension of $\xx'$.   By definition, the assumption $q(\xx)=q(\xx')$ implies that $\xx'-\xx$ is a section of $\TT\B$.  Thus, we have that $\xi'-\xi$ is tangent to $Z$, and 
\[\rho(u'-u)=\iota(\tau)F,\] where $\tau=\rho(\xi'-\xi)\in\Cinf(TX)$ and $F$ is the curvature form of $\L$.  Defining $g=\{g_I=\epsilon \iota(\tau)a_I\}$, we claim that 
\[\Psi=(e^{\epsilon\tau},g)\] is an isomorphism $\hB^{\xx}\toco\hB^{\xx'}$ in $\De^{\B}_{\DN}(X)$.  Recall that 
\[\hB^{\xx}=(\rho e^{\epsilon\tilde{\xi}},(\{c_{IJ}\},\{a_I+\epsilon(\rho\tilde{u})|_{W_I}\})\] and 
\[\hB^{\xx}=(\rho e^{\epsilon\tilde{\xi}'},(\{c_{IJ}\},\{a_I+\epsilon(\rho\tilde{u'})|_{W_I}\}).\]   According to Definition \ref{definition par def}, we need to check three conditions to show that $\Psi$ does indeed define an isomorphism from $\hB^{\xx}\to\hB^{\xx'}$:
\begin{enumerate} \item $\rho e^{\epsilon\tilde{\xi}'}=e^{\epsilon\tau}\rho e^{\epsilon\tilde{\xi}}$, %which is equivalent to $\rho(\xi'-\xi)=\tau\rho$.
\item $g_J-g_I=e^{\epsilon\tau}c_{IJ}-c_{IJ}$%=\epsilon\pounds(\tau)c_{IJ}$, and 
\item $dg_I=e^{\epsilon\tau}(a_I+\epsilon\rho(\tilde{u})|_{W_I})-(a_I+\epsilon\rho(\tilde{u}')|_{W_I}),$ %or equivalently 
%\[dg_I=\epsilon(\pounds(\tau)a_I+\rho(u)|_{W_I}-\rho(u')|_{W_I}).\]
\end{enumerate}

Condition (1) is equivalent to $\rho(\xi'-\xi)=\tau\rho$, which holds by construction.  To check condition (2), first note that, since we are working over $A=\DN$, we have 
\[e^{\epsilon\tau}c_{IJ}-c_{IJ}=\epsilon\pounds(\tau)c_{IJ}.\]  On the other hand, we calculate that
\[g_J-g_I=\epsilon\iota(\tau)(a_J-a_I)=\epsilon \iota(\tau)dc_{IJ}=\epsilon\pounds(\tau)c\sij.\]  To check condition (3), first note that it is equivalent to
\[dg_I=\epsilon(\pounds(\tau)a_I+\rho(u)|_{W_I}-\rho(u')|_{W_I}).\]   We then  calculate 
\begin{align*} dg_I=\epsilon d(\iota(\tau)a_I) &= \epsilon(\pounds(\tau)a_I-\iota(\tau)da_I)\\
& = \epsilon(\pounds(\tau)a_I-\iota(\tau)F|_{W_I}) \\
& = \epsilon(\pounds(\tau)a_I-\rho(u'-u)|_{W_I}) \\
& = \epsilon(\pounds(\tau)a_I+\rho(u)|_{W_I}-\rho(u')|_{W_I}).
\end{align*}
\end{proof}

\begin{lemma}\label{restrict function} For any smooth function $f:X\to \C$, we have the equality
\[\mu qr(\xx_f)=-i\delta_l(\rho(f)),\] where $\xx_f\in\Cinf(\TTX)$ is the generalized Hamiltonian vector field associated to $f$.
\end{lemma} 
\begin{proof} For any section $v\in\Cinf(l)$, the vector field component $\pi(v)\in\Cinf(TX|Z\otimes \C)$ is tangent to $Z$.  Therefore, if $f:X\to \C$ is any smooth function, we have 
\[\la r(0,df),v\ra=\frac{1}{2}\pi(v)\cdot\rho(f).\]

Writing $f=f_R+if_I$, recall that the corresponding generalized Hamiltonian vector field is given by the formula
\[\xx_f=\J(0,df_R)+(0,df_I).\]  Given a section $v\in\Cinf(l)$, we then calculate
\begin{align*} \mu qr(\xx_f)(v) & = 2\la r(\xx_f),v\ra \\
& = 2\la \J r(0,df_R),v\ra+2\la r(0,df_I),v\ra \\
& = -2\la r(0,df_R),\J v\ra+2\la r(0,df_I),v\ra \\
& = -2i\la r(0,df_R),v)+2\la r(0,df_I),v\ra \\
& = -i\pi(v)\cdot\rho(f_R)+\pi(v)\cdot \rho(f_I) \\
& = -i\pi(v)\cdot(f_R+if_I)\\
& -i\delta_l\rho(f)(v).
\end{align*}

\end{proof}

Returning to the proof of part (2) of Proposition \ref{first order induced}, suppose $[\mu q(\xx)]=[\mu q(\xx')]\in H^1(\B)$.  Then there exists a smooth function $f:Z\to \C$ such that 
\[\mu q(\xx')-\mu q(\xx)=-i\delta_l f.\]   Choose a smooth function $\tilde{f}:X\to \C$ that extends $f$, i.e. such that $\rho(\tilde{f})=f$.  By Lemma \ref{restrict function}, we have 
\[\mu q(\xx)=\mu q(\xx'+r(\xx_{\tilde{f}})),\] or equivalently 
\[q(\xx)=q(\xx'+r(\xx_{\tilde{f}})).\]  This implies that 
\[\hB^{\xx'+r(\xx_{\tilde{f}})}=\hB^{\xx'}\cdot e^{\xx_{\epsilon\tilde{f}}}.\]  Lemma \ref{coboundary isomorphism} therefore implies that $\B^{\xx}$ is isomorphic to $\B^{\xx'}\cdot e^{\xx_{\epsilon}\tilde{f}}$ in $\Dt^{\B}_{\DN}(X,\J)$, so that by definition $\B^{\xx}$ and $\B^{\xx'}$ are isomorphic in $\De_{\DN}^{\B}(X,\J)$. Therefore $\B^{\xx}$ and $\B^{\xx'}$ determine the same element of the set $\df_{\B}(\DN):=\pi_0(\De_{\DN}^{\B}(X,\J)).$

To finish the proof of part (2) of Proposition \ref{first order induced}, it remains to show that if $\hB^{\xx}$ and $\hB^{\xx'}$ are isomorphic in $\De^{\B}_{\DN}(X,\J)$, then $[\mu q(\xx)]=[\mu q (\xx')]\in H^1(\B)$.  Let $(\Phi,z)$ be an isomorphism from $\hB^{\xx}$ to $\hB^{\xx'}$ in $\De^{\B}_{\DN}(X,\J)$; here $z=e^{\xx_{\epsilon f}}$ is an element of  $e^{\H_{\DN}(X)}$,  and  
\[\Phi=(e^{\epsilon \tau},\{\epsilon g_I\}):\hB^{\xx}\toco \hB^{\xx'}\cdot z\] is an isomorphism in $\Dt^{\B}_{\DN}(X,\J)$.  Write $\xx=(\xi,u)$, $\xx'=(\xi',u')$, and$\xx_f=(\eta_f,v_f)$, a calculation similar to that used in the proof of Lemma \ref{coboundary isomorphism} implies the following three conditions:
\begin{equation}\label{condition11}\xi'+r(\eta)-\xi=\iota_*\tau,\end{equation}
\begin{equation}\label{condition22} g_J-g_I=\pounds(\tau)c_{IJ},\end{equation} and 
\begin{equation}\label{condition33}dg\si=\pounds(\tau)a_I+\rho(u-u'-r(v_f))|_{W_I}.\end{equation}

Define $h_I:W_I\to \R$ by $h_I=\iota(\tau)a_I-g_I$.   Using the assumption that $a_J-a_I=dc_{IJ}$, together with condition (\ref{condition22}), we see that on each overlap $W_I\cap W_J$ we have 
\[h_J-h_I=0,\] so that there exists a unique function $h:Z\to \R$ with $h_I=h|_{W_I}$ on each $W\si$.   Using the Cartan formula $\pounds(\tau)a_I=d\iota(\tau)a_I+\iota(\tau)da_I$, we may then rewrite condition (\ref{condition33}) as \begin{equation}\label{hI condition}-dh|_{W_I}=(\iota(\tau)F)|_{W\si}+\rho(u-u'-r(v_f))|_{W_I}.\end{equation}  Choosing a smooth function $\tilde{h}:X\to \R$ that extends $h$, equation (\ref{hI condition}) is equivalent to 
\begin{equation}\label{hI condition1}\rho(u'+r(v_f-d\tilde{h})|_{W_I}-\rho(u)=\iota(\tau)F.\end{equation}  Defining $f'=f-i\tilde{h}$, we have 
\[\xx_{f'}=\xx_f-(0,d\tilde{h})=(\eta_f,v_f-d\tilde{h}).\]  Equation (\ref{condition11}) together with equation (\ref{hI condition1}) imply that $\xx'+r(\xx_{f'})-\xx$ is a section of $\TT\B$, so that
\[q(\xx')-q(\xx)=q(r(\xx_{f'})),\] and therefore
\[\mu q(\xx')-\mu q(\xx)=-\mu qr(\xx_{f'})=\delta_l(i\rho(f')),\] where in the last equality we used Lemma \ref{restrict function}.  Therefore, we see that 
\[[\mu q(\xx')]=[\mu q(\xx)]\in H^1(\B),\] as claimed.

Finally, part (3) of Proposition \ref{first order induced} follows easily from Proposition \ref{pi zero bundle}. 

\end{proof}

\section{Functoriality and Invariance}\label{section functoriality}

In this section we explain how various equivalence between GC branes $\B$ and $\B'$ induce equivalences between the corresponding formal deformation groupoids.  To do so, it will be convenient to rephrase the definition of the formal groupoid  $\De^{\B}(X,\J)$ from the action of $e^{\H(X,\J)}$ on $\Dt^{\B}(X,\J)$ in terms of a general construction.
\begin{definition}\label{definition action groupoid}  Let $\CC$ be a groupoid, and $G$ a group with a strict right action on $\CC$.  We define a new category $\CC//G$, called the (right) \emph{action groupoid of $G$ acting on $\CC$}, which has the same objects as $\CC$, and such that for pair of objects $x,y$ we have 
\[\Hom_{\CC//G}(x,y)=\{(\varphi,g):g\in G, \varphi\in \Hom_{\CC}(x,y\cdot g)\}.\]  Composition is defined by 
\begin{equation}\label{compo def}(\varphi,g)\circ (\psi,h)=((\varphi\cdot h)\circ \psi,gh).\end{equation}  For each object $x$, the identity morphism from $x$ to itself in $\CC//G$ is defined to be $(id_x,1)$, where $id_x\in \Hom_{\CC}(x,x)$ is the identity morphism in $\CC$, and $1\in G$ is the identity group element.
\end{definition} 
\begin{rem} It is an easy exercise using Definition \ref{group action on category} to check that the composition given by equation (\ref{compo def}) is well-defined, associative, and unital.  
\end{rem} 
%\begin{ex} Let $S$ be a set on which a group $G$ acts on the right.  If we define $\CC$ to be the category with set of objects $S$ and only trivial morphisms, then the action of $G$ on $S$ lifts trivially to an action of $G$ on $\CC$.  In this case, the groupoid $\CC//G$ may be identified with the usual action groupoid $S//G$ formed from the action of $G$ on $S$. \gnote{Citation?}
%\end{ex}
\begin{rem} Given a formal group $G$ (over $\Art$) with a right action on a formal groupoid $\CC$, we similarly define a formal groupoid $\CC//G$, which associates to  each $A\in\Art$ the groupoid $\CC\da//G\da$ defined as above.  For example, translating Definition \ref{full deformation groupoid} into this language, we see that 
\[\De\uB(X,\J)=\Dt\uB(X,\J)//e^{\H(X,\J)}.\]
\end{rem}

The proof of the following Proposition follows easily from the definitions, and is omitted.  
\begin{proposition}\label{induced equivalence}
Let $\CC$ be a groupoid equipped with a strict right action of the group $G$, and $\CC'$ a groupoid equipped with  a strict right action of a group $G'$.  Let $\varphi:G\toco G'$ be a group isomorphism, and let $\Phi:\CC\to \CC'$ be a functor that intertwines the group actions in the sense that, for every $x\in \CC$ and $g\in G$ we have $\Phi(x\cdot g)=\Phi(x)\cdot \varphi(g)$, and for every $\psi:x\to y$ in $\CC$ we have $\Phi(\psi\cdot g)=\Phi(\psi)\cdot \varphi(g)$.  Then there is a well-defined functor $\Phi//G$ which agrees with $\Phi$ on objects, and which maps a morphism $(\psi,g):x\to y$ in $\CC//G$ to $(\Phi(\psi),\varphi(g)):\Phi(x)\to \Phi(y)$ in $\CC'//G$.  Furthermore, $\Phi//G$ is an equivalence if and only if $\Phi$ is.
\end{proposition}
\begin{rem} This result has an immediate extension to the case that $\CC$ and $\CC'$ are formal groupoids and $\Phi,\Phi'$ are formal groups.
\end{rem}

We now return to our main topic.  Let $\B=(Z,\L)$ and $\B'=(Z,\L')$ be two elements of $\Br(X,\J)$ supported on the same submanifold $Z$.  Furthermore, suppose that $\L$ and $\L'$ are defined with respect to the same open cover $\W$ of $Z$.  Write $\L=(\{c_{IJ}\},\{a\si\})$ and $\L'=(\{c'\sij\},\{a'\si\})$.  Suppose we are given $\gamma=\{g_I:W_I\to \R\}$ satisfying 
 \[c'\sij-c\sij=g_J-g\si\] and \[a'\si-a\si=dg\si.\]  We say that $\gamma$ is an equivalence $\L\toco \L'$.
\begin{proposition}\label{invariance 1}   The equivalence $\gamma:\L\toco\L'$ induces an equivalence of formal groupoids
\[\tilde{\Phi}^{\gamma}:\Dt^{\B}(X,\J)\toco \Dt^{\B'}(X,\J)\] defined as follows. Given $A\in \Art$ and $\hB=(\hrho,\hL)\in \Dt^{\B}_A(X,\J)$, write $\hL=(\{c\sij+f\sij\},\{a\si+v\si\})$, and define 
\begin{equation}\label{onobjects}\tilde{\Phi}\da^{\gamma}(\hB):=(\hrho,(\{c'\sij+f\sij\},\{a'\si+v\si\})).\end{equation}   Given a morphism 
\[\psi=(e^{\tau},\{h\si\},z):\hB_1\to \hB_2\] in $\Def^{\B}_A(X,\J)$, we define 
\begin{equation}\label{gamma mor}\tilde{\Phi}\da^{\gamma}(\psi)=(e^{\tau},\{h\si+e^{\tau}g\si-g\si\},z):\tilde{\Phi}\da^{\gamma}(\B_1)\to\tilde{\Phi}\da^{\gamma}(\B_2).\end{equation}
%\item \ybox{Maybe omit this part} If $\gamma'$ is another equivalence from $\L\to \L'$ then $\Phi_{\gamma'}$ and $\Phi_{\gamma}$ induce the same isomorphism of deformation functors
%\[\Def_{\B}\to \Def_{\B'}.\] 

 \end{proposition}
\begin{proof} First, note that it is clear by inspection that the right-hand side of (\ref{onobjects}) defines an $A$-deformation of $\B'$.  Next, we must check that if $\psi=(e^{\tau},\{h\si\})$ is a morphisms from $\hB_1$ to $\hB_2$, then $\tilde{\Phi}\da^{\gamma}(\psi)$ as defined by equation (\ref{gamma mor}) does in fact define a morphisms from $\tilde{\Phi}\da^{\gamma}(\hB_1)$ to $\tilde{\Phi}\da^{\gamma}(\hB_2)$.  Writing $\hB_1=(\hrho_1,(\{c_{IJ}+f_{IJ,1}\},\{a_I+v_{I,1}\}))$ and $\hB_2=(\hrho_2,(\{c_{IJ}+f_{IJ,2}\},\{a_I+v_{I,2}\}))$, by Definition \ref{definition par def}, this means that
\begin{equation}\label{psi1} \hrho_2=e^{\tau}\hrho_1,\end{equation}
\begin{equation}\label{psi2} e^{\tau}(c_{IJ}+f_{IJ,1})-c_{IJ}-f_{IJ,2}=h_J-h_I,\end{equation} and 
\begin{equation}\label{psi3} e^{\tau}(a_I+v_{I,1})-a_I-v_{I,2}=dh_I.\end{equation} We need to check that these same equations hold when $c_{IJ}$ is replaced by $c'_{IJ}$, $a_I$ is replaced by $a'_I$, and $h_I$ is replaced by $h_I':=h_I+e^{\tau}g_I-g_I$.  Equation (\ref{psi1}) is unchanged under these replacements, and thus continues to hold automatically.  We calculate
\begin{align*} h'_J-h'_I & = h_J+e^{\tau}g_J-g_J-h_I-e^{\tau}g_I+g_I  \\
& =  h_J-h_I+e^{\tau}(g_J-g_I)-(g_J-g_I) \notag \\
& = e^{\tau}(c_{IJ}+f_{IJ,1})-c_{IJ}-f_{IJ,2}+e^{\tau}(c'_{IJ}-c_{IJ})-(c_{IJ}'-c_{IJ})  \\
& = e^{\tau}(c_{IJ}+f_{IJ,1}+c'_{IJ}-c_{IJ})-(c_{IJ}+f_{IJ,2}+c'_{IJ}-c_{IJ})  \\
& = e^{\tau}(c'_{IJ}+f_{IJ,1})-(c_{IJ}'+f_{IJ,2}),
\end{align*} so we see that the ``primed" version of equation (\ref{psi2}) holds.  Finally, we have 
\begin{align*} dh'_I & = dh_I+de^{\tau}g_I-dg_I \\
& = e^{\tau}(a_I+v_{I,1})-a_I-v_{I,2}+e^{\tau}(a'_I-a_I)-(a'_I-a_I) \\
& = e^{\tau}(a'_I+v_{I,1})-a'_I-v_{I,2},
\end{align*} so the ``primed" version of equation (\ref{psi3}) also holds.
This finishes the proof that $\tilde{\Phi}\da^{\gamma}(\psi)$ does in fact define an isomorphism from $\tilde{\Phi}\da^{\gamma}(\hB_1)$ to $\tilde{\Phi}\da^{\gamma}(\hB_2)$.

We need to verify that $\tilde{\Phi}\da^{\gamma}$ is functorial, i.e. we must check that, given isomorphisms $\psi:\hB_1\to\hB_2$ and $\tilde{\psi}:\hB_2\to \hB_3$ that $\tilde{\Phi}\da^{\gamma}(\tilde{\psi})\circ \Phi\da^{\gamma}(\psi)=\Phi\da^{\gamma}(\tilde{\psi}\circ\psi)$.  If we write $\psi=(e^{\tau},\{h_I\})$ and $\tilde{\psi}=(e^{\tilde{\tau}},\{\tilde{h_I}\})$, then by Definition \ref{definition par def} we have 
\[\tilde{\psi}\circ\psi = (e^{\tilde{\tau}}e^{\tau},\{\tilde{h}_I+e^{\tilde{\tau}}h_I\}).\] and therefore 
\[\Phi\da^{\gamma}(\tilde{\psi}\circ\psi)=(e^{\tilde{\tau}}e^{\tau},\{\tilde{h}_I+e^{\tilde{\tau}}h_I+e^{\tilde{\tau}}e^{\tau}g_I-g_I\}).\]  On the other hand, we calculate 
\begin{align*} \tilde{\Phi}\da^{\gamma}(\tilde{\psi})\circ\tilde{\Phi}\da^{\gamma}(\psi) & = (e^{\tilde{\tau}},\{\tilde{h}_I+e^{\tilde{\tau}}g_I-g_I\})\circ (e^{\tau},\{h_I+e^{\tau}g_I-g_I\}) \\
& =(e^{\tilde{\tau}}e^{\tau}, \{\tilde{h}_I+e^{\tilde{\tau}}g_I-g_I+e^{\tilde{\tau}}(h_I+e^{\tau}g_I-g_I)\}) \\
& = (e^{\tilde{\tau}}e^{\tau},\{\tilde{h}_I+e^{\tilde{\tau}}h_I+e^{\tilde{\tau}}e^{\tau}g_I-g_I\}).
\end{align*}

To complete the proof, consider $\gamma^{-1}:=\{-g_I\}$, which is an equivalence from $\L'\to\L$.  It is clear by inspection that
\[\tilde{\Phi}\da^{\gamma^{-1}}:\Dt_A^{\B'}(X,\J)\to \Dt_A^{\B}(X,\J)\] is a (strict) inverse for $\tilde{\Phi}\da^{\gamma}$; in particular, $\tilde{\Phi}\da^{\gamma}$ is an equivalence of groupoids.

\end{proof}

\begin{corollary}  The equivalence $\gamma:\L\toco\L'$ induces an equivalence of formal groupoids
\[\Phi^{\gamma}:\De^{\B}(X,\J)\toco \De^{\B'}(X,\J),\] where 
\[\Phi^{\gamma}=\tilde{\Phi}^{\gamma}//e^{\H(X,\J)}\] is defined as in Proposition \ref{induced equivalence}.
In particular, this induces a well-defined natural isomorphism of  functors
\[\df_{\B}\toco \df_{\B'}.\]

\end{corollary}

Let $\W=\{W_I\}_{I\in \mathcal{I}}$ be an open cover of $Z$ with indexing set $\mathcal{I}$.  Recall that a \emph{refinement} of $\W$ is a pair $(\W',\sigma)$, where  $\W'=\{W'_{I'}\}_{I'\in \mathcal{I}'}$ is an open cover of $Z$, and  $\sigma:\mathcal{I'}\to \mathcal{I}$ is a map of indexing sets such that for each $I'\in \mathcal{I}'$ we have 
\[W'_{I'}\subset W_{\sigma(I')}.\]  Given $\L\in \Herm(Z)$ defined with respect to the open cover $\W$, we obtain a new element $\sigma^*\L\in \Herm(Z)$ defined with respect to the open cover $\W'$.  Explicitly $\sigma^*\L=(\{c'_{I'J'}\},\{a'_{I'}\})$ with $c'_{I'J'}=(c_{\sigma(I')\sigma(J')})|_{W'_{I'J'}}$  and $a'_{I'}=(a_{\sigma(I')})|_{W_{I'}}$.  In particular, if $\B=(Z,\L)$ is a GC brane on $(X,\J)$, then the refinement $(\W',\sigma)$ determines $\sigma^*\B=(Z,\sigma^*2\L)$, which is again a GC brain on $(X,\J)$.  Furthermore, there is a natural functor of formal groupoids
\[\sigma^*:\De^{\B}(X,\J)\to \De^{\B}(X,\J).\]
\begin{proposition}\label{invariance 2} The functor 
\[\sigma^*\De^{\sigma^*\B}_A(X,\J)\to \De^{\sigma^*\B}_A(X,\J)\] is an equivalence of formal groupoids. In particular, it induces a natural isomorphism of functors
\[\df_{\B}\toco \df_{\sigma^*\B}.\]\end{proposition}
\begin{proof} Fix $A\in\Art$, and let  $\hB\in \Def^{\sigma^*\B}\da(X)$.  Using Propositions \ref{pi zero bundle} and, it follows that $\hB$ is isomorphic to an object of the form $\sigma\da^*\B\cdot e^{\xx}$ for some $\xx\in \cg\da(X)$; by definition, $e^{\xx}$ is compatible with $\J$ as in Definition \ref{compatibility1}. It is also clear that $\sigma\da^*\B\cdot e^{\xx}=\sigma\da^*(\B\cdot e^{\xx})$.  This shows that $\sigma^*$ is essentially surjective.  On the other hand, it is also clear that $\sigma\da^*$ is fully faithful, and therefore an equivalence.
\end{proof}

Finally, let $\B$ be a GC brane on $(X,\J)$, and let $u\in \Omega^1(X)$ be a 1-form.  This determines a new GC structure $\J'=e^{u}\cdot\J$.  Writing $\B=(Z,\L)$, define
\[\B'=e^{u}\cdot\B=(Z,e^{\rho(u)}\cdot\L);\] explicitly, if we write $\L=(\{c\sij\},\{a\si\})$, we have 
\[\L'=e^{\rho(u)}\cdot \L=(\{c\sij\},\{a\si-\rho(u)|_{W_I}\}).\]  Using Proposition \ref{action on GC submanifolds}, it follows that $\B'\in\Br(X,\J')$.
  For each $A\in\Art$, let us define 
\[\tilde{\Phi}\da^u:\Dt_A^{\B}(X,\J)\to \Dt_A^{\B'}(X,\J')\] as follows: given $\hB=(\hrho,\hL)\in \De_A^{\B}(X,\J)$, we
define \[\tilde{\Phi}\da^u(\hB)=(\hrho,e^{\hrho(u)}\cdot\hL).\]
\begin{proposition} \label{B transform on deformations}  \begin{enumerate}[(1)] \item $\tilde{\Phi}\da^u(\hB)$ is an object of $\De^{\B'}_A(X,\J')$ for each $\hB\in \Dt_A^{\B}(X,\J)$.
\item For each equivalence $\psi=(e^{\tau},\{g\si\}):\hB_1\toco\hB_2$ in $\Dt_A^{\B}(X,\J)$, $\psi$ also determines an equivalence $\tilde{\Phi}\da^u(\hB_1)\toco \tilde{\Phi}\da^u(\hB_2)$ in $\Dt_A^{\B'}(X,\J')$, which we denote by $\Phi^u\da(\psi)$.
\item The map taking $\hB\in\Dt_A^{\B}(X,\J)$ to $\tilde{\Phi}^u\da(\hB)$ and $\psi:\hB\to\hB'$ in $\Dt_A^{\B}(X,\J')$ to $\tilde{\Phi}^u\da(\psi)$ defines an equivalence of formal groupoids \[\Dt^{\B}(X,\J)\toco \Dt^{\B}(X,\J').\]
\end{enumerate} 
\end{proposition}  
\begin{proof}  The proof of part (2) is a straightforward computation (similar to that in the proof of Proposition \ref{invariance 1}), and is therefore omitted.   

To verify part (1), let $\{\tilde{W}_I\}$ be a collection of open sets in $X$ chosen as in Definition \ref{compatibility def}.  Also choose $\xi\in\Cinf\da(X)$ such that $\hrho=\rho e^{\xi}$ and $w_I\in\Omega\dm(\tilde{W}_I)$ such that $\hat{a}_I=a_I+\rho(w_I)$.  According to Definition \ref{compatibility def}, the compatibility of $\hB$ with $\J$ means that, for each $I$ we have 
\begin{equation}\label{comp1}\la (e^{w_I}e^{\xi}\cdot \J)K^{\B}\da,K_{A}^{\B}\ra\subset I\da^Z.\end{equation}

By construction, $\hB':=\tilde{\Phi}\da^u(\hB)$ is given by 
\begin{align*} & (\rho e^{\xi},(\{c_{IJ}\},\{a_I+\rho(w_I)-\rho e^{\xi}(u|_{\tilde{W_I}})\})
\\  = & (\rho e^{\xi},(\{c_{IJ}\},\{a_I-\rho(u)|_{W_I}+\rho(w_I+ (-e^{\xi}u+u)|_{\tilde{W}_I}))\}) \\
= & (\rho e^{\xi},(\{c_{IJ}\},\{a'_I+\rho(w'_I)\},)
\end{align*} where we have defined $a'_I:=a_I+\rho(u)|_{W_I}$ and $w_I':=w_I+ (u-e^{\xi}u)|_{\tilde{W}_I}$.  Therefore, to verify that $\hB'$ is compatible with $\J'$, we must check that for each $I$ we have 
\begin{equation}\label{comp2}\la (e^{w'_I}e^{\xi\si}\cdot\J')K_A^{\B'},K_A^{\B'}\ra\subset I\da^Z.\end{equation}  We have 
\[ e^{w'_I}e^{\xi\si}\cdot \J' = e^{w_I+ u-e^{\xi\si}u}e^{\xi\si}e^{u}\J=e^{w_I+u}e^{\xi\si}\J.\]  Also, by Proposition \ref{action on GC submanifolds} we have $K_A^{\B'}=e^{u}\cdot K\da^\B$, so that the left hand side of (\ref{comp2}) is equal to 
\begin{align*} \la (e^{u}\cdot e^{w_I}e^{\xi\si}\cdot\J)e^uK\da^{\B},e^{u}K\da^{\B}\ra & = \la e^u (e^{w_I}e^{\xi\si}\J)e^{-u}e^uK\da^{\B},e^uK\da^{\B}\ra \\ & = \la (e^{w_I}e^{\xi\si}\J)K\da^{\B},K\da^{\B}\ra \subset I\da^Z.
\end{align*} Therefore, we see that condition (\ref{comp1}) is equivalent to condition (\ref{comp1}), so the assumption that the former is satisfied implies that the latter is as well.  This completes the proof of part (1).

To see why part (3) of the Proposition is true, first note that the functoriality of $\Phi_u$ follows trivially from its definition.  Also, by inspection $\Phi_u$ is a bijection on both objects and morphisms, and is in particular an equivalence of categories.  Indeed, the functor $\Phi^{-u}\da:\Dt^{\B'}_A(X,\J')\to \Dt^{\B}\da(X,\J)$ is easily verified to be the inverse of $\Phi\da^u$.   

\end{proof}
 
Given $u\in\Omega^1(X)$ and $\xx=(\xi,a)\in \g\da(X)$, we have 
\[[(0,u),\xx]=(0,-\pounds(\xi)u)\] and 
\[[(0,u),[(0,u),\xx]]=0.\]  The adjoint action of $e^u$ on $\cg\da(X)$ is therefore given by 
\begin{equation}\label{adjoint}\Ad_{e^u}(\xi,a)=(\xi,a-\pounds(\xi)u).\end{equation}  If $\J$ is a GC structure on $X$, we clearly have $e^ue^{\T\da(X,\J)}e^{-u}=e^{\T\da(X,e^u\cdot\J)}$.
\begin{lemma} $e^u e^{\H\da(X,\J)}e^{-u}=e^{\H\da(X,e^u\cdot\J)}.$
\end{lemma}
\begin{proof} Recall from Proposition \ref{exp prop} the identity $e^{u}e^{\xx}e^{-u}=e^{\textrm{Ad}_u(\xx)}.$  Therefore, we need to show that for every $f:X\to\C$ there exists $f':X\to \C$ such that $\textrm{Ad}_u\xx^{\J}_f=\xx^{\J'}_{f'}$.  Given $f=f_R+if_I:X\to \C$, recall from Definition \ref{genhamdef} that 
\begin{equation}\label{Ham1}\xx_f^{\J}=\J(0,df_R)+(0,df_I).\end{equation} Therefore, setting $\J'=e^u\cdot\J$ we have 
\begin{equation}\label{Ham2}\xx_f^{\J'}=e^u\J  e^{-u}\cdot (0,df_R)+(0,df_I).\end{equation}  If $f=if_I$ is purely imaginary, we see from (\ref{adjoint}) that 
\[\Ad_{e^u}\xx^{\J}_f=\Ad_{e^u}(0,df_I)=(0,df_I)=\xx^{\J'}_f.\]  If $f$ is real, write $(\xi,a)=\xx_f^{\J}=\J(0,df)$, then we see that 
\[\xx^{\J'}_f=e^u\cdot\J e^{-u}(0,df)=e^u\cdot\J(0,df)=(\xi,a-\iota(\xi)du).\]  On the other hand, we have 
\begin{align*} \Ad_{e^u}\xx^{\J}_f & =(\xx,a-\pounds(\xi)u)\\
& = (\xx,a-\iota(\xi)du)-(0,d\iota(\xi)u)=\xx^{\J'}_{f-i\iota(\xi)u}.\end{align*}    \end{proof}

Let $\Psi:e^{\H(X,\J)}\to e^{\H(X,e^u\J)}$ be the  isomorphism of formal groups given for each $A\in\Art$ and $z\in e^{\H\da(X,\J)}$ by 
\[z\mapsto e^uze^{-u}.\]  For fixed $z\in e^{\H\da(X,\J)}$, let $R_z:\Dt^{\B}_A(X,\J)\to\Dt^{\B}_A(X,\J)$ be the functor $\hB\mapsto\hB\cdot z$ determined by the right action of $e^{\H\da(X,\J)}$ on $\Dt^{\B}_A(X,\J)$, and similarly let $R_{\Psi(z)}:\Dt^{e^u\cdot\B}_A(X,e^u\cdot\J)\to \Dt^{e^u\cdot\B}_A(X,e^u\cdot\J)$ be the functor $\hB\mapsto \hB\cdot \Psi(z)$.  

\begin{lemma} $\tilde{\Phi}\da^u\circ R_z=R_{\Psi(z)}\tilde{\Phi}\da^u$.
\end{lemma} 
\begin{proof} The result follows from an explicit calculation very similar to that used in the proof of Proposition \ref{right action par}. \end{proof}
 Therefore, using the construction described in Proposition \ref{induced equivalence}, we may extend the functor $\tilde{\Phi}^u:\Dt^{\B}(X,\J)\to \Dt^{e^u\cdot\B}(X,e^u\cdot\J)$ to a functor $\Phi^u:\De^{\B}(X,\J)\to \De^{e^u\cdot\B}(X,e^u\cdot\J)$.  Furthermore, since $\tilde{\Phi}^u$ is an equivalence, Proposition \ref{induced equivalence} implies the following result.
\begin{corollary} The functor $\Phi^u:\De^{\B}(X,\J)\to \De^{e^u\cdot\B}(X,e^u\cdot\J)$ is an equivalence of formal groupoids.
\end{corollary}

\section{Example: Leaf-wise Lagrangian branes}\label{LWL branes example}
Let $(X,J)$ be a complex manifold, and $\varphi:X\toco X$ a diffeomorphism satisfying $\varphi^*J=J$, i.e. a symmetry of $(X,J)$.  Given any complex submanifold $Z\subset X$, the submanifold $\varphi^{-1}(Z)\subset X$ will again be compatible with $J$: symmetries of $(X,J)$ act on its collection of complex submanifolds.  Similarly, for every Artin algebra $A$ and every holomorphic vector field $\xi\in \m\otimes_{\R}\Cinf(TX)$, the action of the formal symmetry $e^{\xi}$ on $Z$ induces a deformation of $Z$, i.e. an element $\hat{Z}$ of $\df_Z(A)$.  Since applying the inverse $e^{-\xi}$ to $\hat{Z}$ gives the trivial deformation (Z itself), we say that $\hat{Z}$ is trivializable. Of course, not every deformation is trivializable in this way; however, for each point $z\in Z$, we can find an open set $U\subset X$ containing $z$, so that every deformation of $Z\cap U$ in $(U,J|_U)$ is induced by a formal symmetry of $(U,J|_U)$. 

Similarly, given a GC brane $\B\in (X,\J)$ and a formal symmetry $g=e^{\xx}\in e^{\cg\da(X)}$ for some $A\in \Art$, we saw in Proposition \ref{action of formal symmetries} that $e^{\xx}$ induces an $A$-deformation $\B\cdot e^{\xx}$ of $\B$.  Motivated by the complex case described above, a natural question is whether every $A$-deformation is locally isomorphic to one of this form.  
\begin{definition}\label{definition locally extendible}  Let $\B=(Z,\L)\in \Br(X,\J)$ be a GC brane.  We say that $\B$ has \emph{trivializable} deformations if for each $A\in \Art$, and each $A$-deformation $\hB$ of $\B$, there exists a formal symmetry $e^{\xx}\in e^{\T\da(X)}$ such that $\hB$ is isomorphic to $\B\cdot e^{\xx}$ in $\Dt\da\uB(X,\J)$.  Similarly, we say that $\B$ has \emph{locally trivializable deformations} if for each $z\in Z$ there exists an open set $U\subset X$ containing $z$, such that $\B|_U\in\Br(U,J|_U)$  has trivializable deformations.
\end{definition}  

 Examination of examples shows that--unlike in the complex case discussed above--not every GC brane has locally trivializable deformations.  However, in this section we will prove the following result.  
\begin{theorem}\label{essentially surjective} Let $\B\in\Br(X,\J)$ be a leaf-wise Lagrangian brane (that is, a brane whose underlying GC submanifold is leaf-wise Lagrangian as in Definition \ref{definition LWL submanifold}). Then $\B$ has locally trivializable deformations.  \end{theorem}
%\begin{rem} In particular, Theorem \ref{essentially surjective} implies that the same is true if we replace the groupoid $\Dt^{\B}\da(U,\J)$ with $\De^{\B}\da(U,\J)$, since the later groupoid has the same objects, but more isomorphisms, than the former.
%\end{rem}
\begin{proof} Let $\B=(Z,\L)$ be a LWL brane, and $z\in Z$.  Given a neighborhood $U\subset X$ of $z$ and a 1-form $u\in\Omega^1(U)$, let $\J'=e^{-u}\J|_U$ and $\B'=\B|_U\cdot e^u$. By Proposition \ref{action on GC submanifolds}, we have $\B'\in\Br(U,\J'|_U)$.  Suppose that for every $A\in\Art$ and $\hB'\in\Dt^{\B'}_A(U,\J')$ we can find $x'\in e^{\T\da(U,\J')}$ such that $\hB$ is isomorphic to $\B'\cdot x$.  Then by applying Proposition \ref{B transform on deformations}, together with the fact that 
\[e^{u}e^{\T\da(U,\J')}e^{-u}=e^{\T\da(U,\J|_U)},\] it follows that same is true when $\B'$ is replaced by $\B$ and $\J'$ is replaced by $\J$. 
Using Theorem \ref{LWL normal form}, it is therefore sufficient to consider the case where $X$ is a neighborhood of the origin in $X^{m,n}_0=\R^{2m}\times\C^{n}$, that $Z=Z^k_0\cap X$ for natural numbers $m,n,k$ (with $k\leq n$), and where the curvature form $F$ of $\L$ is zero; here $Z^k_0\subset X_0$ is submanifold described in Example \ref{standard brane}.  Furthermore, by applying Propositions \ref{invariance 1} and \ref{invariance 2}, it is enough to consider the case that $\L$ is the trivial element of $\Herm(Z)$.

The rest of the proof will be similar to the proof of Theorem \ref{LWL normal form}, and we begin by recalling some of the notation introduced there.  We introduce coordinates $(s_0,\cdots, s_{2m},t_0,\cdots,t_{2n})$ on $X_0=\R^{2m}\times \C^n$, where $(s_0,\cdots,s_{2m})$ are coordinates on $\R^{2m}$ and $(t_0,\cdots, t_{2n})$ are (real) coordinates on $\C^n\cong\R^{2n}$.  We then relable these to obtain new coordinates $(x^1,\cdots x^{d},y^1,\cdots y^{d'})$, with $d=m+k$ and $d'=m+2n-2k$, defined by 
\[x_1=s_1,\cdots x_m=s_m,x_{m+1}=t_1,\cdots, x_{m+2k}=t_{2k},\] and \[y_1=s_{m+1}\cdots y_m=s_{2m}, y_{m+1}=t_{2k+1}\cdots y_{m+2n-sk}=t_{2n}.\]  In terms of these coordinates, $Z$ is defined by the equations $y^I=0$ for $I=1,\cdots, d'$.

Let $\hB=(\hrho,\hL)$ be an object of $\De^{\B}_A(U,\J)$.  Let $\pi:X\to Z$ be the projection map $(x,y)\mapsto x$, and $\pi^*:\Cinf_{Z}\to \Cinf_{X}$ the pull-back homomorphism.  %Since $\rho\circ\pi^*$ is the identity on $\Cinf\da(Z)$, it follows that $\hrho\circ\pi^*:\Cinf\da(X)\to\Cinf\da(X)$ is a homomorphism of $A$-algebras which is the identity modulo $\m\da$.  
\begin{lemma} There exists a unique $\tau\in \g\da(Z)$ such that 
\[\hrho\circ\pi^*=e^{\tau}.\]
\end{lemma}

\begin{proof}   

Suppose that $\mu:A\to A'$ is a small extension in $\Art$ with kernel $I$, and let $\sigma:A'\to A$ be an $\R$-linear splitting of $\mu$. Inductively, assume that the result holds for the Artin algebra $A'$.  Write $\Phi=\hrho\circ\pi^*:\Cinf\da(Z)\to \Cinf\da(Z)$, and let $\Phi'=\mu(\Phi):\Cinf_{A'}(Z)\to \Cinf_{A'}(Z)$ , then exists a unique $\tau'\in \g_{A'}(Z)$ such that $e^{\tau'}=\Phi'$.  Let $\tilde{\tau}=\sigma(\tau')$, i.e. for each $f\in \Cinf\da(Z)$ we have \[\tilde{\tau}(f)=\sigma(\tau'\mu(f)).\]  Define $\eta\in I\otimes_{\R}\End_{\R}(\Cinf(Z))$ such that for every $f\in\Cinf\da(Z)$ we have 
\[e^{\tilde{\tau}}(f)=\Phi(f)-\eta(f).\]  Define $\tau=\tilde{\tau}+\eta$.  Noting that $I\otimes\End_{\R}(\Cinf(Z))$ lies in the center of $A\otimes_{\R}\End_R(\Cinf(Z))$, we see that 
\begin{align*} e^{\tau}f & = e^{\eta}e^{\tilde{\tau}}(f) \\
& = e^{\eta}(\Phi(f)-\eta(f)) \\
& = \Phi(f)-\eta(f)+\eta(\Phi(f)-\eta(f)) \\
& =\Phi(f)-\eta(f)+\eta(f)\\
& = \Phi(f).
\end{align*}  
 Furthermore, we claim that $\eta$--and therefore $\tau$--is a derivation.  To see this, note that by Proposition \ref{action} $\tilde{\Phi}=e^{\tilde{\tau}}$ is a homomorphism (actually automorphism) of $R$-algebras. Therefore we have 
\begin{align*} \Phi(f)\Phi(g)& =(\tilde{\Phi}(f)-\eta(f))(\tilde{\Phi}(g)-\eta(g)) \\
& = \tilde{\Phi}(f)\tilde{\Phi}(g)-f\eta(g)-g\eta(f)  \\
& = \tilde{\Phi}(fg)-f\eta(g)-g\eta(f).\end{align*}  By comparing this with the equation 
\[\Phi(fg)=\tilde{\Phi}(fg)-\eta(fg),\] we see that indeed $\eta(fg)=f\eta(g)+\eta(f)g$ for every $f,g\in\Cinf\da(X)$.

\end{proof} 

 Therefore $e^{-\tau}\hrho\circ \pi^*$ is the identity.  Since $\hB$ is isomorphic in $\Dt^{\B}\da(U,\J)$ to $(e^{-\tau}\hrho,e^{-\tau}\hat{\L})$, we may assume without loss of generality that 
\begin{equation}\label{projection} \hrho\circ\pi^*=id_{\Cinf\da(Z)}.\end{equation}

In terms of the coordinates $(x,y)$, this means that
\[\hrho(x^i)=x^i\] and \[\hrho(y^I)\in \Cinf\dm(Z).\]  Let us write $\phi^I=\hrho(y^I)\in \Cinf\dm(Z)$, and define 
\[\xi=(\pi^*\phi^I)\frac{\partial}{\partial y^I}\in \g\da(X).\]  %By construction, $(\xi,0)$ is an element of $\E^Y$.
\begin{lemma} $\hrho=\rho e^{\xi}$.
\end{lemma}
\begin{proof}

By definition, there exists some $\eta\in \g\da(X)$ such that $\hrho=\rho e^{\eta}.$  Defining $\zeta\in\cg\da(X)$ by $e^{\zeta}=e^{\eta}e^{-\xi}$, we see that the desired result is equivalent to the equation
\[\rho e^{\zeta}=\rho.\]
Introducing the notation $\{w^{\alpha}\}$ for the collective coordinates $\{x^i,y^I\}$ on $U\subset X_0$, by assumption we have $\rho e^{\zeta}w^{\alpha}=\rho w^{\alpha}$ for each $\alpha$; we want to show that this implies the identity $\rho e^{\zeta}f=\rho f$ for arbitrary $f\in\Cinf\da(X)$.   Let us expand $\zeta=\zeta^{\alpha}\frac{\partial}{\partial w^{\alpha}}$, noting that the component functions $\zeta^{\alpha}$ are elements of $\Cinf\dm(U)$.  For every function $f\in\Cinf\da(X)$, we have 
\[\rho(\zeta(f))=\rho(\zeta^{\alpha})\rho(\frac{\partial}{\partial w^{\alpha}}f);\] in particular, if we can show that each component function $\zeta^{\alpha}$ satisfies $\rho(\zeta^{\alpha})=0$, then $\rho(\zeta(f))=0$ for every $f\in\Cinf\da(X)$; in this case we say that the \emph{restriction of $\zeta$ to $Z$ vanishes}.  Inductively, for each $k\geq 1$ we then have $\rho(\zeta^k(f))=0$ for every $f\in\Cinf\da(X)$, which in turn implies that $\rho e^{\zeta}=\rho$.  Therefore, it is sufficient to show that, if $\zeta$ is any element of $\g\da(X)$ such that $\rho e^{\zeta}w^{\alpha}=\rho w^{\alpha}$ holds for each $\alpha$, then each of the component functions $\zeta^{\alpha}$ satisfies $\rho(\zeta^{\alpha})=0$ (equivalently, the restriction of $\zeta$ to $Z$ vanishes).

Inductively, assume that this result holds for some $A'\in\Art$, and let $\mu:A\to A'$ is a small extension with kernel $I$.  Define $\zeta'=\mu(\zeta)\in\g_{A'}(X)$.  Since $\mu(w^{\alpha})=w^{\alpha}$ and $\mu(\rho)=\rho$, it follows that $\rho e^{\zeta'}w^{\alpha}=\rho w^{\alpha}$ for each $\alpha$, so by the inductive hypothesis the restriction of $\zeta'$ to $Z$ vanishes.  Let $\sigma:A'\to A$ be a linear splitting of $\mu$, and define $\lambda=\zeta-\sigma(\zeta')$, which by construction is an element of $I\otimes_{\R}\Cinf(TX)$.  It is easy to see that the vanishing of $\zeta'$ along $Z$ implies that $\sigma(\zeta')$ vanishes along $Z$ as well, so that $\rho e^{\sigma(\zeta')}=\rho$.  For each $f\in\Cinf\da(X)$ we therefore have 
\[\rho e^{\zeta}(f)=\rho(e^{\sigma(\zeta')}f+\lambda(f))=\rho f+\rho\lambda (f).\]  Expanding $\lambda=\lambda^{\alpha}\frac{\partial}{\partial w^{\alpha}}$, and noting that the component functions satisfy $\lambda^{\alpha}=\lambda(w^{\alpha})$, we then have
\[\rho(\lambda^{\alpha})=\rho e^{\zeta}w^{\alpha}-\rho w^{\alpha}=0.\]  Therefore $\lambda$ vanishes along $Z$, so that $\zeta=\sigma(\zeta')+\lambda$ does as well.

\end{proof}

Let us write $\hL=\L+u=e^{-u}\L$ for $u\in \Omega^1\dm(Z)$.  Define $w=\pi^*u\in \Omega^1\dm(X)$, and $\xx=(\xi,w)$.    The proof of Theorem \ref{essentially surjective} then follows from the following lemma.
\begin{lemma}\label{useful lemma}
\begin{enumerate} 
\item$\hB=\B\cdot e^{\xx},$ and 
 \item $\xx\in \T\da(X)$.
\end{enumerate}
\end{lemma}
\begin{proof}

As in the proof of Theorem \ref{LWL normal form}, we introduce $R\subset \TTX$ as the sub-bundle spanned by $\{(\frac{\partial}{\partial y^I},0),(0,dx^i)\}_{i,I}$ and $S\subset \TTX$ as the sub-bundle spanned by $\{(\frac{\partial}{\partial x^i},0),(0,dy^I)\}_{i,I}$.  In that proof, we saw that\begin{enumerate} \item $\TTX=R\oplus S$.  Furthermore, both $R$ and $S$ are maximal isotropic so the pairing gives an identification $S\cong R\uv$ and $R\cong S\uv$. 
\item $\J(R)=R$ and $\J(S)=S$.
\item $S|_{Z}=\TT \B$.
\end{enumerate}
We also introduced the space $\Y$ of vector fields on $X_0$ of the form $\sum_{I}c^I\frac{\partial}{\partial y^I}$ for constants $\{c^I\}_I$, and defined 
$\RR\subset \Cinf(R)$ and $\SS\subset \Cinf(S)$ to be the subspaces of elements $\xx\in \Cinf(R),\Cinf(S)$ satisfying $(\xi,0)\cdot\xx=0$ for all $\xi\in \Y$.  Via straightforward calculation, we proved the following lemma, which we state here again for convenience.
\begin{lemma}\label{RSlemma}
\begin{enumerate}
\item $\ll\RR,\SS\rr\subset \RR,$
\item $\ll\RR,\RR\rr=0.$
\item with respect to the isomorphism $\End(\TTX)\cong (R\oplus S)\otimes (R\oplus S)$ induced by the pairing, we have 
\[\J\in \RR\otimes\SS\oplus\SS\otimes\RR.\]
\end{enumerate}
\end{lemma}

By construction, both $(\xi,0)$ and $(0,u)$ are elements of $\RR$, so that we have 
\[0=\ll(\xi,0),(0,u)\rr=(0,\pounds(\xi)u).\]  Using Proposition \ref{exponential}, we see that 
\[e^{\xx}=e^{(0,u)}e^{(\xi,0)},\] from which the first part of Lemma \ref{useful lemma} easily follows.

It remains to show that 
\[\xx\cdot \J=0.\]

Since $\xx\in \RR$, it follows from Lemma \ref{RSlemma} that $\xx\cdot\J\in \RR\otimes\RR$, and also that $e^{-\xx}\cdot\J=\J-\xx\cdot\J$.  Since the pairing identifies $R$ with $S\uv$, we see that $\xx\cdot \J$ vanishes if and only if 
\[\la(\xx\cdot\J)\SS,\SS\ra = 0;\] this is equivalent to
\[\la(e^{-\xx}\cdot\J)\SS,\SS\ra =0,\] which is in turn equivalent to 
\[\la \J e^{\xx}(\SS),e^{\xx}\SS)\ra =0.\]  Since $\SS\subset K^{\B}$, it follows from the discussion in Remark \ref{rem1}  that 
\[\la \J e^{\xx}(\SS),e^{\xx}\SS)\ra\subset I^{Z}.\]  On the other hand, for each $x,y\in\SS$ the function $\la \J e^{\xx}(x),e^{\xx}y)\ra\subset I^{Z}$ is independent of the $y^I$ coordinates, so if it vanishes on $Z$ it vanishes on all of $X$.  We therefore conclude that 
\[\la(e^{-\xx}\cdot\J)\SS,\SS\ra =0,\] which finishes the proof of the lemma. 
\end{proof}
\end{proof}

\section{Induced Deformations}\label{induced deformations}  
In this section, we continue the study of induced (trivializable) deformations.

Fix a GC brane $\B\in\Br(X,\J)$.  For each $A\in\Art$ and $x\in e^{\T\da(X)}$, we introduce the notation

\[\Sigma\da\uB(x)=\B\cdot x\in\Dep.\]   In this section we will define a formal groupoid $\De\uB(X,\J)^{tr}$, whose objects for $A\in\Art$ are the elements of $e^{\T\da(X)}$.  We then extend the map $x\mapsto \Sigma\da\uB(x)$ to a functor of formal groupoids
\[\Sigma\uB:\De\uB(X,\J)^{tr}\to \De\uB(X,\J),\] which we prove is fully faithful; by construction, it is essentially surjective (and hence an equivalence) precisely when $\B$ has trivializable deformations.  

In order to define the groupoid $\De\uB(X,\J)^{tr}$ we will need some preliminary definitions.
\begin{definition} Let $\rrr(Z)$ be the following Lie algebra: as a vector space $\rrr(Z)=\Cinf(TZ)\oplus \Cinf(Z)$, and the bracket is given by 
\begin{equation}\label{r bracket}[(\xi,f),(\eta,g)]=([\xi,\eta],\xi(g)-\eta(f)+\iota(\eta)\iota(\xi)F).\end{equation}  
\end{definition}
\begin{rem} The Jacobi identity for the bracket (\ref{r bracket}) follows from the fact that $F$ is closed.
\end{rem} 

Consider the formal group $e^{\rrr(Z)}$.  For each $A\in\Art$, by Proposition \ref{exp prop} we see that for every $\xi\in\g\da(Z)$ and every $f\in\Cinf\dm(Z):=\m\otimes_A\Cinf(Z)$ we have
\begin{equation}\label{r multiplication}e^{(\xi,0)}e^{(0,f)}=e^{(0,e^{\xi}f)}e^{(\xi,0)}.\end{equation}   We also have the following result; the proof is identical to the proof of Proposition \ref{exponential}, and is therefore omitted.  

\begin{lemma}\label{exponential r} For each $A\in\Art$ and $(\xi,f)\in\rrr\da(Z)$, we have 
\[e^{(\xi,f)}=e^{(0,\int_0^1e^{t\xi}fdt)}e^{(\xi,0)}.\]
\end{lemma}

\begin{proposition} The map 
\[\mu:\rrr(Z)\to \cg(Z)\] given by 
\[(\xi,f)\mapsto (\xi,\iota(\xi)F-df)\] is a Lie algebra homomorphism.
\end{proposition}
\begin{proof}  Along with the Cartan formula $\pounds(\xi)=d\iota(\xi)+\iota(\xi)d$, recall also the identity $\iota([\xi,\eta])F=[\pounds(\xi),\iota(\eta)]F$.  Using these, we calculate $[\mu(\xi,f),\mu(\eta,g)]$ to be
\begin{align*}  &  [(\xi,\iota(\xi)F-df),(\eta,\iota(\eta)F-dg)] \\
= &  ([\xi,\eta],\pounds(\xi)\iota(\eta)F-\pounds(\xi)dg-\pounds(\eta)\iota(\xi)F+\pounds(\eta)df)\\
= &  ([\xi,\eta],\iota([\xi,\eta])F+\iota(\eta)\pounds(\xi)F-d\iota(\xi)dg-\pounds(\eta)\iota(\xi)F+d\iota(\eta)df)\\
= &  ([\xi,\eta],\iota([\xi,\eta])F+\iota(\eta)d\iota(\xi)F-d\iota(\eta)\iota(\xi)F-\iota(\eta)d\iota(\xi)F-d(\xi\cdot g-\eta\cdot f))\\
= &  ([\xi,\eta],\iota([\xi,\eta])F-d(\xi\cdot g-\eta\cdot f+\iota(\eta)\iota(\xi)F)) \\
=& \mu(([\xi,\eta],\xi\cdot g-\eta\cdot f+\iota(\eta)\iota(\xi)F)) \\
= &  \mu([(\xi,f),(\eta,g)]).
\end{align*}   

\end{proof}

%Given $\xi\in \g\da(X)$, consider the element $(\xi,\iota(\xi)F)\in\cg\da(X)$.   Regarding $\L$ as an element of $\De_{\L}(A)$-i.e. the trivial deformation-we obtain an element 
%\[e^{(\xi,\iota(\xi)F)}\cdot\L\in\De_{\L}(A).\]  We claim that this deformation is isomorphic, in a canonical way, to the trivial deformation.  Using Proposition \rnote{ref}, we calculate that 
%\[e^{(\xi,\iota(\xi)F)}=e^{(0,I_{\xi}(F))}e^{(\xi,0)},\] where we recall that $I_{\xi}(F):=\int_0^1e^{t\xi}Fdt$.  It follows that 
%\[e^{(\xi,\iota(\xi)F)}\cdot\L=(\{e^{\xi}a\si-I_{\xi}(F)|_{W\si}\},\{e^{\xi}c\sij\}).\]  If we define $\lambda_{\xi}=\{(\lambda_{\xi})_I\}$ by 
%\[(\lambda_{\xi})\si=\int_0^1e^{t\xi}\iota(\xi)a\si,\] then a simple calculation shows that $\lambda_{\xi}$ defines an isomorphism 
%\[\L\to e^{(\xi,\iota(\xi)F)}\cdot \L\]

%\begin{definition} Let $\KK\da(X)$ be the Lie algebra consisting of pairs $(\xx,f)$, where $\xx=(\xi,u)\in \T\da(X)$ with $\xi$ tangent to $Z$, and 
%\[\varphi(\rho(\xi),f)=(\rho(\xi),\rho(u))\in \cg\da(Z).\]  The bracket is defined by 
%\[[(\xi,u,f),(\eta,v,g)]=([\xi,\eta],\pounds(\xi)v-\pounds(\eta)u,\rho(\xi)g-\rho(\eta)f+\iota(\rho(\eta))\iota(\rho(\xi))F).\]  
%\[\chi:\KK\da(X)\to \T\da(X)\] be the Lie algebra homomorphism given by 
%\[(\xx,f)\mapsto \xx.\]
%\end{definition}

Let $\B$ be a GC brane on $(X,\J)$, with underlying GC submanifold $(Z,F)$.   Denote by $\T^{\B}(X)$ the Lie sub-algebra of $\T(X)$ consisting of elements $\xx=(\xi,u)\in\T(X)$ such that $\xi$ is tangent to $Z$.  This Lie algebra comes equipped with a restriction homomorphism
\[\rho:\T^{\B}(X)\to \cg(Z).\]  
\begin{definition}\label{definition KKK} Let $\KKK(X)$ be the fiber-product of Lie algebras
\[\KKK(X)=\T^{\B}(X)\times_{\cg(Z)} \rrr(Z).\]  In other words, an element of $\KKK(X)$ is a pair $(\xx,\hat{\tau})$, where $\xx=(\xi,w)\in \T\uB(X)$, $\hat{\tau}=(\tau,f)\in\rrr(Z)$, such that $\rho(\xi)=\tau$ and $\rho(w)=\iota(\tau)F-df.$  The Lie bracket is defined component-wise.  
\end{definition}  By construction, $\KKK(X)$ comes equipped with three homomorphisms:
\begin{enumerate} \item $\chi:\KKK(X)\to \T(X)$
\item $\rho_1:\KKK(X)\to \cg(Z)$ and 
\item $\rho_2:\KKK(X)\to \rrr(Z)$,
\end{enumerate} where $\rho_1=\mu\rho_2$.  We will denote the group homomorphisms induced by these Lie algebra homomorphisms by the same symbols. 

\begin{definition}\label{definition tilde extended} Let $\Dt\uB(X,\J)^{tr}$ be the following formal groupoid: For each $A\in\Art$, the objects of $\Dt\da\uB(X,\J)^{tr}$ are elements of the group $e^{\T\da(X)}$.  A morphism from $x\in e^{\T\da(X)}$ to $x'\in e^{\T\da(X)}$ is an element $y\in e^{\KKK\da(X)}$ such that 
\[x'=\chi(y)x.\]  Given $y,y'\in e^{\KKK\da(X)}$ and $x,x',x''\in e^{\T\da(X)}$ such that $x'=\chi(y)x$ and $x''=\chi(y')x'$, the composition
\[(\xymatrix{x' \ar[r]^{y'} & x'')}\circ (\xymatrix{ x \ar[r]^y & x'})\] is defined by 
\[\xymatrix{x \ar[r]^{y'y} & x''},\] where we note that 
\[x''=\chi(y')x'=\chi(y')\chi(y)x=\chi(y'y)x.\]  For every $x\in e^{\T\da(X)}$, the identity morphism $x\to x$ is the group identity element $1_{e^{\KKK\da(X)}}\in e^{\KKK\da(X)}$.
%Let $\B_U$ be the restriction of $\B$ to $U$, as described in \rnote{ref}.  
\end{definition}

  For each $x\in e^{\T\da(X)}$, denote 
\[\widetilde{\Sigma}^{\B}\da(x)=\B\cdot x\in\Dt\uB\da(X,\J).\]  Given $x,x'\in e^{\T\da(X)}$ and $y\in e^{\K\da(X)}$ such that $x'=\chi(y)x$, we will construct
\[\widetilde{\Sigma}^{\B}\da(y):\B\cdot x\to \B\cdot x'.\]  Write $x=e^{u}e^{\xi}$, $x'=e^{u'}e^{\xi'}$, $y=e^{(\yy,\hat{\tau})}$, where $\yy=(\eta,v)$ and $\hat{\tau}=(\tau,h)$.  By Proposition \ref{exponential}, we have 
\[\chi(y)=e^we^{\eta},\] where 
\begin{equation}\label{zeroth condition}w=\int_0^1e^{t\eta}vdt.\end{equation}  The equation $x'=\chi(y)x$ implies that  \begin{equation} \label{first condition} e^{\xi'}=e^{\eta}e^{\xi},\end{equation} and 
\begin{equation}\label{second condition} u'=w+e^{\eta}u.\end{equation}  Furthermore, by definition of the Lie algebra $\KKK\da(X)$, we have 
\begin{equation} \label{third condition} \rho e^{\eta}=e^{\tau}\rho.\end{equation} and \begin{equation}\label{fourth condition}\rho(v)=\iota(\tau)F-dh.\end{equation}

\begin{lemma}\label{sigma morphism lemma} $\widetilde{\Sigma}\uB\da(y)=(e^{\tau},\{g_I\})$ with 
\begin{equation}\label{g definition} g_I:=\int_0^1(e^{t\tau}(\iota(\tau)a_I+h|_{W_I})dt\end{equation} is an equivalence from $\widetilde{\Sigma}\uB\da(x)$ to $\widetilde{\Sigma}\uB\da(x')$ in $\Dt\da\uB(X,\J)$.
\end{lemma}
\begin{proof}
Write $\widetilde{\Sigma}\da\uB(x)=\B\cdot x=\hB=(\hrho,\hL)$ and $\tilde{\Sigma}\uB\da(x')=\B\cdot x'=\hB'=(\hrho',\hL')$ .  According to Definition \ref{definition par def}, we need to check that 
\begin{equation}\label{iso condition 1} \hrho'=e^{\tau}\hrho,\end{equation} and that $\{g_I\}$ is an isomorphism in $\Def^{\L}_A(X)$  from $\hL'\to e^{\tau}\hL$.   By construction, we have $\hrho=\rho e^{\xi}$ and $\hrho'=\rho e^{\xi'}$.  Combining equation (\ref{first condition}) with equation (\ref{third condition}), we see that (\ref{iso condition 1}) is indeed satisfied.

Next, note that by construction \[\hL'=(\{c\sij\},\{a_I+\rho(u')|_{W_I}\})\]  and \[\hL=(\{c\sij\},\{a_I+\rho(u)|_{W_I}\});\] therefore, we also have \[e^{\tau}\hL=(\{e^{\tau}c\sij\},\{e^{\tau}a_I+e^{\tau}\rho(u)|_{W_I}\}).\]  According to Definition \ref{Herm def}, therefore, $\{g_I\}$ is an isomorphism in $\Def^{\L}_A(Z)$ from $\hL'\to e^{\tau}\hL$ if and only if we have 
\begin{equation}\label{iso condition 2} g_J-g_I=e^{\tau}c_{IJ}-c_{IJ}\end{equation} and 
\begin{equation}\label{iso condition 3} dg_I=e^{\tau}a_I+e^{\tau}\rho(u)|_{W_I}-a_I-\rho(u')|_{W_I}.\end{equation}  By equation (\ref{g definition}) we have 
\begin{align*} g_J-g_I & = \int_0^1e^{t\tau}(\iota(\tau)(a_J-a_I))dt \\
& = \int_0^1(e^{t\tau}(\iota(\tau)dc_{IJ})dt \\ 
& =\int_0^1(\pounds(\tau)e^{t\tau}c_{IJ}dt \\
& =\int_0^1\frac{d}{dt}e^{t\tau}c_{IJ}dt \\
& = e^{\tau}c_{IJ}-c_{IJ},
\end{align*} which verifies condition (\ref{iso condition 2}).  To verify condition (\ref{iso condition 3}), we calculate
\begin{align*} dg_I & = \int_0^1e^{t\tau}(d\iota(\tau)a_I+dh_I)dt \\
& =\int_0^1e^{t\tau}(\pounds(\tau)a_I-\iota(\tau)da_I+\iota(\tau)F-\rho(v))dt\\
& =e^{\tau}a_I-a_I-\rho(w) \\ 
& = e^{\tau}a_I-a_I-\rho(u'-e^{\eta}u)\\
& = e^{\tau}a_I-a_I-\rho(u')+e^{\tau}\rho(u).
\end{align*}  Note that, in going from the first line to the second we used the Cartan formula for the Lie derivative, together with (\ref{fourth condition}); in going from the second to the third we used that $da_I=F$, $e^{t\tau}\pounds(a_I)=\frac{d}{dt}e^{t\tau}a_I$, and (\ref{zeroth condition}); in going from the third to the fourth we used (\ref{second condition}); and in going from the fourth line to the last we used (\ref{third condition}). 

\end{proof}  

\begin{proposition}\label{functoriality prop} The map sending $x\in e^{\T\da(X)}$ to $\tilde{\Sigma}\uB\da(x)=\B\cdot x$ and $y\in e^{\K\da\uB(X)}$ to $\tilde{\Sigma}\uB\da(y)$ is functorial; that is, for every $x,x',x''\in e^{\T\da(X)}$ and $y,y'\in e^{\KKK\da(X)}$ such that $x'=\chi(y)x$ and $x''=\chi(y')x'$, we have 
\begin{equation}\label{functoriality1}\tilde{\Sigma}\uB\da(y')\circ \tilde{\Sigma}\uB\da(y)=\tilde{\Sigma}\uB\da(y'y).\end{equation}
\end{proposition} 
\begin{proof} Suppose $x,x',x''\in e^{\T\da(X)}$ and $y,y',y''\in e^{\K\da(X)}$ satisfy $x'=\chi(y)x$ and $x''=\chi(y')x'$.  Defining $y''=y'y$, we may rewrite the condtion (\ref{functoriality1}) in the slightly altered form \begin{equation}\label{functoriality}\tilde{\Sigma}\uB\da(y')\circ \tilde{\Sigma}\uB\da(y)=\tilde{\Sigma}\uB\da(y'').\end{equation}
  Writing $\tilde{\Sigma}\uB\da(y)=(e^{\tau},\{g\si\})$, $\tilde{\Sigma}\uB\da(y')=(e^{\tau'},\{g\si'\})$, and $\tilde{\Sigma}\uB\da(y'')=(e^{\tau''},\{g''_I\})$, by Definition \ref{definition par def} equation (\ref{functoriality}) means that 
\begin{equation}\label{fun cond 1} e^{\tau'}e^{\tau}=e^{\tau''}, \end{equation} and that
\begin{equation}\label{fun cond 2} g'\si+e^{\tau'}g\si=g''\si \end{equation} holds for each $I$. Using formula (\ref{r multiplication}), we see that the conditions (\ref{fun cond 1}) and (\ref{fun cond 2}) may be combined into the equations
\begin{equation}\label{group prop}e^{(0,g_I')}e^{(\tau_I',0)}e^{(0,g_I)}e^{(\tau_I,0)}=e^{(0,g''_I)}e^{(\tau_I'',0)}\end{equation} in the groups $e^{\rrr\da(W_I)}$, where by definition we have $\tau_I:=\tau|_{W\si},\tau'\si:=\tau'|_{W\si},$ and $\tau''\si:=\tau''|_{W\si}$ .

Define $h\si:=h|_{W\si}$, $h'\si:=h'|_{W\si}$, and $h''\si:=h''|_{W\si}$, and also  $f_I=\iota(\tau\si)a_I+h\si$, $f'_I=\iota(\tau'\si)a_I+h'\si$, and $f''_I=\iota(\tau''\si)a_I+h''\si$.  Using the definition (\ref{g definition}) together with Lemma \ref{exponential r}, we see that (\ref{group prop}) may be rewritten as
\begin{equation}\label{group prop 2}e^{(\tau'\si,f'_I)}e^{(\tau\si,f_I)}=e^{(\tau''\si,f''_I)}.\end{equation}  By hypothesis, we have $y''=y'y$;  applying the homomorphism $\rho_2:e^{\KKK\da(X)}\to e^{\rrr\da(Z)}$ and restricting to $W_I$, this implies that
\begin{equation}\label{group prop 3}e^{(\tau'\si,h\si')}e^{(\tau\si,h\si)}=e^{(\tau\si'',h\si'')}\end{equation} holds in $e^{\rrr\da(W_I)}$. On the other hand, if we define $\sigma_I:\rrr\da(W_I)\to \rrr\da(W_I)$ by 
\[(\zeta,l)\mapsto (\zeta,l+\iota(\zeta)a_I),\] we see that $(\tau\si,f_I)=\sigma_I(\tau\si,h\si)$, $(\tau',f'_I)=\sigma_I(\tau'\si,h'\si)$ and $(\tau'',f''_I)=\sigma_I(\tau'',h''\si)$.  It is easy to check that $\sigma_I$ is a homomorphism of Lie algebras; therefore equation (\ref{group prop 2}) follows from (\ref{group prop 2}) by applying the group homomorphism $e^{\sigma_I}$.
\end{proof}  

Using Lemma \ref{sigma morphism lemma} and Proposition \ref{functoriality prop} in hand, we may now give the following definition.
\begin{definition}\label{definition sigma tilde} The functor $\widetilde{\Sigma}\uB:\De\uB(X,\J)^{tr}\to\Dep$ is defined as follows.  For each $A\in\Art$ and each $x\in e^{\T\da(X)}$, we have 
\[\widetilde{\Sigma}\uB\da(x)=\B\cdot x.\]  On morphisms $\widetilde{\Sigma}\uB\da$ is defined as  described in Lemma \ref{sigma morphism lemma}.
\end{definition}

\begin{proposition}\label{fully faithful}  The functor $\widetilde{\Sigma}\uB$ is fully faithful.
\end{proposition}
\begin{proof}   We first prove that $\widetilde{\Sigma}\uB$ is faithful.  Given $A\in\Art$, let $x,x'$ be elements of  $e^{\T\da(X)}$, and let $y,y'\in e^{\KKK\da(X)}$ be elements satisfying $x'=\chi(y)x$ and $x'=\chi(y')x$.  In particular, this implies that $\chi(y)=\chi(y')$.  Let us write  $y=e^{((\eta,v),(\tau,h))}$, where $(\eta,v)\in\T\da(X)$ and $(\tau,h)\in \KKK\da$, such that $\tau=\rho(\eta)$ and $\rho(v)=\iota(\tau)F-dh$.  Then $y'$ is of the form $e^{((\eta,v),(\tau,h'))}$, where $\rho(v)=\iota(\tau)F-dh'$.  By construction, we have $\widetilde{\Sigma}(y)=\widetilde{\Sigma}(y')$ if and only if, on each $W_I$ we have 
\[\int_0^1e^{t\tau}(\iota(\tau)a_I+h|_{W_I})dt=\int_0^1e^{t\tau}(\iota(\tau)a_I+h'|_{W_I})dt.\]  By Lemma \ref{exp formula cor}, this holds if and only if $h|_{W_I}=h'|_{W_I}$ for each $I$, or equivalently if $h=h'$.  This in turn holds if and only if $y=y'$.  Therefore we see that $\widetilde{\Sigma}\uB\da$ is faithful.

To show that $\widetilde{\Sigma}\uB\da$ is full, suppose we are given $x,x'\in e^{\T\da(X)}$ and a morphism $(e^{\tau},\{g\si\})$ from $\widetilde{\Sigma}\uB\da(x)\to \widetilde{\Sigma}\uB\da(x')$.  Define $\tilde{y}=x'x^{-1}.$ Writing $x=e^{u}e^{\xi}$, $x'=e^{u'}e^{\xi'}$, and $\tilde{y}=e^{w}e^{\eta}$, we have $e^{\eta}=e^{\xi'}e^{-\xi}$, and $w=u'-e^{\eta}u$.  Writing $\B=(Z,(\{c_{IJ}\},\{a_I\}))$, we have 
\[\widetilde{\Sigma}\uB\da(x)=(\rho e^{\xi},(\{c_{IJ}\},\{a_I+\rho(u)|_{W_I}\})\] and 
\[\widetilde{\Sigma}\uB\da(x')=(\rho e^{\xi'},(\{c_{IJ}\},\{a_I+\rho(u')|_{W_I}\}).\] The fact that $(e^{\tau},\{g\si\})$ is an equivalence from $\widetilde{\Sigma}\uB\da(x)$ to $\widetilde{\Sigma}\uB\da(x')$ is then equivalent to the following three conditions:
\begin{enumerate}[(I)] \item $\rho e^{\xi'}=e^{\tau}\rho e^{\xi}$, or equivalently $e^{\tau}\rho=\rho e^{\xi'}e^{-\xi}=\rho e^{\eta}$, which implies that $\eta$ is tangent to $Z$ and $\rho(\eta)=\tau$. %\bnote{Do I need some proof for this statement?}
\item $g_J-g_I=e^{\tau}c_{IJ}-c_{IJ}$ and 
\item $dg_I = e^{\tau}(a_I+\rho(u)|_{W_I})-(a_I+\rho(u')|_{W\si})
 = e^{\tau}a_I-a_I-\rho(w)|_{W_I}.$
 \end{enumerate}  Write $\tilde{y}=e^{(\eta,v)}$, so that 
 \[w=\int_0^1e^{t\eta}vdt\] and therefore
 \[\rho(w)=\int_0^1e^{t\tau}\rho(v)dt,\] we see from condition (III) that 
 \[ dg\si=\int_0^1e^{t\tau}(\pounds(\tau)a_I-\rho(v)|_{W_I})dt,\] or equivalently
 \begin{equation}\label{int equat} dg\si-d\int_0^1e^{t\tau}\iota(\tau)a_Idt=\int_0^1e^{t\tau}(\iota(\tau)F|_{W\si}-\rho(v)|_{W\si})dt.\end{equation}  Also write $e^{(0,g_I)}e^{\tau|_{W_I}}=e^{(\tau|_{W_I},k_I)}\in e^{\rrr\da(W_I)}$, so that 
 \[g_I=\int_0^1e^{t\tau}k_Idt.\]  Defining $h_I=k_I-\iota(\tau)a_I$, equation (\ref{int equat}) implies that
 \[d\int_0^1 e^{t\tau}h_Idt= \int_0^1e^{t\tau}(\iota(\tau)F|_{W_I}-\rho(v)|_{W_I})dt.\] Therefore, by Lemma \ref{exp formula cor} we have 
 \begin{equation}\label{Lie equation}dh_I=(\iota(\tau)F-\rho(v))|_{W_I}.\end{equation}  On the other hand, using condition (II), we see that on each overlap $W_{IJ}$ we have 
 \begin{align*} \int_0^1e^{t\tau}(k_J-k_I)dt & = \int_0^1e^{t\tau}\iota(\tau)dc_{IJ}dt \\
 & = \int_0^1 e^{t\tau}(\iota(\tau)a_J-\iota(\tau)a_I)dt,\end{align*} which implies (again using Lemma \ref{exp formula cor}) that $h_J=h_I$, so there is a unique function $h\in\Cinf\dm(Z)$ such that $h_I=h|_{W_I}$.  By equation (\ref{Lie equation}), it follows that $((\eta,v),(\tau,h))$ is an element of $\KKK\da(X)$, and by construction
 \[\widetilde{\Sigma}\uB\da(e^{((\eta,v),(\tau,h))})=(e^{\tau},\{g_I\}).\]  This completes the proof that $\widetilde{\Sigma}\uB\da$ is full. \end{proof}

Given an open set $U\subset X$,  denote $\Dt\uB(U,\J):=\Dt^{\B|_U}(U,\J|_U)$ and $\Dt\uB(U,\J)^{tr}:=\Dt^{\B|_U}(U,\J|_U)^{tr}$, as well as $\widetilde{\Sigma}\uB(U):=\widetilde{\Sigma}^{\B|_U}$.
\begin{corollary}\label{equivalence of induced} Let $\B=(Z,\L)$ be a $GC$ brane with locally trivializable deformations.  Then for every $z\in Z$, there exists a neighborhood $U\subset X$ of $z$ such that 
\[\widetilde{\Sigma}\uB(U):\Dt^{\B}(U,\J)^{ex}\to \Dt^{\B}(U,\J)\] is an equivalence.  In particular, this applies to LWL branes (by Proposition \ref{essentially surjective}). 
\end{corollary}
\begin{proof} Given $\B=(Z,\L)\in\Br(X,\J)$ with locally trivializable deformations and $z\in Z$, by definition there exists a neighborhood $U\subset X$ of $z$ such that for every $A\in\Art$ the functor 
\[\widetilde{\Sigma}\da^{\B}(U):\Dt^{\B}\da(U,\J)^{ex}\to \Dt^{\B}\da(U,\J)\] is essentially surjective.  By Proposition \ref{fully faithful}, it is also fully faithful, and therefore an equivalence.
\end{proof}

%We next incorporate the action of generalized Hamiltonian vector fields.
The proof of the following Proposition, which we omit, is a straightforward consequence of Definition \ref{group action on category} and Definition \ref{definition tilde extended}
\begin{proposition}  There is a strict right action of $e^{\H(X)}$ on $\Dt\uB(X,\J)^{tr}$ defined as follows: given $x\in e^{\T\da(X)}$ and $z\in e^{\H\da(X)}$, we define $x\cdot z$ to be the product $xz$ (recall that $e^{\H\da(X)}$ is a subgroup of $e^{\T\da(X)}$).  Given a morphism $y:x\to x'$, we define $y\cdot z=y$, now regarded as a morphism from $xz\to xz$.
\end{proposition}

For the following definition, recall Definition \ref{definition action groupoid}.
\begin{definition}  Let $\De\uB(X,\J)^{tr}$ be the action groupoid $\Dt\uB(X,\J)^{tr}//e^{\H(X)}$.  Explicitly, for $A\in\Art$ an object of $\Dee$ is an element of $e^{\T\da(X)}$.  Given $x,x'\in e^{\T\da(X)}$, a morphism from $x$ to $x'$ is  a pair $(y,z)\in e^{\KKK\da(X)}\times e^{\H\da(X)}$ such that $x'z=\chi(y)x$.  Composition is given by group multiplication, i.e. 
\[(y',z')\circ (y,z)=(y'y,z'z).\]
\end{definition}

It is clear that the functor $\widetilde{\Sigma}\uB:\Dt\uB(X,\J)^{tr}\to\Dt\uB(X,\J)$ is compatible with the right actions of $e^{\H\da(X)}$ on both formal groupoids. Therefore, it follows from Proposition \ref{induced equivalence} that $\widetilde{\Sigma}^{\B}$ naturally extends to a functor $\Sigma\uB:\De\uB\da(X,\J)^{tr}\to\De\uB\da(X,\J)$.  Furthermore, $\Sigma\uB$ is an equivalence if and only if $\widetilde{\Sigma}^{\B}$ is an equivalence.  This leads to the following corollary.
 
\begin{theorem} Let $\B=(Z,\L)$ be a leaf-wise Lagrangian brane on a GC manifold $(X,\J)$ (or more generally a brane with locally trivializable deformations).  Then for every $z\in Z$, there exists a neighborhood $U\subset X$ of $z$ such that 
\[\Sigma\uB(U):\De\uB(U,\J)^{ex}\to \De\uB(U,\J)\] is an equivalence of formal groupoids.
\end{theorem}

\section{Deformations via gluing}   

Let $\B$ be a GC brane on $(X,\J)$.  As discussed in Remark \ref{brane def sheaf}, the formal groupoid  $\Dt\uB(X,\J)$ extends in a natural way to a presheaf of formal groupoids on $X$.  In particular, for any open cover $\U=\{U\sa\}$ of $X$, we may define the groupoid of \emph{descent data} for $\Dt\uB$ with respect to $\U$.  This is described explicitly in the following definition.

%Let $\B=(Z,\L)\in\textbf{Br}(X,\I)$.  Given an open set $U\subset X$, we may restrict to obtain an element $\B|_{U}\in \Br(U,\I|_{U})$.  Given another open subset $V\subset U$, we have an equality $(\B|_{U})|_V=\B_{V}$.  \bnote{This is a consequence of using a Cech model}  Furthermore, for each $A\in \textbf{Art}_{\R}$ \bnote{make sure I define this notation earlier}, there is a restriction functor \bnote{Do I have to say more about how this is compatible with maps of Artin rings $A\to B$?}
%\[\Dt_{\B}(A)\to \Dt_{\B|_{U}}(A),\] such that for each pair of open sets $V\subset U\subset X$, we have a commutative diagram
%\[\xymatrix{ \Dt_{\B}(A) \ar[r] \ar[dr] & \Dt_{\B|_U}(A) \ar[d] \\ & \Dt_{\B|_V}(A). }\]
%\begin{proposition}\label{sheaf of groupoids} This defines a (strict) \bnote{pre?}sheaf of groupoids on $X$.
%\end{proposition}
%Let $\U=\{U\sa\}$ be an open cover of $X$.  

%Let $\B\sa$ denote the restriction of $\B$ to $U\sa$. \bnote{Maybe just say the following is the descent groupoid associated to a certain co-simplicial groupoid?  Cite Manetti reference at least for convention of how the maps go}
\begin{definition} \label{glued deformations} Let  $\Dt\uB(\U,\J)$ be the following formal groupoid.  For each $A\in\Art$,  an object of $\Dt\uB\da(\U,\U)$ is pair  $\hB=(\{\hB\sa\},\{\Psi\sab\})$, where 
\begin{enumerate} \item each $\hB\sa$ is an object of $\Dt\uB\da(U\sa,\J)$,
\item each $\Psi\sab:\hB\sb\toco \hB\sa$ is an isomorphism, and 
\item on triple overlaps we have $\Psi_{\alpha\gamma}=\Psi_{\alpha\beta}\Psi_{\beta\gamma}:\B_{\gamma}\to \B_{\alpha}$.
\end{enumerate}
A morphism from $\hB=(\{\hB\sa\},\{\Psi\sab\})$ to $\hB'=(\{\hB'\sa\},\{\Psi'\sab\})$ is a collection $\{\Phi\sa:\hB\sa\to \hB'\sa\}$ such that on each overlap we have $\Psi'_{\alpha\beta}\Phi_{\beta}=\Phi_{\alpha}\Psi_{\alpha\beta}$.
\end{definition}

By construction, there is a natural restriction functor 
\[\tilde{R}\uB:\Dt\uB(X,\J)\to \Dt\uB(\U,\J).\]  Explicitly, given $A\in|Art$ and $\hB\in\Dtp$, we define $\tR(\hB)=(\{\hB\sa\},\{\Psi\sab\})$ by setting $\hB\sa=\hB|_{U\sa}$ for each $\alpha$, and $\Psi\sab=\textrm{id}_{\hB|_{U\sab}}$ on each $U\sab$.  Given a morphism $\Phi:\hB\to\hB'$, we define 
\[\tR(\Phi)=\{\Phi|_{U\sa}\}:\tR(\hB)\to\tR(\hB').\]

%which are clearly compatible with the actions of $e^{\H\da(X)}$ on both groupoids.   Using Proposition \ref{induced equivalence}, we may therefore extend $\tilde{R}(A)$ to a functor $R(A):\Dep\to \De\uB\da(\U,\J).$ % In particular, this induces a map of deformation functors 
%\[R:\De_{\B}\to \De_{\U,\B}.\]
\begin{proposition}\label{descent property}The restriction functor $\tilde{R}\uB:\Dt\uB(X,\J)\to \Dt\uB(\U,\J)$ is an equivalence of formal groupoids.  %In particular, the induces map $R$ of deformation functors is an isomorphism.
\end{proposition}
   
\begin{proof} %By Proposition \ref{induced equivalence}, it is sufficient to check that $\tilde{R}(A):\Dtp\to \Dt^{\B}\da(\U,\J)$ is an equivalence.

Fix $A\in\Art$.  By inspection, it is clear that $\tilde{R}\uB$ is faithful, i.e. for each pair of objects $\hB=(\hrho,\hL)$ and $\hB'=(\hrho',\hL')$ in $\Dt^{\B}\da(X,\J)$ the induced map
\[\Hom_{\Dt^{\B}\da(X,\J)}(\hB,\hB')\to \Hom_{\Dt_A^{\B}(\U,\J)}(\tilde{R}\uB\da(\hB),\tilde{R}\uB\da(\hB'))\]  is injective.  On the other hand, suppose we are given a morphism $\{\Phi\sa\}:\tilde{R}\uB\da(\hB)\to \tilde{R}\uB\da(\hB')$.  By definition, on each open set $U\sa$ we have an equivalence $\Phi\sa:\hB|_{U\sa}\to \hB'|_{U\sa}$, and on each overlap we have $\Phi\sa|_{U\sab}=\Phi\sb|_{U\sab}$.  It is straightforward to see from Definition \ref{definition par def} that this implies the existence of $\Phi:\hB\to \hB'$ whose restriction to each $U\sa$ is equal to $\Phi\sa$.  Thus we see that $\tR$ is fully faithful.

To finish the proof, we will show that $\tR$ is also essentially surjective.
Given $(\{\hB\sa\},\{\Psi\sab\})\in \Dt^{\B}_A(\U,\J)$, write $\Psi\sab=(e^{\tau\sab},\psi\sab)$.  In particular, part (3) of Definition \ref{glued deformations}, implies that on each $Z\cap U_{\alpha\beta\gamma}$ we have 
\[e^{\tau_{\alpha\beta}}e^{\tau_{\beta\gamma}}=e^{\tau_{\alpha\gamma}}.\]
\begin{lemma}\label{nonabelian coboundary} There exist $\{\sigma\sa\in \Cinf\dm(Z\cap U\sa)\}$ such that on $Z\cap U\sab$ we have 
\[e^{\tau\sab}=e^{-\sigma\sa}e^{\sigma\sb}.\]
\end{lemma}
\begin{proof} By Proposition \ref{induct on small extensions}, we may proceed by induction on small extensions.  Thus, suppose the result holds for some $A'\in\Art$ with unique maximal ideal $\m'$, and let $\mu:A\to A'$ be a small extension. Denote the kernel of $\mu$ by $I$.  Choose a linear splitting $\nu:A\to A'$ of $\mu$.  Since $e^{\tau_{\alpha\beta}}e^{\tau_{\beta\gamma}}=e^{\tau_{\alpha\gamma}}$ holds in $e^{\g_{\m}(U_{\alpha\beta\gamma})}$, we also have $e^{\mu(\tau_{\alpha\beta})}e^{\mu(\tau_{\beta\gamma})}=e^{\mu(\tau_{\alpha\gamma})}$ in $e^{\g_{\m'}(U_{\alpha\beta\gamma})}$.  By the inductive hypothesis, we can find $\{\tilde{\sigma}\sa\in \g_{\m'}(U_{\alpha})\}$ such that 
\[e^{\mu(\tau_{\alpha\beta})}=e^{-\tilde{\sigma}\sa}e^{\tilde{\sigma}\sb}.\]  We must then have 
\begin{equation}\label{almost}e^{-\nu(\tilde{\sigma}\sa)}e^{\nu(\tilde{\sigma}\sb)}=e^{\tau\sab+\zeta_{\alpha\beta}}\end{equation} for some $\zeta_{\alpha\beta}\in I\otimes \Cinf(T(Z\cap U_{\alpha\beta})).$  Since $I\otimes \Cinf(T(Z\cap U_{\alpha\beta})$ lies in the center of $\g_{\m}$, it is easy to see from equation (\ref{almost}) that on triple overlaps $\zeta_{\beta\gamma}-\zeta_{\alpha\gamma}+\zeta_{\alpha\beta}=0$, so that we can choose $\eta_{\alpha}\in I\otimes \Cinf(T(Z\cap U_{\alpha}))$ satisfying $\eta_{\beta}-\eta_{\alpha}=-\zeta_{\alpha\beta}$.  Setting $\sigma_{\alpha}=s(\tilde{\sigma}\sa)+\eta\sa$, it follows from equation (\ref{almost}) that
\[e^{-\sigma\sa}e^{\sigma\sb}=e^{\tau\sab}.\]

\end{proof}

Returning to the proof of Proposition \ref{descent property}, choose $\{\sigma\sa\}$ as in Lemma \ref{nonabelian coboundary}, and define $\hrho'\sa=e^{\sigma\sa} \hrho\sa$, $\hL'\sa=e^{\sigma\sa}\hL\sa$, and 

\[\hB'\sa=(\hrho'\sa,\hL'\sa).\]
Note that, by Lemma \ref{compatibility under isomorphism}, for each $\alpha$ we have $\hB'\sa\in \Dt^{\B}_A(U\sa,\J)$.

By construction, we have equivalences 
\[\Phi\sa:=(e^{\sigma\sa},id_{\hL'\sa}):\hB\sa\to \hB'\sa.\]

  If we define $\Psi'_{\alpha\beta}=\Phi\sa \Psi_{\alpha\beta}\Phi^{-1}\sb$, then by construction $\hB'=(\{\hB'\sa\},\{\Psi'\sab\})\in \Dt^{\B}_A(\U,\J)$, and $\Phi=(\{\Phi\sa\})$ defines an isomorphism in $\Dt^{\B}_A(\U,\J)$ from $\hB\to \hB'$.  Furthermore, by construction we have 
\begin{align*}\Psi'_{\alpha\beta} & =(e^{\sigma\sa},id_{e^{\sigma\sa\cdot} \hL\sa})(e^{\tau\sab},\psi\sab)(e^{-\sigma\sb},id_{e^{-\sigma\sb}\cdot \hL\sb})\\
& = (e^{\sigma\sa}e^{\tau\sab}e^{-\sigma\sb},\psi'\sab)\\
& = (1,\psi'\sab), \end{align*} where  
\[\psi'\sab=(e^{\sigma\sa}\cdot\psi\sab):\hL'\sa\to\hL'\sb.\]  In particular, on the overlaps $Z\cap U\sab$ we have $\hrho'\sa=\hrho'\sb$, so there exists a unique $\hrho:\Omega\ub\da(X)\to \Omega\ub\da(Z)$ such that each $\hrho'\sa$ is given by restricting $\hrho$. 
 
For each $\alpha$ write $\hB'_{\alpha}=(\hrho\sa,\hL\sa)$.  Using Proposition \ref{pi zero bundle}, we may assume without loss of generality that each $\hL\sa$ is of the form $(\{(c\sa)_{IJ}\},\{(\ha\sa)\si\})$, i.e. has undeformed transition functions $\{(\hc\sa)\sij=c_{IJ}|_{U\sa}\}$.  Therefore, if we write $\Psi'\sab=(1,\{(g\sab)\si\})$, it follows from Definition \ref{definition par def} and Definition \ref{Herm def} that on each $(U\sab\cap W_I)\cap (U\sab\cap W_J) $ we have $(g\sab)_I=(g\sab)_J$, so that there is a well-defined function $g\sab:Z\cap U\sab\to \R$ whose restriction to each $U\sab\cap W\si$ is equal to $(g\sab)_I$.  Furthermore, we see that on $U_{\alpha\beta\gamma}\cap Z$ we must have 
\[g_{\alpha\beta}+g_{\beta\gamma}=g_{\alpha\gamma},\] so we may choose $h_{\alpha}\in\Cinf\dm(Z\cap U\sa)$ such that $g\sab=h_{\beta}-h_{\alpha}$.  If we define $\{(\ha''\sa)_I=(\ha'\sa)_I-d(h\sa)|_{W\si}\}$, $\hL''\sa = (\{c_{IJ}\},\{(\ha''\sa)\si\}),$ and 
\[\hB''\sa=(\hrho\sa,\hL''\sa),\] then $\hB''=(\{\hB''\sa\},\{\Psi''\sab=id\})$ defines an element of $\Dt^{\B}\da(\U,\J)$, which by construction  is isomorphic to $\hB'$, and hence also to $\hB$.  Also by construction, there exists a unique $\hB_0\in \Dt^{\B}_A(X)$ such that $\tR(\hB_0)=\hB''$.  This completes the proof that $\tR$ is essentially surjective.
\end{proof}

We saw in Proposition \ref{right action par} that there is a strict right action of the formal group $e^{\T(X)}$ on $\Dt\uB(X,\J)$.  Similarly, for each open set $U\subset X$, we have a strict right action of $e^{\T(X)}$ on $\Dt\uB(U,\J)$ constructed using the restriction homomorphism $e^{\T(X)}\to e^{\T(U)}$ together with the action of $e^{\T(U)}$ on $\Dt\uB(U,\J)$.    We then have the following easy result, the proof of which is omitted.  
\begin{proposition} There is a strict right action of $e^{\T(X)}$ on $\Dt\uB(\U,\J)$ defined as follows: for each $A\in\Art$, given $g\in e^{\T\da(X)}$ and $\hB=(\{\hB\sa\},\{\Psi\sab\})\in \Dt\uB\da(\U,\J)$, we define 
\[\hB\cdot g=(\{\hB\sa\cdot g|_{U\sa}\},\{\Psi\sab\cdot g|_{U\sab}\}).\]  Similarly, given an isomorphism $\{\Phi\sa\}:\hB\to\hB'$ in $\Dt\uB\da(\U,\J)$, we define 
\[\{\Phi\sa\}\cdot g=\{\Phi\sa\cdot g\}.\]
\end{proposition}

\begin{definition} Let $\De^{\B}(\U,\J)$ be the formal groupoid 
\[\Dt\uB(\U,\J)//e^{\H(X)}.\]  Explicitly, for each $A\in\Art$ $\Dt\uB\da(\U,\J)$ has the same objects as $\Dt\uB\da(\U,\J)$.  A morphism from $\hB\to \hB'$ in $\De\uB\da(\U,\J)$ is a pair $(\Phi,z)$, where $z\in e^{\H\da(X)}$, and $\Phi$ is a morphism in $\Dt\uB\da(\U,\J)$ from $\hB$ to $\hB'$.  Composition is given by $(\Phi',z')(\Phi,z)=(\Phi'\Phi,z'z)$.  
\end{definition}

Clearly, the restriction functor 
\[\tilde{R}\uB:\Dt\uB(X,\J)\to \Dt\uB(\U,\J)\] is compatible with the actions of $e^{\H(X)}$ on both formal groupoids.  Using Proposition \ref{induced equivalence}, we may then extend $\tilde{R}\uB$ to a functor 
\[R\uB:\De\uB(X,\J)\to \De\uB(\U,\J).\]  Furthermore, combining Proposition \ref{induced equivalence} and Proposition \ref{descent property}, we arrive at the following result.
\begin{theorem}\label{restriction equivalence} The restriction functor $R\uB:\De\uB(X,\J)\to \De\uB(\U,\J)$ is an equivalence of formal groupoids.
\end{theorem}

\section{Semicosimplicial groupoids and descent}\label{cosimplicial groupoids} The remainder of the paper will be devoted to constructing a DGLA governing the deformation theory of branes with locally trivializable deformations.  To do so, we will adapt techniques developed in \cite{I}, as well as \cite{FMM}\cite{BM}.  As a first step, in this section we review the framework of semicosimplicial groupoids and the descent groupoid construction.  We will use this to reformulate the definitions introduced in the previous section more systematically. This will enable us to formulate and prove some results which would be much more cumbersome otherwise. 

%bnote{Say something like ``the remainder of this paper is devoted to studying deformations using cosimplicial objects and DGLAs..."  Could also give motivation by constructing a DGLA governing $\Dee$  Also, say one of the points of this section is to strengthen theorem about induced deformations and LWL branes}
 Most of the following discussion--including the notation--is taken from \cite{BM}.  Let $\Delta_{mon}$ be the category whose objects are the finite ordinal sets $[n]=\{0,1,\cdots,n\}$ for $n=0,1,\cdots$, and whose morphisms are order-preserving injective maps among them.   Recall that a \emph{semicosimplicial object} in a category $\CC$ is a functor $A\ut:\Delta_{mon}\to \CC$.  This may be pictured as a diagram  
 \begin{equation}\label{simplicial diagram}A\ut=\xymatrix @C=.3in{A_0 \ar@<+.4ex>[r]\ar@<-.4ex>[r] & A_1 \ar@<+.6ex>[r]\ar@<-.6ex>[r] \ar[r]& A_2\ar@<+1ex>[r]\ar@<-1ex>[r]\ar@<+.3ex>[r]\ar@<-.3ex>[r] & \dots}, \end{equation} where each $A_n:=A\ut([n])$ is an object of $\CC$. The \emph{coface maps} $\partial^i_n:A_n\to A_{n+1}\}_{i=0}^{n+1}$ (for $i=0,\cdots, n+1$) satisfy a number of relations determined by the combinatorial structure of $\Delta_{mon}$.  
 
 More generally, given a (strict) 2-category $\CC$, we may similarly define a (strict) semicosimplicial object in $\CC$: this consists of a diagram of the form (\ref{simplicial diagram}), where the entries are objects of $\CC$ and the arrows are 1-morphisms, which we require to satisfy cosimplicial relations on the nose.  The example we will need is when $\CC$ is the 2-category of formal groupoids (over $\Art$). 

 %\begin{definition} A \emph{semi-cosimplicial object} in a category $\CC$ is a (covariant) functor $A\db:\Delta_{mon}\to \CC$.
 %\end{definition}
%\bnote{Spell this out a little bit.}

%\begin{definition} A (strict) \emph{semi-cosimplicial groupoid} is a (strict, covariant) functor from $\Delta_{mon}$ to the 2-category of groupoids.
%\end{definition}
%\bnote{Maybe make remark about groupoids over $\Art$.}
%\bnote{Maybe include reference for notion of a strict functor.}  Thus, a semi-cosimplicial groupoid $\G\db$ consists of a collection of groupoids $\G_0, \G_1,\cdots$ together with functors
%\[\partial_0,\partial_1,\cdots,\partial_{k+1}:\G_k\to \G_{k+1}\] which strictly satisfy the usual \gnote{Citation:May} co-simplicial relations on the nose.

\begin{ex}\label{example nerve groupoid} Let $\F$ be a sheaf on a topological space $X$, valued in a category (or strict 2-category) $\CC$.  %In other words, for every open set $U\subset X$ we have a groupoid $\CC(U)$, for each $V\subset U$ we have a functor $\rho_{V,U}:\CC(U)\to \CC(V)$, such that for each $W\subset V\subset U$ we have $\rho_{W,V}\rho_{V,U}=\rho_{W,U}$, and such that $\rho_{U,U}$ is the identity functor.  
Given an open cover $\U=\{U\sa\}$ of $X$, we may define a semicosimplicial object $A\ut$ in $\CC$, called the \emph{nerve} of $\F$.  For each $n=0,1,\cdots$ we take 
\[A\ut_{n}=\Pi_{\alpha_{0},\cdots\alpha_{n}}\F(U_{\alpha_0,\cdots\alpha_n}).\]  The coface maps $\partial^i_n:A\ut_n\to A\ut_{n+1}$ are constructed using the restriction maps of the sheaf $\F$: for each $n=0,1,\cdots$ and each $i=0,\cdots, n+1$, we have 
\[\partial_n^i=\Pi_{\alpha_0,\cdots,\alpha_{n+1}}\F(U_{\alpha_0\cdots\alpha_{n+1}},U_{\alpha_0\cdots \hat{\alpha}_i\cdots\alpha_{k+1}}),\] where for each inclusion of open subset $U\subset V$ the notation $\F(U,V):\F(V)\to\F(U)$ denotes the restriction map. % For example, given $x=\{x_{\alpha_0\cdots\alpha_k}\}$, we have 
%\[(\partial_ix)_{\alpha_0\cdots\alpha_{k+1}}=x_{\alpha_0\cdots\hat{\alpha_i}\cdots\alpha_{k+1}}.\]
\end{ex}
 
\begin{definition}\label{definition descent groupoid} \cite[Thm. 2.6]{BM} Given a semicosimplicial groupoid $\G_{\bullet}$, the \emph{groupoid of descent data} $\Desc(\G\db)$ is defined as follows:
\begin{enumerate} \item An object of $\Desc(\G\db)$ is a pair $(l,m)$ with $l$ an object of $\G_0$, and $m$ a morphism in $\G_1$ from $\partial_0l\to\partial_1l$, such that the equation 
\[(\partial_0m)(\partial_1m)^{-1}(\partial_2m)=1\] holds in $\G_2$.
\item A morphism from $(l,m)$ to $(l',m')$ is a morphism $a:l\to l'$ in $\G_0$, such that the following diagram commutes:
\[\xymatrix{ \partial_0l \ar[r]^{m} \ar[d]_{\partial_0a} & \partial_1l \ar[d]^{\partial_1a} \\ \partial_0l' \ar[r]^{m'} & \partial_1l'.}\]
\end{enumerate}
\end{definition}

\begin{ex}\label{nerve groupoid 2}  Continuing Example \ref{example nerve groupoid}, let us describe explicitly the groupoid $\Desc(\G\db)$ in the case that $\G\db$ is the nerve of a strict presheaf of groupoids on $X$.  An object is a pair $(\{x\sa\},\{\varphi\sab\}\}$, where $x\sa\in \CC(U\sa)$, and the $\varphi\sab$ are morphisms
\[\varphi\sab:x_{\beta}\to x_{\alpha}\] (where both objects have implicitly been restricted to $U_{\alpha\beta}$).  The morphisms must satisfy 
\[\varphi_{\gamma\beta}\varphi_{\beta\alpha}=\varphi_{\gamma\alpha}\] on the triple intersections $U_{\alpha\beta\gamma}$. A morphism from $x=(\{x\sa\},\{\varphi\sab\})$ to $x'=(\{x'\sa\},\{\varphi'\sab\})$ is a collection $\{\psi\sa\}$ of morphisms $\psi\sa:x\sa\to x'\sa$ that satisfy \[\psi\sa\varphi_{\alpha\beta}=\varphi'_{\alpha\beta}\psi\sb\] on overlaps $U\sab$.
\end{ex}

\begin{rem} Definition \ref{definition descent groupoid} and Remark \ref{nerve groupoid 2} extend immediately to the case of formal groupoids.  For example, in the case of the sheaf of formal groupoids $\De\uB(X,\J)$ associated to a GC brane $\B$,  the formal groupoid in Definition \ref{glued deformations} with respect to an open cover $\U$ is precisely the descent groupoid for the nerve of $\De\uB(X,\J)$. 
\end{rem}

Given a GC brane $\B\in\Br(X,\J)$, recall from \S\ref{induced deformations} the formal groupoid $\Dt\uB\da(X,\J)^{tr}$ described in Definition \ref{definition tilde extended}.  This groupoid may also be refined to a sheaf of formal groupoids over $X$, 
which assigns to each open subset $U\subset X$ the formal groupoid $\Dt\uB(U,\J)^{tr}:=\Dt^{\B|_U}\da(U,\J|_U)^{tr}$.  %\bnote{I may need to say something about the case when $U$ doesn't intersect $Z$.  In this case, $\Dt\uB\da(U,\J)^{tr}$ should be equivalent to the trivial groupoid.  Probably what I want is to define $\mathcal{K}\uB\da(U)$ to be equal to $\T\uB\da(U)$ in this case.}
\begin{definition} Let $\Dt\uB(\U,\J)^{tr}$ be the groupoid formal of descent data associated to the nerve of the sheaf of groupoids $U\mapsto \Dt(U,\J)^{tr}$ with respect to an open cover $\U$.  Explicitly, for each $A\in\Art$  an object of $\Dt\uB\da(\U,\J)^{tr}$ is a pair $(\{x\sa\},\{y\sab\})$ where 
\begin{enumerate} \item $x\sa$ is an element of $e^{\T\da(U\sa)}$, 
\item $y\sab$ is an element of $e^{\KKK\da(U\sab)}$,
\item on each overlap $U\sab$ we have $x\sa=\chi(y\sab)x\sb$, 
\item on each triple overlap $U_{\alpha\beta\gamma}$ we have $y_{\beta\gamma}y_{\alpha\gamma}^{-1}y_{\alpha\beta}=1.$
\end{enumerate} 
A morphism from $(\{x\sa\},\{y\sab\})$ to $(\{x'\sa\},\{y'\sab\})$ in $\Dt\uB\da(\U,\J)^{tr}$ is a collection $\{w\sa\}$, where 
\begin{enumerate} \item $w\sa\in e^{\KKK\da(U\sa)}$, 

\item on each $U\sa$ we have $x'\sa =\chi(w\sa) x\sa$, 
\item on each double overlap $U\sab$ we have $y'\sab w\sb=w\sa y\sab$.
\end{enumerate}
The composition of morphisms is given by group multiplication, i.e. 
\[\{w\sa'\}\circ\{w\sa\}=\{w'\sa w\sa\}.\]  The identity morphisms is given by the identity group elements.
\end{definition}

We then have the following easy result, whose proof is omitted.  
\begin{proposition}  For each $A\in\Art$, each object $(\{x\sa\},\{y\sab\})\in \Dt\uB\da(\U,\J)^{tr}$, and each $g\in e^{\H\da(X)}$, let 
\[(\{x\sa\},\{y\sab\})\cdot g=(\{x\sa\cdot g|_{U\sa}\},\{y\sab\}.\]  For each morphism 
\[\{w\sa\}: (\{x\sa\},\{y\sab\})\to (\{x'\sa\},\{y'\sab\}),\] let $\{w\sa\}\cdot g=\{w\sa\}$, regarded as a morphism from $(\{x\sa\},\{y\sab\})\cdot g$ to $(\{x'\sa\},\{y'\sab\})\cdot g$.  Then this defines a strict right action of $e^{\H(X)}$ on $\Dt\uB(\U,\J)^{tr}$.
\end{proposition}

\begin{definition}\label{definition of extended glued} Let 
\[\De\uB(\U,\J)^{tr}:=\Dt\uB(\U,\J)^{tr}//e^{\H(X)}\] be the formal groupoid associated to the right action of $e^{\H(X)}$ on $\Dt\uB(\U,\J)^{tr}$, as described in Definition \ref{definition action groupoid}.
\end{definition}

The map taking a formal semicosimplicial groupoid to its formal groupoid of descent data is functorial in a natural sense \cite{BM}.  Namely, there is a (strict) 2-category of  semicosimplicial formal groupoids, and the descent construction defines a 2-functor from this category to the two-category of formal groupoids.  For example,  a map from a semicosimplicial formalgroupoid $\G\db$ to another one  $\H\db$ by definition consists a collection of functors $\Psi_k:\G_k\to \H_k$ that are compatible with the boundary maps.  Such a map induces a functor $\textrm{Desc}(\Psi\db):\textrm{Desc}(\G\db)\to \textrm{Desc}(\H\db)$. 

\begin{proposition}\label{equivalence of descent groupoids}\cite{BM} Let $\Psi\db:\G\db\to\H\db$ be a map of semicosimplicial formal groupoids, such that for each $k$ the functor $\Psi_k:\G_k\to \H_k$ is an equivalence.  Then $\textrm{Desc}(\Psi\db):\textrm{Desc}(\G\db)\to \textrm{Desc}(\H\db)$ is also an equivalence.
\end{proposition}

\begin{ex} Let $\CC$ and $\DD$ be (strict) sheaves of formal groupoids on $X$, and let $\Sigma:\CC\to \DD$ be a strict map of such sheaves, i.e. a collection of functors $\Sigma(U):\CC(U)\to\DD(U)$ that are strictly compatible with the restriction functors.  Then $\Sigma$ induces a map of semicosimplicial formal groupoids from the nerve of $\CC$ to the nerve of $\DD$, which in turn induces a functor $\Sigma(\U)$ of the associated descent groupoids. Furthermore, if for each open subset $U\subset X$ the functor $\Sigma(U):\CC(U)\to\DD(U)$ is an equivalence of groupoids, then then induced functor $\Sigma(\U)$ is also an equivalence.
\end{ex}
\begin{definition} Let $\widetilde{\Sigma}\uB(\U):\Dt\uB(\U,\J)^{tr}\to \Dt\uB\da(\U,\J)$ be the functor induced by the map of sheaves described in Definition \ref{definition sigma tilde}.
\end{definition}

By construction, the functor $\widetilde{\Sigma}\uB(\U):\Dt\uB(\U,\J)^{tr}\to \Dt\uB(\U,\J)$ is compatible with the actions of $e^{\H(X)}$ on both formal groupoids.  Therefore, using Proposition \ref{induced equivalence}, we may extend $\widetilde{\Sigma}\uB$ to a functor from $\Def^{\B}(\U,\J)^{tr}$ to $\Def^{\B}(\U,\J)$.
\begin{definition}\label{definition extended nerve} $\Sigma\uB(\U): \De\uB(\U,\J)^{tr}\to \De\uB(\U,\J)$ is the functor $\widetilde{\Sigma}\uB(\U)//e^{\H(X)}$ induced from $\widetilde{\Sigma}\uB(\U)$ using Proposition \ref{induced equivalence}. \end{definition}

\begin{theorem}\label{theorem extended and LWL}  Let $\B$ be a leaf-wise Lagrangian GC brane on $(X,\J)$ (or more generally a brane with locally trivializable deformations).  Then there exists an open cover $\U$ of $X$ such that 
\[\Sigma\uB(\U): \De\uB(\U,\J)^{tr}\to \De\uB(\U,\J)\] is an equivalence of formal groupoids.  Furthermore, by Theorem \ref{restriction equivalence}, the formal groupoids $\De\uB(\U,\J)^{tr}$ and $\De\uB(X,\J)$ are also equivalent. %In particular, there is an equivalence of deformation functors
%\[\De_{\B}\cong \pi_0(\Def_{\U,\B}^{tr}).\] \bnote{Make sure I say the right-hand side is a deformation functor.}
\end{theorem}
\begin{proof}  By Corollary \ref{equivalence of induced}, we there exists an open cover $\U=\{U\sa\}$ of $X$ such that, for each $\alpha$, the functor 
\[\widetilde{\Sigma}\uB(U\sa):\Dt\uB(U\sa,\J)^{tr}\to \Dt\uB(U\sa,\J)\] is an equivalence.   It then easily follows from Proposition \ref{equivalence of descent groupoids} that 
\[\widetilde{\Sigma}\uB(\U):\Dt\uB(\U,\J)^{tr}\to \Dt\uB(\U,\J)\] is an equivalence as well.  The proof then follows by applying Proposition \ref{induced equivalence}.
\end{proof}

\section{Construction of the DGLA for branes with locally trivializable deformations}\label{DGLA theory}
%In this final section construct, for each LWL brane $\B$ a DGLA, which we prove governs the deformation functor associated to $\B$.  We will then use this to prove one of our main results, Theorem \ref{}, which gives a criterior for the deformations of $\B$ to be unobstructed.  We first recall the definition of a DGLA, as well as the construction of the associated Deligne groupoid and deformation functor.  Following \gnote{cite} we then introduce semicosimplicial DGLAs and bisemicosimplicial DGLAs, as well as the Thom-Whitney totalization and its important properties.   With the necessary definitions and results in hand, we then give our explicit construction of the DGLA.   Finally, we prove Theorem \ref{}, which involves an explicit examination of the cohomology of the DGLA. 

\subsection{DGLA's and the Deligne functor}\label{DGLA obstructions}

%One powerful approach to studying geometric deformation problems--such as that encoded in the functor in Definition \ref{brane deformation functor}--is through the use of differential graded Lie algebras (DGLA's).  \gnote{citation}  The strategy may be explained roughly as follows.  Any DGLA $L$ over a field $k$ canonically determines a functor
%\begin{equation}\label{DGLA deformation functor}\textrm{Def}_L:\textrm{Art}_k\to\textrm{Set},\end{equation} sometimes called the \emph{deformation functor} associated to $L$.   There exist many powerful tools for studying functors of this form; in particular, many important properties of $\textrm{Def}_L$ can be translated into straightforward calculations involving the DGLA $L$.  Given a functor $F:\textrm{Art}_k\to \textrm{Set}$ encoding some geometric (or algebraic) problem of interest, one then tries to find a DGLA $L$ together with a natural isomorphism $\Phi:\textrm{Def}_L\toco F$.  Ideally, one should also be able to describe $L$ (and $\Phi$) in fairly concrete terms.

%Given a nilpotent DGLA $\g$, let $\Del(\g)$ denote the corresponding Deligne groupoid of Maurer-Cartan elements.
%\begin{enumerate}
%\item \bnote{Motivate by discussing strategy of using DGLAs to study geometric deformation theory problems}
%\item \bnote{briefly give definition of DGLA, Mauer-Cartan elements, Deligne groupoid and associated deformation functor}.

%\item \bnote{Explain why first order deformations correspond to elements of 1st cohomology, and how obstructions relate to second cohomology}

%\end{enumerate}

For simplicity, in the following we work over the field $\R$.  A differential graded Lie algebra (DGLA)  consists of a cochain complex $\g\ub$, together with a cochain map (the ``bracket")
\[[\cdot,\cdot]:\g\ub\otimes \g\ub\to \g\ub\] of degree zero, which is skew-symmetric (in the graded sense), and satisfies a graded version of the Jacobi identity \cite{I}.
\begin{ex}\label{Lie algebra as DGLA} Any Lie algebra may be regarded as a DGLA concentrated in degree zero.  Conversely, given an arbitrary DGLA $\g\ub$, the restriction of the bracket to $\g^0$ gives it the structure of a Lie algebra.
\end{ex}
For any DGLA $\g\ub$, the set of \emph{Maurer-Cartan elements} of $\g\ub$ are defined as 
\[\textrm{MC}(\g\ub)=\{x\in\g^1:dx+\frac{1}{2}[x,x]=0\}.\]  If $\g\ub$ is nilpotent, the group $e^{\g^0}$ is well-defined (sometimes called the \emph{gauge group}).  This group has a left action on the set of Maurer-Cartan elements of $\g\ub$ given by the formula
\begin{equation}\label{gauge action}e^y\cdot x=x+\sum_{n=0}^{\infty}\frac{[y,-]^n}{(n+1)!}([y,x]-dy).\end{equation}  Given an arbitrary DGLA $\g\ub$ (not necessarily nilpotent), one may define a formal groupoid $\Del_{\g\ub}$ over $\Art$, known as the \emph{Deligne groupoid}.  This is defined for each $A\in\Art$ (with unique maximal ideal $\m\subset A$) as the action groupoid 
\begin{equation}\label{deligne groupoid} \Del_{\g\ub}(A):=\textrm{MC}(\g\ub\otimes\m)//e^{\g^0\otimes\m}.\end{equation}  Passing to equivalence classes, we obtain a functor $\df_{\g\ub}:\Art\to\textrm{Set}$, which assigns to every $A\in\Art$ the set of Maurer-Cartan elements of $\g\otimes\m$ modulo the action of $e^{\g^0\otimes\m}$.

\begin{ex}  As in Example \ref{Lie algebra as DGLA}, let $\g$ be a Lie algebra, regarded as a DGLA concentrated in degree $0$.  Then for each $A\in\Art$, the Deligne groupoid $\Del_{\g}(A)$ may be identified with the groupoid
\begin{equation}\label{Deligne for Lie algebra} *//e^{\g\otimes\m}.\end{equation}  By definition, the groupoid (\ref{Deligne for Lie algebra}) has a unique object $*$, and the set of morphisms from $*$ to itself are the elements of the group $e^{\g\otimes\m}$, with composition given by group multiplication. 
\end{ex} 

Consider a functor $F:\Art\to \Set$.  Let $\mu:A'\to A$ be a surjective map in $\Art$, and $x\in F(A)$.  One says that $x$ can be \emph{extended} to an element of $F(A')$ if there exists $x'\in F(A')$ such that $F(\mu)(x')=x$.  In general there will be an obstruction to the existence of such an extension.  The functor $F$ is called \emph{unobstructed}, however, if for every surjective map $A'\to A$ the induced map $F(A')\to F(A)$ is also surjective.  The fact that every surjective map in $\Art$ can be factored through a sequence of small extensions (Proposition \ref{induct on small extensions}) implies that $F$ is unobstructed if and only if $F(A')\to F(A)$ is surjective for every small extension $A'\to A$.

%Suppose $F$ is of the form $\De_{\g}$ for some DGLA $\g$.  Given a small extension $\mu:A'\to A$ and a Maurer-Cartan element $x$ of $\g\otimes\m$ let us try to solve the problem of finding a Maurer-Cartan element $x'$ of $\g\otimes\m'$ such that $\mu(x'):=id_{\g}\otimes\mu(x')=x$.  Let $\sigma:A\to A'$ be a linear splitting of $\mu$, and choose a generator $\epsilon\in A'$ for the kernel $I$ of $\mu$; since $I$ is annihilated by the maximal ideal $\m'\subset A'$, not that we have we have $I=\epsilon A'=\epsilon\R$.  Any $x'$ as above is necessarily of the form
%\[x'=\sigma(x)+\epsilon y\] for a unique $y\in \g^1$.  We then see that 
%\begin{align*} dx'+\frac{1}{2}[x',x'] & = \sigma(dx)+\epsilon dy +\frac{1}{2}[\sigma(x),\sigma(x)] \\ & = -\frac{1}{2}\sigma([x,x])+\epsilon dy+ \frac{1}{2}[\sigma(x),\sigma(x)].\end{align*}  Writing 
%\begin{equation}\label{z formula}\frac{1}{2}\sigma([x,x])-\frac{1}{2}[\sigma(x),\sigma(x)]=\epsilon z\end{equation} for a (uniquely determined) $z\in\g^2$, we see that $x'$ satisfies the Maurer-Cartan equation if and only if 
%\[dy=z.\]  On the other hand, if we define $z$ by equation (\ref{z formula}), a straightforward calculation shows that $dz=0$ and therefore determines a cohomology class $[z]\in H^2(\g)$ (which turns out to depend only on $[x]\in\De_{\g}(A)$.); we can find an $x'$ as above satisfying the Maurer-Cartan equation precisely when this cohomology class vanishes.  In particular, we arrive at the following result. \

\begin{theorem}\label{H2 and obstructions}\cite[\S3]{M2}Let $\g$ be a DGLA with $H^2(\g)=0.$  Then the functor $\df_{\g}:\Art\to\Set$ is unobstructed.
\end{theorem}

\subsection{Bisemicosimplicial DGLAs, totalization, and descent}\label{bisemi}

The notation and terminology in this section (mostly) follows \cite{I}.
Let 
\[V\ut=\xymatrix @C=.3in{V_0 \ar@<+.4ex>[r]\ar@<-.4ex>[r] & V_1 \ar@<+.6ex>[r]\ar@<-.6ex>[r] \ar[r]& V_2\ar@<+1ex>[r]\ar@<-1ex>[r]\ar@<+.3ex>[r]\ar@<-.3ex>[r] & \dots} \]
 be a semicosimplicial cochain complex, i.e. a semicosimplicial object in the category of cochain complexes (differential graded vector spaces).  As in \S\ref{cosimplicial groupoids}, for each $n=0,1\cdots$, we denote by $d_n:V_n\to V_n$ the differential on the cochain complex $V_n$, and we denote by $\partial_n^i:V_n\to V_{n+1}$ for $i=0,\cdots n+1$ the coface maps.    From $V\ut$ we may construct a cochain complex $\Tot(V\ut)$, called the \emph{total complex} or \emph{totalization} of $V\ut$.  As a graded vector space, $\Tot(V\ut)$ is equal to the direct sum $\bigoplus_{n\geq 0} V_n[-n]$, where for any graded vector space $W=\oplus_i W^i$ and any integer $k$, the shifted space $W[k]$ is defined by $(W[k])^i=W^{i+k}$.  For each $n$, let
 \[\partial_n=\sum_{i=0}^{n+1}(-1)^i\partial_n^i:V_n\to V_{n+1}.\]  Define 
 \[d=\sum_{n=0}^{\infty}(-1)^nd_n:\bigoplus_{n\geq0} V_n[-n]\to \bigoplus_{n>0} V_n[-n]\] and 
 \[\partial=\sum_{n=0}^{\infty}\partial_n: \bigoplus_{n\geq0} V_n[-n]\to \bigoplus_{n\geq0} V_n[-n].\]  The differential on $\Tot(V\ut)$ is then defined to be the sum
 \[D=d+\delta.\]  
 
For any category $\CC$, the collection of semicosimplicial objects in $\CC$ themselves form the objects of a category:  a morphism $A\ub\to B\ub$ of semicosimplicial objects in $\CC$ is by definition a collection of morphisms $\{\psi_n:A_n\to B_n\}$ in $\CC$ that commute with the coface maps.  The totalization construction sending a semicosimplicial cochain complex $V\ut$ to its total complex $\Tot(V\ut)$ extends in an natural way to a functor (from the category of semicosimplicial cochain complexes to the category of cochain complexes).

Next, let $\g\ut$ be a semicosimplicial DGLA.  If we apply the above construction to the semicosimplicial cochain complex underlying $\g\ut$, there is no natural way to give the resulting cochain complex $\Tot(\g\ut)$ a DGLA structure.  There is an alternate (functorial) construction, however, that takes as input a semicosimplicial DGLA $\g\ut$, and gives as output a DGLA $\Tot_{TW}(\g\ut)$, known as the \emph{Thom-Whitney totalization of $\g\ut$}.  We will not need the explicit form of this construction, only its existence and the properties summarized in Propositions \ref{cochain iso1}, \ref{descent for Deligne}, \ref{bi chain iso}, and \ref{double descent} below.   
\begin{proposition}\label{cochain iso1}\cite[Lemma 2.5]{I}  Let $\g\ut$ be a semicosimplicial DGLA.  Then the cochain complexes 
$\Tot_{TW}(\g\ut)$ and $\Tot(\g\ut)$ are quasi-isomorphic.
\end{proposition} 

Since $\Ttot(\g\ut)$ is a DGLA, we may apply the Deligne construction described above to construct the formal groupoid $\Del_{\Ttot(\g\ut)}$.  On the other hand, applying the Deligne construction term by term to $\g\ut$ produces a semicosimplicial formal groupoid
\[\G\ut(A)=\xymatrix @C=.3in{\Del_{\g_0} \ar@<+.4ex>[r]\ar@<-.4ex>[r] & \Del_{\g_1} \ar@<+.6ex>[r]\ar@<-.6ex>[r] \ar[r]& \Del_{\g_2}\ar@<+1ex>[r]\ar@<-1ex>[r]\ar@<+.3ex>[r]\ar@<-.3ex>[r] & \dots} \]  We may then form its descent groupoid $\Desc(\G\ut)$. 
\begin{proposition} \label{descent for Deligne}\cite[Thm. 2.6]{BM} There is a natural equivalence of formal groupoids 
\[\Del_{\Ttot(\g\ut)}\toco \Desc(\G\ut).\]
\end{proposition}

Next, let $V\utt$ be a bisemicosimplicial cochain complex, which is a diagram
\begin{equation}\label{square diagram}\xymatrix @C=.3in {  \vdots &  \vdots & \vdots \\ V_{0,1} \ar@<+.6ex>[u]\ar@<-.6ex>[u] \ar[u] \ar@<+.4ex>[r]\ar@<-.4ex>[r] &V_{1,1}  \ar@<+.6ex>[u]\ar@<-.6ex>[u] \ar[u] \ar@<+.6ex>[r]\ar@<-.6ex>[r] \ar[r]& V_{2,1}  \ar@<+.6ex>[u]\ar@<-.6ex>[u] \ar[u]\ar@<+1ex>[r]\ar@<-1ex>[r]\ar@<+.3ex>[r]\ar@<-.3ex>[r]  & \cdots \\ V_{0,0} \ar@<+.4ex>[u]\ar@<-.4ex>[u] \ar@<+.4ex>[r]\ar@<-.4ex>[r] &V_{1,0} \ar@<+.4ex>[u]\ar@<-.4ex>[u]\ar@<+.6ex>[r]\ar@<-.6ex>[r] \ar[r]&  V_{2,0} \ar@<+.4ex>[u]\ar@<-.4ex>[u]\ar@<+1ex>[r]\ar@<-1ex>[r]\ar@<+.3ex>[r]\ar@<-.3ex>[r]  & \cdots}\end{equation} We may succinctly define $V\utt$ as a functor from the category $\Delta_{mon}\times\Delta_{mon}$ to the category of cochain complexes.  Each individual row $V_{\bullet,n}$ in the above diagram itself forms a semicosimplicial cochain complex, and we alternatively view $V\utt$ as a semicomisimplical object in the category of semicosimplicial cochain complexes: 
\[V\utt=\xymatrix @C=.3in{V_{\bullet,0} \ar@<+.4ex>[r]\ar@<-.4ex>[r] & V_{\bullet,1} \ar@<+.6ex>[r]\ar@<-.6ex>[r] \ar[r]& V_{\bullet,2} \ar@<+1ex>[r]\ar@<-1ex>[r]\ar@<+.3ex>[r]\ar@<-.3ex>[r] & \dots} \]  From this point of view, it is clear how to generalize the totalization functor to the bisemicosimplicial case: first, form the total complex of each row, yielding a semicosimplicial cochain complex
\[\Tot^{\triangle}(V\utt):=\xymatrix @C=.3in{\Tot(V_{\bullet,0}) \ar@<+.4ex>[r]\ar@<-.4ex>[r] & \Tot(V_{\bullet},1) \ar@<+.6ex>[r]\ar@<-.6ex>[r] \ar[r]& \Tot(V_{\bullet},2) \ar@<+1ex>[r]\ar@<-1ex>[r]\ar@<+.3ex>[r]\ar@<-.3ex>[r] & \dots} \]  We may then applying the totalization procedure again to form the cochain complex $\Tot(\Tot^{\triangle}(V\utt))$, which we will simply denote by $\Tot(V\utt)$.  Similarly, if $\g\utt$ is a bisemicosimplicial DGLA, we may iterate the Thom-Whitney totalization procedure to construct a DGLA $\Tot_{TW}(\g\utt)$

\begin{proposition}\label{bi chain iso} For any bisemicosimplicial DGLA $\g\utt$, the cochain complexes $\Tot(\g\utt)$ and $\Tot_{TW}(\g\utt))$ are quasi-isomorphic.
\end{proposition}  

One may similarly define a strict  bisemicosimplicial formal groupoid $\G\utt$ to be a diagram of the form \ref{square diagram}, where the entries are groupoids and coface maps are required to satsify the same compatibility conditions (on the nose).  To define the descent groupoid of $\G\utt$, we first form the descent groupoid $\Desc(\G_{\bullet,n})$ of each row, resulting in a semicosimplicial groupoid $\Desc\ut(\G\utt)$.  We may then form its formal groupoid of descent data $\Desc(\Desc\ut(\G\utt))$, which we denote simply by $\Desc(\G\utt)$.
\begin{proposition}\label{double descent}  Let $\g\utt$ be a bisemicosimplicial DGLA, and let $\Del_{\g\utt}$ be the bisemicosimplicial formal groupoid formed by applying the Deligne construction term by term to $\g\utt$.  Then there is a natural equivalence of formal groupoids between $\Del_{\Tot(\g\utt)}$ and $\Desc(\Del_{\g\utt})$.
\end{proposition}

\subsection{Construction of the DGLA}\label{construction of the DGLA}
We now apply the theory described in the previous section to the deformation theory of generalized complex branes.  Specifically, for every brane $\B\in\Br(X,\J)$ and open cover $\U$ of $X$, we will construct a DGLA $L_{\B,\U}$.  Using results from the previous section, we then prove that the formal groupoid $\Del_{L_{\B,\U}}$ is equivalent to the formal $\De\uB(\U,\J)$ introduced in Definition \ref{definition extended nerve}.  For $\B$ a LWL brane, it follows from Theorem \ref{theorem extended and LWL} that $\Del_{L_{\B,\U}}$ is also equivalent to $\De\uB(X,\J)$.  In particular, there is a natural isomorphism of functors 
\[\df_{L_{\B,\U}}\toco \df_{\B}.\]

As a concrete application of this construction we prove Theorem \ref{brane obstructions}, stated in the introduction.  This generalizes the well-known result in complex geometry that the obstructions to deforming a complex submanifold $Z$ of a complex manifold $X$ are contained in the sheaf cohomology group $H^2(Z;\mathcal{O}_{NZ})$.

\subsection{The construction}
Let $\B$ be a GC brane on a GC manifold $(X,\J)$, and let $\U=\{U\sa\}$ be an open cover of $X$.  Consider the following diagram of  Lie algebras 
\begin{equation}\label{V diagram} \xymatrix @C=.15in {\prod \T(U\sa) \ar@<+.4ex>[r]\ar@<-.4ex>[r] & \prod\T(U_{\alpha\beta}) \ar@<+.6ex>[r]\ar@<-.6ex>[r] \ar[r]& \prod\T(U_{\alpha\beta\gamma}) \ar@<+1ex>[r]\ar@<-1ex>[r]\ar@<+.3ex>[r]\ar@<-.3ex>[r]  & \cdots \\  \H(X)\oplus\prod\KKK(U\sa) \ar@<+.4ex>[u]\ar@<-.4ex>[u] \ar@<+.4ex>[r]\ar@<-.4ex>[r] &  \H(X)\oplus\prod\KKK(U\sab) \ar@<+.4ex>[u]\ar@<-.4ex>[u]\ar@<+.6ex>[r]\ar@<-.6ex>[r] \ar[r]&  \H(X)\oplus\prod\KKK(U\sabc) \ar@<+.4ex>[u]\ar@<-.4ex>[u]\ar@<+1ex>[r]\ar@<-1ex>[r]\ar@<+.3ex>[r]\ar@<-.3ex>[r]  & \cdots}\end{equation} The maps in the diagram are defined as follows: The top row is the semicosimplicial DGLA which is the nerve (with respect to the open cover $\U$) of the sheaf $\T$ of DGLAs on $X$, as described in Example \ref{example nerve groupoid}.  The bottom row is the direct sum of the nerve (with respect to $\U$) of the sheaf $\KKK$, and the semicosimplicial DGLA 
\[\xymatrix @C=.3in{\H(X) \ar@<+.4ex>[r]\ar@<-.4ex>[r] & \H(X) \ar@<+.6ex>[r]\ar@<-.6ex>[r] \ar[r]& \H(X)\ar@<+1ex>[r]\ar@<-1ex>[r]\ar@<+.3ex>[r]\ar@<-.3ex>[r] & \dots}\] with all maps equal to the identity. Next we describe the vertical maps, which we denote by $\partial_V^i$ for $i=0,1$.  Given $v=(\xx_f,\{y_{\alpha_0\alpha_1\cdots\alpha_k}\})\in \H(X)\oplus\prod\KKK(U_{\alpha_0\alpha_1\cdots\alpha_k})$, we have 
\[\partial_V^0(v)=\{(\xx_f)|_{U_{\alpha_0\alpha_1\cdots\alpha_k}}\}\] and 
\[\partial_V^1(v)=\{\chi(y_{\alpha_0\alpha_1\cdots\alpha_k})\}.\]  It is easy to see that the diagram (\ref{V diagram}), extended upwards by zero, defines a bisemicosimplicial DGLA, which we denote by $V\utt_{\B,\U}$.
\begin{proposition}\label{concrete description} The formal groupoid $\Desc(\Del_{V_{\B,\U}\utt})$ is equivalent to  $\De\uB(\U,\J)^{tr}.$
\end{proposition}

\begin{proof}
For each $A\in\Art$, the bisemicosimplicial groupoid $\Del_{V_{\B,\U}\utt}(A)$ may be identified with 
\[\xymatrix @C=.1in {{*}//\prod e^{\T\da(U\sa)} \ar@<+.4ex>[r]\ar@<-.4ex>[r] & {*}//\prod e^{\T\da(U_{\alpha\beta})} \ar@<+.6ex>[r]\ar@<-.6ex>[r] \ar[r]& {*}//\prod e^{\T\da(U_{\alpha\beta\gamma})} \ar@<+1ex>[r]\ar@<-1ex>[r]\ar@<+.3ex>[r]\ar@<-.3ex>[r]  & \cdots \\  {*}//e^{\H\da(X)}\times\prod e^{\KKK\da(U\sa)} \ar@<+.4ex>[u]\ar@<-.4ex>[u] \ar@<+.4ex>[r]\ar@<-.4ex>[r] &  {*}//e^{\H\da(X)}\times\prod e^{\KKK\da(U\sab)} \ar@<+.4ex>[u]\ar@<-.4ex>[u]\ar@<+.6ex>[r]\ar@<-.6ex>[r] \ar[r]&  {*}//e^{\H\da(X)}\times\prod e^{\KKK\da(U\sabc)} \ar@<+.4ex>[u]\ar@<-.4ex>[u]\ar@<+1ex>[r]\ar@<-1ex>[r]\ar@<+.3ex>[r]\ar@<-.3ex>[r]  & \cdots}\]

Let us label these groupoids as $\GG_{ij}$, where $i=1,2$ is the row index, and $j=0,1,\cdots$ is the column index.  Consider first the descent groupoid $\CC_0:=\Desc(\GG_{0,\bullet})$ of the bottom row.  Recall from Definition \ref{definition descent groupoid}, that an object of $\CC$ is a pair $(l,m)$ where $l$ is an object of $\GG_{00}$ and $m$ is a morphism in $\GG_{01}$ from $\partial^0l$ to $\partial^1l$, which satisfies the associativity condition 
\begin{equation}\label{assoc1} \partial^0(\partial^1)^{-1}\partial^2m=id.\end{equation} We must have $l=*$ (since this is the only object in $\GG_{00}$), and $m$ must be an element $(z,\{y_{\alpha\beta}\}))\in e^{\H\da(X)}\times e^{\K\da(X)}$; the associativity condition (\ref{assoc1}) is equivalent to 
\[1=(z,y_{\beta\gamma}y_{\alpha\gamma}^{-1}y_{\alpha\gamma}),\] so we see that $z=1$ must be trivial, and the collection $\{y_{\alpha\beta}\}$ must satisfy the nonabelian cocycle condition $y_{\beta\gamma}y_{\alpha\gamma}^{-1}y_{\alpha\gamma}=1$. To simplify the notation, we will denote the object $(*,(1,\{y_{\alpha\beta}\}))$ of $\CC_0$ by $\{y_{\alpha\beta}\}$.  Spelling out part (2) of Definition \ref{definition descent groupoid}, we see that a morphism in $\CC^0$ from $\{y_{\alpha\beta}\}$ to $\{y'_{\alpha\beta}\}$ consists of a pair $(z,\{w_{\alpha}\})$, where $z$ is an arbitrary element of $e^{\H\da(X)}$, and $\{w_{\alpha}\in e^{\KKK\da(U_{\alpha})}\}$ are subject to the condition $y'_{\alpha\beta}w_{\beta}=w_{\alpha}y_{\alpha\beta}$.   

Similarly, an object of $\CC_1=\Desc(\GG_{1,\bullet})$ consists of $\{x_{\alpha\beta}\in e^{\T\da(U_{\alpha\beta})}\}$ satisfying $x_{\beta\gamma}x_{\alpha\gamma}^{-1}x_{\alpha\beta} =1$.  A morphism in $\CC^1$ from $\{x_{\alpha\beta}\}$ to $\{x'_{\alpha\beta}\}$ consists of $\{v_{\alpha}\in e^{\T\da(U_{\alpha})}\}$ satisfying $x'_{\alpha\beta}v_{\beta}=v_{\alpha}x_{\alpha\beta}$.

To complete the construction of $\Desc(\Del_{V_{\B,\U}\utt}(A))$, we then then apply the descent groupoid construction in the vertical direction:
\[\Desc(\Del_{V_{\B,\U}\utt}(A))=\Desc(\xymatrix @C=.3in{\CC_0 \ar@<+.4ex>[r]\ar@<-.4ex>[r] & \CC_1}).\]  Here, $\partial_V^0:\CC_0\to\CC_1$ is the functor taking and object $\{y_{\alpha\beta}\}$ in $\CC_0$ to the trivial object $\{x_{\alpha\beta}=1\}$, and a morphism $(z,\{w_{\alpha}\}):\{y_{\alpha\beta}\} \to \{y'_{\alpha\beta}\}$ to the morphism $\{z|_{U_{\alpha}}\}$ from the trivial object to itself.   We also see that $\delta^1_V:\CC_0\to\CC_1$ takes $\{y_{\alpha\beta}\}$ to $\{\chi(y_{\alpha\beta})\}$, and morphism $\{(z,w_{\alpha}\}:\{y_{\alpha\beta}\}\to\{y'_{\alpha\beta}\}$ to the morphism $\{\chi(w)_{\alpha}\}:\{\chi(y_{\alpha\beta})\}\to \{\chi(y'_{\alpha\beta})\}$.  Therefore, an object of $\Desc(\Del_{V_{\B,\U}\utt}(A))$ is a pair $(\{y_{\alpha\beta}\},\{x_{\alpha}\})$ with $y_{\alpha\beta}\in e^{\KKK\da(U_{\alpha\beta})}$ satisfying $y_{\beta\gamma}y^{-1}_{\alpha\gamma}y_{\alpha\beta}=1$, and $x_{\alpha}\in e^{\T\da(U_{\alpha})}$ satisfying $\chi(y_{\alpha\beta})x_{\beta}=x_{\alpha}$.  A morphism in $\Desc(\Del_{V_{\B,\U}\utt}(A))$ from $(\{y_{\alpha\beta}\},\{x_{\alpha}\})$ to $(\{y'_{\alpha\beta}\},\{x'_{\alpha}\})$ is a pair $(z,\{w_{\alpha}\})$ with $z\in e^{\T\da(X)}$, $w_{\alpha}\in e^{\KKK\da(U_{\alpha})}$ such that $y'_{\alpha\beta}w_{\beta}=w_{\alpha}y_{\alpha\beta}$ and $\chi(w_{\alpha})x_{\alpha}=x'_{\alpha}z$.    Comparing this with the Definition \ref{definition of extended glued} of $\Def^{\B}_A(\U,\J)^{tr}$, we arrive at the desired result.
\end{proof}
Introduce the notation $L_{\B,\U}:=\Tot_{TW}(V_{\B,\U}\utt)$.
Combining Proposition \ref{concrete description} and Proposition \ref{double descent}, we have the following result.
\begin{corollary}\label{DGLA for extended} There is an equivalence of formal groupoids 
\[\De\uB(\U,\J)^{tr}\cong \Del_{L_{\B,\U}}.\]  \end{corollary}

Finally, Corollary \ref{DGLA for extended} with Theorem \ref{theorem extended and LWL} we arrive at the following result.
 \begin{theorem}\label{DGLA for LWL}  Let $\B$ be a leaf-wise Lagrangian GC brane on $(X,\J)$ (or more generally a brane with locally trivializable deformations).  Then there exists an open cover $\U$ of $X$ such that $\De\uB(X,\J)$ is equivalent to $\Del_{L_{\B,\U}}$.  In particular,  the deformation functor $\df_{L_{\B,\U}}$ is isomorphic to the functor $\df_{\B}$ associated to $\B$ as described in Definition \ref{brane deformation functor}.
\end{theorem}

\subsection{Lie algebroid cohomology and obstructions}\label{cohomology}  Let $\B\in\Br(X,\J)$ be a LWL brane (or more generally, a brane with locally trivializable deformations), and let fix an open cover $\U$ of $X$, chosen as in Theorem \ref{DGLA for LWL}.  Let $V\utt:=V\utt_{\B,\U}$ be the bisemicosimplicial DGLA constructed in \S\ref{construction of the DGLA}, and  $L:=L_{\B,\U}$ the DGLA $\Ttot(V_{\B,\U}\utt)$.  We may also form the cochain complex $C:=\Tot(V\utt)$, which by Proposition \ref{bi chain iso} is quasi-isomorphic to (the underlying cochain complex of) $L$; in particular the cohomology groups of $C$ and $L$ are isomorphic.  As state in \S\ref{DGLA obstructions}, the deformation functor $\df_L$ is unobstructed if the cohomology group $H^2(L)\cong H^2(C)$ vanishes.  By Theorem \ref{DGLA for LWL}, the condition $H^2(C)=0$ also implies that the functor $\df_{\B}$ is unobstructed.  The following theorem makes use of this fact by relating the group $H^2(C)$ to the more familiar Lie algebroid cohomology group $H^2(\B)$ associated to $\B$, which were described in  \S\ref{Lie alg cohomology}.
\begin{theorem} There is an injective linear map
\[\Phi:H^2(C)\hookrightarrow H^2(\B).\] 
\end{theorem}
In particular, we obtain Theorem \ref{brane obstructions} as a corollary.

\begin{proof}
Using the explicit construction of the total complex $C=\Tot(V\utt)$  given in \S\ref{bisemi}, we see that \[C^1=\prod_{\alpha}\T(U_{\alpha})\oplus \prod_{\alpha\beta}\KKK(U_{\alpha\beta})\oplus \H(X),\] 
\[C^2=\prod_{\alpha\beta}\T(U_{\alpha\beta})\oplus \prod_{\alpha\beta\gamma}\KKK(U_{\alpha\beta\gamma})\oplus \H(X),\] and 
\[C^3=\prod_{\alpha\beta\gamma}\T(U_{\alpha\beta\gamma})\oplus \prod_{\alpha\beta\gamma\delta}\KKK(U_{\alpha\beta\gamma\delta})\oplus \H(X).\]  The differential $D$ from $C^1$ to $ C^2$ is given by
\begin{equation}\label{D1}(\{x\sa\},\{y\sab\},\{z\})\mapsto (\{x_{\alpha}-x\sb-\chi(y\sab)\}+z|_{U_{\alpha\beta\gamma}},\{y_{\beta\gamma}-y_{\alpha\gamma}+y_{\alpha\beta}\},z),\end{equation}  and from $C^2$ to $C^3$ by 
\begin{equation}\label{D2} (\{x_{\alpha\beta}\},\{y_{\alpha\beta\gamma}\},z)\mapsto (\{-x_{\beta\gamma}+x_{\alpha\gamma}-x_{\alpha\beta}-\chi(y_{\alpha\beta\gamma})+z|_{U_{\alpha\beta\gamma}},\{\delta(y)_{\alpha\beta\gamma\delta}\},0),\end{equation} with 
\[(\delta y)_{\alpha\beta\gamma\delta}=y_{\beta\gamma\delta}-y_{\alpha\gamma\delta}+y_{\alpha\beta\delta}-y_{\alpha\beta\gamma}.\]

We wish to construct a map $\Phi:H^2(C\ub)\to H^2(\B)$.  Let $c=(\{x_{\alpha\beta}\},\{y_{\alpha\beta\gamma}\},z)\in C^2$ be a cocycle representing a class $[c]\in H^2(C\ub)$.   From (\ref{D1}), we see that every element of $C^2$ is cohomologous to one with $z=0$, so without loss of generality we may assume $c$ is of this form.  Write $y_{\alpha\beta\gamma}=(\tilde{y}_{\alpha\beta\gamma},f_{\alpha\beta\gamma})$, with $\tilde{y}_{\alpha\beta\gamma}\in \T\da(U_{\alpha\beta\gamma})$ and $f_{\alpha\beta\gamma}\in \Cinf(Z\cap U_{\alpha\beta\gamma})$.  
We claim that $c\in C^2$ is cohomologous to a cocycle of the form $(\{x'\sab\},\{(\tilde{y}'_{\alpha\beta\gamma},0)\},0)$.  To see this, note using (\ref{D1}) that that the condition $D(c)=0$ implies that 
\[f_{\beta\gamma\delta}-f_{\alpha\gamma\delta}+f_{\alpha\beta\delta}-f_{\alpha\beta\gamma}=0.\] Therefore, using a partition of unity we may find functions $g_{\alpha\beta}\in \Cinf(Z\cap U_{\alpha\beta})$ such that $g_{\beta}-g_{\alpha}=f_{\alpha\beta}$.  Let $\{\tilde{g}\sa\in\Cinf(U\sa)\}$ be a choice of functions extending $\{g_{\alpha}\}$.  Note that $(0,d\tilde{g}\sa)$ is the generalized Hamiltonian vector field associated to the complex-function $i\tilde{g}\sa$, so in particular $(0,d\tilde{g}\sa)\in \T\da(U\sa)$.  Furthermore, by construction $((0,d\tilde{g}\sa),-g\sa)$ is an element of $\KKK(U\sa)$.  Writing 
\[c+D(\{0\},\{((0,d\tilde{g}\sa),-g\sa)\},0)=(\{x'\sab\},\{(\tilde{y}'_{\alpha\beta\gamma},f'_{\alpha\beta\gamma})\},0),\] using (\ref{D1}) we see that $f'_{\alpha\beta\gamma}=0$.  Therefore, without loss of generality we may assume that $c$ is of the form
\[c=(\{x\sab\},\{(\tilde{y}_{\alpha\beta\gamma},0\},0).\]  By Definition \ref{definition KKK}, we see that $\tilde{y}_{\alpha\beta\gamma}\in \KK\uB(U\sab)$, so that 
\[qr(\tilde{y}_{\alpha\beta\gamma})=0.\]  Therefore, if we define $\eta_{\alpha\beta}=qr(x\sab)\in \Cinf(\N|_{Z\cap U\sab})$, then using (\ref{D2}) we see that 
\[\eta_{\beta\gamma}-\eta_{\alpha\gamma}+\eta_{\alpha\beta}=0.\]  Choose $\{\sigma_{\alpha}\in \Cinf(\N|_{Z\cap U\sa})\}$ such that $\sigma_{\beta}-\sigma_{\alpha}=\eta_{\alpha\beta}$.   By Proposition \ref{induced holomorphic}, the fact $x\sab\in \T(\sab)$ implies that $\delta_l\mu(\eta\sab)=0$.  It follows there exists a unique section $\zeta\in\Cinf(\lambda^2\l\ub)$ such that on each open set $Z\cap U\alpha$ we have $\zeta|_{Z\cap U\sa}=\delta_l\mu(\sigma\sa)$. By construction we have $\delta_l\zeta=0$, and we define $\Phi([c])=[\zeta]\in H^2(l)$.  

To show that $\Phi$ is well-defined, we must check that the class $[\zeta]$ does not depend on the choice of $\{\sigma\sa\}$ or the representative $c$ for the class $[c]$.  First, suppose that we have another collection $\{\sigma'\sa\in \Cinf(\N|_{Z\cap U\sa})\}$ that also satisfy $\sigma'_{\beta}-\sigma'_{\alpha}=\eta_{\alpha\beta}$, and that $\zeta'\in\Cinf(\Lambda^2l\uv)$ such that the restriction of $\zeta'$ to each $Z\cap U\sa$ is equal to $\delta_l\mu(\sigma'\sa)$.  Then there is a unique section $\tau\in\Cinf(\N \B)$ such that $\sigma'_{\alpha}-\sigma\sa=\tau|_{Z\cap U\sa}$.  We then have $\zeta'-\zeta=\delta_l\tau$, which implies that $[\zeta']=[\zeta]\in H^2(\B)$.  This verifies that $[\zeta]$ does not depend on the choice of $\{\sigma\sa\}$.  To see that $[\zeta]$ is also independent of the choice of $c$,  suppose that  \[c'=(\{x'\sab\},\{(\tilde{y'}_{\alpha\beta\gamma},0\},0)\] is different cocycle such that $[c']=[c]$ in $H^2(C\ub)$.  Given $b\in C^1$ such that $c'-c=Db$, by a nearly identical argument to that used above, we may assume that $b$ is of the form $b=(\{v\sa\},\{(w_{\alpha\beta},0)\},0$.  We have 
\[x'\sab-x\sab=v_{\alpha}-v_{\beta}-w_{\alpha\beta}.\] Since $w_{\alpha\beta}\in\K(U_{\alpha\beta})$, it follows that 
\begin{equation}\label{diff1}qr(x'\sab)-qr(x\sab)=qr(v_{\alpha})-qr(v_{\beta}).\end{equation}  Given $\{\sigma\sa\in \Cinf(\N\B|_{Z\cap U\sa})\}$ satisfying $\sigma_\beta-\sigma_{\alpha}=x_{\alpha\beta}$, define $\sigma'\sa=\sigma\sa+qr(v\sa)$.  It follows from (\ref{diff1}) that $\sigma'\sb-\sigma'\sa=qr(x'\sab)$.  Furthermore, since $v\sa$ is a generalized holomorphic vector field, it follows that $\delta_l\mu qr(v\sa)=0$, so we see that $\delta_l\mu(\sigma'\sa)=\delta_l\mu(\sigma\sa)$.  This finishes the demonstration that $\Phi$ is well-defined.

To show that $\Phi$ is injective, suppose that $\Phi([c])=[\zeta]=0$ in $H^2(\B)$, say $\zeta=\delta_l\mu(\gamma)$ for some $\gamma\in \Cinf(\N\B)$, where $c=(\{x\sab\},\{(\tilde{y}_{\alpha\beta\gamma},0\},0)$ and $\{\sigma\sa\}$ are as above.  Then on each $U\sa\cap Z$ we have 
\[\delta_l\mu(\sigma\sa-\gamma|_{Z\cap U\sa})=0.\]  It is not hard to show (using, for example, the arguments in \S\ref{LWL branes example} in the special case $A=\R[\epsilon]/(\epsilon^2$)) that there exists $\lambda\sa\in \T(U\sa)$ with $qr(\lambda\sa)=\sigma\sa-\gamma|_{Z\cap U\sa}$.  Therefore we also have $qr(\lambda_{\beta})-qr(\lambda_{\alpha})=qr(x_{\alpha\beta})$, which implies that $x_{\alpha\beta}-(\lambda\sb-\lambda\sa)\in K^{\B}(U\sab)$.  Setting 
\[b=(\{-\lambda\sa\},\{(\lambda_{\beta}-\lambda\sa-x_{\alpha\beta},0)\},0)\in C^1,\] using the formula (\ref{D1}) we see that $D(b)=c$.  We therefore conclude that $[c]=0$ in $H^2(C\ub)$.
\end{proof}

\section{Appendix}
\emph{Proof of Proposition \ref{action}}:

Consider the following general situation: $\g$ is a nilpotent Lie algebra acting on vector spaces $V,W,T$, we have a bilinear map $V\times W\to T$ sending 
\[x,y\mapsto x*y,\] and for each $\xi\in \g$ we have 
\[\xi(x*y)=\xi(x)*y+x*\xi(y).\]  
\begin{lemma}\label{little lemma} With these assumptions, the exponentiated action satisfies
\[e^{\xi}(x*y)=e^{\xi}(x)*e^{\xi}(y).\]
\end{lemma}
\begin{proof}
Note that  
\[\xi^2(x*y)=\xi^2(x)*y+2\xi(x)*\xi(y)+x*\xi^2(y),\] and more generally
\[\xi^k(x*y)=\sum_{j=0}^k{k \choose j}\xi^j(x)*\xi^{k-j}(y).\]  Therefore
\begin{align} e^{\xi}([x*y])&= \sum_{k=0}^{\infty}\frac{1}{k!}\sum_{j=0}^k\frac{k!}{j!(k-j)!}\xi^j(x)*\xi^{k-j}(y) \nonumber \\
&= \sum_{j,l=0}^{\infty} \frac{1}{j!l!}\xi^j(x)*\xi^l(y) \nonumber \\
&= \sum_{j=0}^{\infty}\frac{1}{j!}\xi^j(x)*\sum_{l=0}^{\infty}\frac{1}{l!}\xi^l(y) \nonumber \\
&=e^{\xi}(x)*e^{\xi}(y).
\end{align}
\end{proof}  We note that these assumptions hold in each of the following situations (with $\g=\g\da(X)$):
\begin{enumerate} \item In the case that $V=W=\Cinf\da(X)$, the action of $\g\da(X)$ is by the Lie bracket, and $*$ the Lie bracket.
\item In the case that $V=W=\Omega^{\bullet}\da(X)$, the action of $\g\da(X)$ is by the Lie derivative, and $*$ is the wedge-product.
\item In the case where $V=\Cinf\da(X)$ with $\g\da(X)$ acting by the Lie bracket, $W=\Omega^{\bullet}\da(X)$ with $\g\da(X)$ acting by the Lie derivative, and $*$ is the contraction operation of vector fields with differential forms.
\end{enumerate}

To see that $e^{\xi}$ is compatible with the exterior derivative, note that (using the Cartan formula for the Lie derivative), we have
\[\pounds(\xi)(da)=(\iota(\xi)d+d\iota(\xi))da=d\iota(\xi)da=d\pounds(\xi)a.\]  Iterating, we see that for each $k\geq 0$ we have $\pounds(\xi)^k(da)=d\pounds(\xi)^k(a)$, and therefore $e^{\xi}(da)=de^{\xi}(a).$

\end{document}